\documentclass[onefignum,onetabnum]{siamart190516}

\pdfoutput=1
\usepackage{amsfonts}
\usepackage{graphicx}
\usepackage{epstopdf}
\usepackage{color}
\usepackage{algorithmic}
\usepackage[title]{appendix}

\graphicspath{{figures/}}

\usepackage[labelformat=simple]{subcaption}

\usepackage{xspace}

\ifpdf
  \DeclareGraphicsExtensions{.eps,.pdf,.png,.jpg}
\else
  \DeclareGraphicsExtensions{.eps}
\fi

\newsiamremark{remark}{Remark}
\newsiamremark{hypothesis}{Hypothesis}
\crefname{hypothesis}{Hypothesis}{Hypotheses}
\newsiamthm{claim}{Claim}

\newtheorem{lem}[theorem]{Lemma}
\newtheorem{Prop}{Proposition}

\newcommand{\Real}{\mathbb{R}}
\renewcommand{\v}[1]{\mathbf{#1}}
\renewcommand{\matrix}[1]{\mathbf{#1}}
\newcommand{\x}{\v{x}}
\newcommand{\cp}{\textrm{cp}}
\newcommand{\cpbar}{\overline{\cp}}
\newcommand{\dx}{\Delta x}
\newcommand{\tol}{\texttt{tol}\xspace}
\newcommand{\pdedomain}{\bar{\Omega}}
\newcommand{\eps}{\varepsilon}
\newcommand{\mc}{\mathcal}
\newcommand{\penaltyparam}{\bar{\gamma}}

\title{Simulation and Optimization of Mean First Passage Time Problems in
  2-D using Numerical Embedded Methods and Perturbation Theory}
\author{Sarafa Iyaniwura\thanks{Dept. of Mathematics,
    Univ. of British Columbia, Vancouver, B.C., Canada.} \and
    Tony Wong\footnotemark[1] \and
    Michael J. Ward\footnotemark[1] \and
    Colin~B.~Macdonald\footnotemark[1]\,\,\thanks{corresponding
      author, \texttt{cbm@math.ubc.ca}}
  }

  \headers{Simulation and Optimization of MFPT}{S. Iyaniwura, T. Wong, M. J. Ward, and C. B. Macdonald}

\ifpdf
\hypersetup{
  pdftitle={Simulation and Optimization of MFPT},
  pdfauthor={S. Iyaniwura, T. Wong, M. Ward, and C. B. Macdonald}
}
\fi

\begin{document}

\maketitle

\begin{abstract}
  We develop novel numerical methods and perturbation approaches to
  determine the mean first passage time (MFPT) for a Brownian particle
  to be captured by either small stationary or mobile traps inside a
  bounded 2-D confining domain. Of particular interest is to identify
  optimal arrangements of small absorbing traps that minimize the
  average MFPT.  Although the MFPT, and the associated optimal trap
  arrangement problem, has been well-studied for disk-shaped domains,
  there are very few analytical or numerical results available for
  general star-shaped domains or for thin domains with large aspect
  ratio. Analytical progress is challenging owing to the need to
  determine the Neumann Green's function for the Laplacian, while the
  numerical challenge results from a lack of easy-to-use and fast
  numerical tools for first computing the MFPT and then optimizing
  over a class of trap configurations. In this direction, and for the
  stationary trap problem, we develop a simple embedded numerical
  method, based on the Closest Point Method (CPM), to perform MFPT
  simulations on elliptical and star-shaped domains. For periodic
  mobile trap problems, we develop a robust CPM method to compute the
  average MFPT. Optimal trap arrangements are identified numerically
  through either a refined discrete sampling approach or from a
  particle-swarm optimization procedure. To confirm some of the
  numerical findings, novel perturbation approaches are developed to
  approximate the average MFPT and identify optimal trap
  configurations for a class of near-disk confining domains or for an
  arbitrary thin domain of large aspect ratio.
\end{abstract}



\section{Introduction}

The concept of first passage time has been successfully
  employed in studying problems in several fields of physical and
  biological sciences such as physics, biology, biochemistry, ecology,
  and biophysics, among others (see \cite{grigoriev2002kinetics},
  \cite{holcman2004escape}, \cite{Venu} \cite{schuss2007narrow},
  \cite{ricc1985}, and the references therein). The mean first passage
  time (MFPT) is defined as the average timescale for which a
  stochastic event occurs \cite{van1992stochastic}. Some interesting
  problems formulated as MFPT or narrow escape problems include
  calculating the time it takes for a predator to locate its prey
  \cite{Venu}, the time required for diffusing surface-bound molecules
  to reach a localized signaling region on a cell membrane
  \cite{coombs2009}, and the time needed for proteins searching for
  binding sites on DNA \cite{mirny2009}, among others. In this paper,
  we are interested in the time it take for a Brownian particle to be
  captured by small absorbing traps in a bounded 2-D domain. Narrow escape
  or MFPT problems have been studied extensively both numerically and
  analytically using techniques such as the method of matched
  asymptotic expansions, and there is a growing literature on this
  topic (see \cite{pillay2010asymptotic}, \cite{cheviakov2010asymptotic},
  \cite{kolokolnikov2005optimizing}, \cite{lindsay2017optimization},
  \cite{tzou2015mean}, \cite{redner}, \cite{coombs2009}, and
  \cite{Venu}, and the references therein).

  There are two main classifications of MFPT problems in this context:
  one where the absorbing traps are stationary \cite{coombs2009},
  \cite{Venu}, \cite{cheviakov2010asymptotic}, and the other where the
  traps are mobile \cite{lindsay2017optimization},
  \cite{tzou2015mean}. For the situation with stationary traps, the
  MFPT can be calculated analytically and explicitly for a
  one-dimensional domain, and for a disk-shaped domain with a circular
  trap located at the center of the disk. For domains with multiple
  traps where the trap radius is relatively small compared to the
  length-scale of the domain, the method of matched asymptotic
  expansions can be used to derive an approximation for the MFPT (see
  \cite{cheviakov2010asymptotic}, \cite{kolokolnikov2005optimizing},
  \cite{lindsay2017optimization}, \cite{tzou2015mean},
  \cite{redner}). This method can also be used to approximate the MFPT
  in a regular one- or two-dimensional domain with a moving trap
  \cite{pillay2010asymptotic}, \cite{tzou2015mean},
  \cite{lindsay2017optimization}. However, in the case of an irregular
  domain, computing the MFPT has proven to be challenging both
  analytically and numerically. The main difficulty in solving this
  problem analytically arises from determining the corresponding
  Green's function in the noncircular confining domain, while the
  challenges in the numerical computation arises from implementing the
  appropriate boundary conditions, especially for the case of a moving
  trap, where the location of the trap changes over time. Tackling
  such a problem numerically requires a technique that continuously
  updates the location of the trap, while enforcing the necessary
  boundary conditions at each time-step. Some commercial finite
  element software packages have been employed in studying MFPT
  problems of this form \cite{tzou2015mean}. However, for other
  complicated MFPT problems such as determining the optimal
  configuration of a prescribed number of traps that minimizes the
  average MFPT under a continuous deformation of the boundary of the
  domain, the use of standard software packages is both tedious and
  challenging since the user has little control of the software.

  In this paper, we develop a closest point method (CPM) to numerically
  compute the mean first passage time for a Brownian particle to
  escape a 2-D bounded domain for both stationary and mobile traps.
  CPMs are embedded numerical techniques that use e.g., finite
  differences to discretize partial differential equations (PDEs) and
  interpolation to impose boundary conditions or other geometric
  constraints \cite{ruuth2008, macdonald2011, macdonald2009,
    MacdonaldMerrimanRuuth:ptclouds}.  In addition to computing the MFPT,
  we will explore some interesting optimization experiments that
  focus on minimizing the average capture time of a Brownian particle
  with respect to both the location of small traps in the domain and
  the geometry of irregular 2-D domains.

  More specifically, we will use the CPM to compute the average MFPT
  for a Brownian particle in both an elliptical domain and a class of
  star-shaped domains that contains small stationary traps. One
  primary focus is to use the CPM together with a particle swarm
  optimization procedure \cite{kennedy2010} so as to numerically
  identify trap configurations that minimize the average MFPT in 2-D
  domains of a fixed area whose boundary undergoes a continuous
  deformation starting from the unit disk. In particular, we will show
  numerically that an optimal ring pattern of three traps in the unit
  disk, as established in \cite{kolokolnikov2005optimizing}, deforms
  into a colinear arrangement of traps for a long thin ellipse of the
  same area. For stationary traps, novel perturbation approaches will
  be developed to approximate the optimal average MFPT in near-disk
  domains and for long-thin domains of high aspect ratio. Moreover,
  certain optimal closed trajectories of a moving trap in a circular
  or elliptical domain are identified numerically from our CPM
  approach. In the limit of large rotation frequency analytical
  results for the optimal trajectory of a moving trap are presented to
  confirm our numerical findings.

  In the remainder of this introduction we introduce the relevant
  PDEs for the MFPT and average MFPT in 2-D domains with stationary
  and mobile traps. A brief outline of the paper is given at the end
  of this introductory material.

\subsection{Derivation of the MFPT model}\label{Derivation1}

Consider a Brownian particle on a 1-D interval $[0, L]$ that makes a
discrete jump of size $\Delta x$ within a small time interval
$\Delta t$. Suppose that this particle can exit the interval
only through the end points at $x=0$ and $x=L$.  Let $u(x)$ be the
MFPT for the particle to exit the interval starting from a point
$x \in [0, L]$. Then, $u(x)$ can be written in terms of the
  MFPT at the two neighboring points of $x$ by
  $u(x) = \frac{1}{2} \left[ u(x-\Delta x) + u(x + \Delta x)  \right] +
  \Delta t$.
The absorbing end points imply the boundary conditions
$u(0) = 0$ and $u(L) = 0$: the particle escapes immediately if it
starts at a boundary point.  By Taylor-expanding and taking the limits
$\Delta x \to 0$ and $\Delta t \to 0$ such that
$D = (\Delta x)^2/\Delta t$, the discrete equation reduces to 
the ODE problem 
\begin{equation*}
  D \,u_{xx} = -1\,, \quad  0 < x < L\,; \qquad u(0)=0\,, \quad u(L)=0\,,
\end{equation*}
where $D$ is the diffusion coefficient of the particle.  This
derivation can be readily adopted to a scenario where the ends of the
interval $[0,L]$ are reflecting but the interval contains a stationary
absorbing trap of length $2\varepsilon$, with $\varepsilon > 0$,
centered at the point $x_* \in [0, L]$. In this case, the end points
have no-flux boundary conditions, while zero-Dirichlet boundary
conditions are specified on the boundaries of the trap.  Consequently,
the MFPT $u(x)$ for the Brownian particle satisfies
\begin{equation*}
\begin{gathered}
  D \,u_{xx} = -1\,, \quad
  x \in (0, x_{*} - \varepsilon) \, \cup\, (x_{*} + \varepsilon, L)\,;\\
  u_x(0)= u_x(L)=0\,; \quad
  u(x_{*} - \varepsilon)=u(x_{*} + \varepsilon)=0\,.
\end{gathered}
\end{equation*}

Next, we derive the MFPT problem for a moving trap.  This derivation
is slightly different from that of a stationary trap because it
requires tracking the location of the moving trap at each time-step.
We start by considering a particle performing a 1-D random
walk on the interval $[0,L]$, which contains a small mobile absorbing
trap that moves in a periodic path contained in the interval.
Similar to above, the discrete equation for the MFPT $u(x,t)$ satisfies
\begin{equation*}
  u(x,t) = \frac{1}{2} \left[ u(x-\Delta x, t + \Delta t) +
    u(x + \Delta x, t + \Delta t)  \right] + \Delta t\,.
\end{equation*}
	Upon Taylor expanding in $\Delta x$ and $\Delta t$, and
  taking the limits $\Delta x \to 0$ and $\Delta t \to 0$, such that
  $D = (\Delta x)^2/(2\Delta t)$, the resulting PDE for
the MFPT $u(x,t)$ is
\begin{equation*}
\begin{gathered}
  u_t + D  u_{xx}  + 1 = 0, \quad
  x \in (0, x_{*}(t) - \varepsilon) \, \cup\, (x_{*}(t) + \varepsilon, L),
  \;\;
  0 < t< T,\\
  u(x, 0) = u(x, T),  \,\,
  u(x_{*}(t)-\varepsilon, t) = 0,  \,\,
  u(x_{*}(t)+\varepsilon, t) = 0,  \,\,
  u_x(0, t) = u_x(L, t) = 0,
\end{gathered}
\end{equation*}
where $T$ is the period of oscillation of the trap.  Due to the
oscillations of the trap, we have imposed the time-periodic boundary
condition $u\left(x,0 \right) = u\left(x, T \right)$, which specifies
that the MFPT at each point in the domain should be the same after
each period.  The conditions $u(x_{*}(t) - \varepsilon,t)=0$ and
$u(x_{*}(t) + \varepsilon,t)=0$ imply that the particle is captured by
the edges of the moving trap.  Finally, we impose the no-flux
conditions $u_x(0,t) = u_x(L,t) = 0$ to ensure that the outer
boundaries are reflecting.

\subsection{MFPT problems in 2-D}

For an arbitrary bounded domain $\Omega \subset \mathbb{R}^2$,
containing $m$ small stationary absorbing traps
$\Omega_1, \dots, \Omega_m$ (such as shown in
Figure~\ref{particle_disk_examples:a} for $m=1$), the 
MFPT $u(\x)$ for a Brownian particle starting at a point
$\x \in \pdedomain$ is
\begin{equation}\label{MFPT_DiskSation}
\begin{split}
D\, \nabla^2 u  = -1 \,, &\quad \x \in  \pdedomain\,;\\
\partial_n u = 0 \,, \,\,\, \x \in \partial \Omega\,; &
\qquad u =0\,, \,\,\, \x \in \partial \Omega_i\,, \qquad i = 1, \dots, m\,,
\end{split}
\end{equation} 
where $\x \equiv (x,y)$, $D$ is the diffusion coefficient of the
particle, $\partial_n$ denotes the outward normal derivative on the
domain boundary $\partial \Omega$, and
$\pdedomain = \Omega \setminus \cup_{i=1}^{m} \Omega_i$.

  \begin{figure}
  \centering
  \makebox{
    \scriptsize{(a)}%
    \includegraphics[height=22ex]{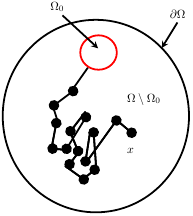}
    \phantomsubcaption
    \label{particle_disk_examples:a}
  }
  \qquad
  \qquad
  \makebox{
    \scriptsize{(b)}%
    \includegraphics[height=21ex]{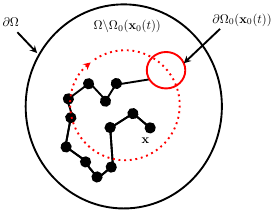}
    \phantomsubcaption
    \label{particle_disk_examples:b}
  }
  \vspace*{-1ex}
    \caption{Brownian particles in disk-shaped regions with absorbing
      traps.  In (a), a particle starting at
      $\x \in \Omega \setminus \Omega_0$ in $\Omega$ eventually hits a
      stationary absorbing trap $\Omega_0$.  In (b), the trap
      $\Omega_0(\x_0(t))$ rotates about the center of the region.}
  \label{particle_disk_examples}
\end{figure}

If the traps are moving in periodic paths with positions $\x_i(t)$
(see Figure~\ref{particle_disk_examples:b}),
then the corresponding MFPT problem is
\begin{equation}\label{MFPT_timedep}
\begin{split}
  &u_t + D \nabla^2  u + 1 = 0\,, \quad \x \in \pdedomain(t)\,; \\
  &\partial_n  u = 0\,,  \,\, \x  \in \partial \Omega\,;  \qquad
  u = 0\,, \,\, \x \in \partial \Omega_i (t)\,;  \qquad
  u(\x, 0) = u(\x, T)\,,
\end{split}
\end{equation}
where $T$ is the period of the moving traps.  Often it will be useful
to write the periodic motion in terms of an angular frequency
$\omega$, where $T = {2\pi/\omega}$.

\subsubsection{Time reversal}
Our numerical calculations will work significantly better if we solve
problem \eqref{MFPT_timedep} ``backwards'' in time, e.g., after the
change of variables $\tau = -t$.
The problem is still periodic in $\tau$ with periodic $T$, namely
\begin{equation}\label{MFPT_timedep_rev}
\begin{split}
  &u_{\tau} = D \nabla^2  u + 1\,, \quad \x \in \Omega
  \setminus\pdedomain(\tau)\,; \\
  &\partial_n  u = 0\,,  \,\, \x  \in \partial \Omega\,;  \qquad
  u = 0\,, \,\, \x \in \partial \Omega_i (\tau)\,;  \qquad
  u(\x, 0) = u(\x, T)\,.
\end{split}
\end{equation}

\subsection{An elliptic problem}

Suppose that the domain $\Omega \subset \mathbb{R}^2$ is a disk
containing a single moving trap centered at $\x_0(t)$ that rotates
about the center of the disk on a ring in the clockwise direction,
such as illustrated in Figure~\ref{particle_disk_examples:b}.  In this
case, using the change of variables
$(x,y)=( r\cos \theta, r\sin \theta)$, with $0 < r \leq 1$, and
$0 \leq \theta \leq 2\pi$, \eqref{MFPT_timedep} can be written in
polar coordinates, with the trap center given by
$\x_0(t) = ( r_0 \cos (\omega t), r_0 \sin (\omega t))$, where $r_0$
is the distance from the center of the trap to the center of the disk.
Furthermore, setting $\phi = \theta - \textrm{mod}(\omega t, 2\pi)$
with $0 < \phi < 2 \pi$, and $u(r,\theta, t) = u(r, \phi(t))$, the
MFPT problem reduces to the elliptic PDE problem
\begin{equation}\label{MFPT_EllipticalMFPT}
\begin{split}
  D\, \nabla^2 u + \omega u_{\phi} +  1=0\,, & \quad
  \mathbf{x} \in \Omega \setminus \Omega_0(r_0)\,;\\
  u =0\,, \,\,\, \mathbf{x} \in \partial \Omega_0 (r_0)\,; \quad &
  \partial_n u = 0 \,, \,\,\, \mathbf{x} \in \partial \Omega\,.
\end{split}
\end{equation}
Here $\nabla^2 u $ is the Laplacian in polar coordinates, and
$u_{\phi}$ is the derivative of $u$ in the transformed angular
coordinate (see \cite{lindsay2017optimization}, \cite{tzou2015mean}
for more details). This reformulation enables us to study an elliptic
PDE, as compared to a more challenging time-dependent parabolic
problem.  However, \eqref{MFPT_EllipticalMFPT} can only be employed in
studying MFPT problems in a domain that is invariant with respect to
the location of the moving trap.

\subsection{Feature extraction}
\label{sec:features}
The MFPT depends on the starting location $\x$ of the
particle. Assuming a uniform distribution of starting locations, the
\textit{average/expected} MFPT for a particle to exit the region
starting from anywhere in the domain is
\begin{equation}\label{AveMFPT}
  \begin{split}
    \overline{u} = \frac{1}{|\pdedomain |} \int_{\pdedomain } u(\x) \,
    \text{d}\x\,,
    \qquad \mbox{where } \quad |\pdedomain| = |\Omega \backslash \cup_{i=1}^{m}
    \Omega_i | \,,
\end{split}
\end{equation}
and $|\pdedomain|$ denotes the area of $\pdedomain$. For the case of a
moving trap, the average MFPT is
\begin{equation}\label{AveMFPT_MovTrap}
\begin{split}
  \overline{u} = \frac{1}{T\,|\pdedomain |} \int_{0}^{T} \int_{\pdedomain }
  u(\x,t) \,\,\, \text{d}\x \, \text{d}t\,.
\end{split}
\end{equation} 
The time integral averages the MFPT over a period, which ensures that
the escape time of the particle is independent of the location of the
trap.  These average MFPT quantities will be used below in our
computation and optimization experiments.

In \S~\ref{CPM_section}, we discuss numerical techniques to compute
solutions to our MFPT problems. In \S~\ref{StationaryTrap_Section} and
\S~\ref{MovingTrap_Section}, we give numerical results for some MFPT
problems with stationary traps and a moving trap,
respectively. Moreover, some numerical optimization experiments are
performed to identify trap configurations that minimize the average
MFPT for a Brownian particle. In \S~\ref{sec:analysis}, asymptotic
results for the MFPT, based on various perturbation schemes, are used
to confirm some of our numerical results in
\S~\ref{StationaryTrap_Section} and \S~\ref{MovingTrap_Section}.  A
brief discussion in \S~\ref{sec:discuss} concludes the paper.

\section{The numerical algorithm}\label{CPM_section}

Closest Point Methods (CPMs) are numerical techniques for
solving PDEs on curved surfaces and other irregularly
shaped domains.  The key idea is to embed the physical domain of
interest into an unfitted numerical grid enveloping the
surface.  All grid points that lie on the interior of
the domain are simply physical solution values, while those that lie
outside the domain are used to impose boundary conditions.  In this
paper, we apply the closest point method to mean first passage time
problem in 2-D domains.  Solving MFPT problems numerically in 2-D
domains using regular finite difference methods comes with certain
difficulties.  Most notably, implementing boundary conditions on
curved boundaries is complicated because grid points do not lie on
those curves.  Fitted grids (such as triangulations) can approximate
curved boundaries but require frequent remeshing in moving boundary
problems.  Embedded methods avoid these remeshing steps.

\subsection{Closest points}
Every grid point is associated with its closest point (by Euclidean
distance) in the physical domain
  $\cp(\x) := \textrm{argmin}_{\v{y} \in \pdedomain} \|\x - \v{y}\|_2$,
where we recall that the domain of our PDE is
$\pdedomain = \Omega \setminus \cup_{i=1}^{m} \Omega_i$.  Note if $\x$
is an interior point, its closest point is simply itself:
$\cp(\x) = \x$.  The closest point function can be precomputed in
closed form for simple shapes, for example, for a disc of radius $R$
punctured by a small $\varepsilon$-radius hole, such a function could
be given by
\begin{align*}
  \cp_{\textrm{punc.disc}}(\x) =
  \begin{cases}
    (\varepsilon, 0) &  \text{if $\x = (0,0)$,} \\
    \varepsilon\frac{\x}{\|\x\|}  &  \text{if $\|\x\| < \varepsilon$,} \\
    \x  &  \text{if $\varepsilon \le \|\x\| \le R$,} \\
    R\frac{\x}{\|\x\|}  &  \text{otherwise (i.e., $\|\x\| > R$).}
  \end{cases}
\end{align*}
We assume that we have either approximate or exact samples
of the closest-point function available for our method; this is our
preferred \emph{representation} of the geometry.

The $\cp$ function can be used to
\emph{extend} functions defined in the domain out into the ambient
space surrounding the domain.
The simplest such \emph{extension} is
\begin{align}\label{cpext}
  v(\x) := u(\cp(\x)),
\end{align}
which defines a function $v : B(\pdedomain) \to \Real$ which agrees
with $u : \pdedomain \to \Real$ for points $\x \in \pdedomain$
and is constant in the normal direction outside of the domain $\pdedomain$.
Here
$B(\pdedomain) \supset \pdedomain$, for example all of $\Real^2$ or a
padded bounding box of $\pdedomain$.
In practice, we only need $B(\pdedomain)$ to be only a few grid points
larger than $\pdedomain$ itself.

\subsection{Imposing boundary conditions using extensions}

Suppose we want to impose a homogeneous Neumann boundary condition
$\partial_n u=0$ at all points along some curve $\gamma$ making up all
or part of the boundary of $\pdedomain$.  Given
$u : \pdedomain \to \Real$, we construct $v(\x) := u(\cp(\x))$ to
obtain a function $v$ which is constant in the normal direction, and
thus satisfies the homogeneous Neumann boundary condition.  A spatial
differential operator applied to $v$ will then respect the
zero-Neumann condition automatically.

For a more general Neumann boundary condition,
$\partial_n u = g_1(\x)$ for $\x \in \gamma$, we (formally) perform a
finite difference in the normal direction to obtain
  $\frac{u(\x) - u(\cp(\x))}{\|\x - \cp(\x)\|_2}
  \approx u_n(\cp(\x)) = g_1(\cp(\x)).$
Rearranging to solve for $u(\x)$ we define the extension:
\begin{align*}
  v(\x) := u(\cp(\x)) + \|\x - \cp(\x)\|_2 \, g_1(\cp(\x)),
\end{align*}
Note as $\x \to \cp(\x)$, we have $v(\x) \to u(\cp(\x))$ so $u$ is
continuous at the boundary.  However, the extended solution is not
very smooth (it may have a corner at $\gamma$) and this leads to a
loss of numerical accuracy \cite{macdonald2011}.  Indeed
the above formula was constructed using first-order
finite differences; we can improve the formal order of accuracy to
at least two by using a centered difference \cite{macdonald2011}.

\subsubsection{Second-order accurate boundary extensions: Neumann}
\label{subsubsection:2nd_order_boundary_extension}

We construct a ``mirror point''
$\cpbar(\x) := \x + 2(\cp(\x) - \x) = 2\cp(\x) - \x$ which consists of
a point reflected across the boundary $\gamma$ \cite{macdonald2011}.
As above, we then apply centered differences around the point
$\cp(\x)$ and solve for $u(\x)$, in order to define
\begin{align*}
  v(\x) := u(\cpbar(\x)) + \|\x - \cpbar(\x)\|_2 \, g_1(\cp(\x)).
\end{align*}
Again we see continuity as $\x \to \cp(\x)$ but now we can expect the
extension to be smoother because instead of just $u(\cp(\x))$ we now
have information about \emph{how} $u(\cpbar(\x)) \to u(\cp(\x))$ is
included.

\subsubsection{Dirichlet boundary extensions}

The general Dirichlet boundary condition, that $u(\x) = g_2(\x)$ for
some specified function $g_2$, can be implemented by copying
the value of $g_2$ for points outside the domain using
  $v(\x) := g_2(\cp(\x))$,
but as before this is a low-accuracy approximation due to lack of
smoothness.  A more accurate extension comes from specifying that the
average value matches the given data
$\frac{1}{2}\big(v(\x) + u(\cpbar(\x)\big) = g_2(\cp(\x))$ from which
we define
\begin{align*}
  v(\x) := 2 g_2(\cp(\x)) - u(\cpbar(\x)),
\end{align*}
which differs from the Neumann case primarily by a change of sign in
front of $u(\cpbar(\x))$.

\subsubsection{Combinations of boundary conditions}
Combining these various extensions
we define an operator $E$ which extends solutions by
\begin{subequations}  \label{eq:extension_bcs}
\begin{align}
  v := E u + g \,,
\end{align}
where operator $E$ and functional $g$ are the homogeneous and
non-homogeneous parts of the extensions respectively:
\begin{align}
  v(\x) :=
  \begin{cases}
    u(\x)          &  \x \in \pdedomain \\
    u(\cpbar(\x)  &  \cp(\x) \in \gamma_{\textrm{n}} \\
    -u(\cpbar(\x)   &  \cp(\x) \in \gamma_{\textrm{d}}
  \end{cases}
  +
  \begin{cases}
    0  &  \x \in \pdedomain, \\
    \|\x - \cpbar(\x)\|_2 \, g_1(\cp(\x))  &  \cp(\x) \in \gamma_{\textrm{n}}, \\
    2 g_2(\cp(\x)) &  \cp(\x) \in \gamma_{\textrm{d}},
  \end{cases}
\end{align}
\end{subequations}
where $\gamma_{\textrm{n}}$ and $\gamma_{\textrm{d}}$ indicate boundaries
with Neumann and Dirichlet conditions respectively.
Although not needed here, all of the above constructions can also be
applied on curved surfaces embedded in $\Real^3$ or higher and of
arbitrary codimension~\cite{macdonald2011}.

\subsubsection{Discretizations of extensions}

Although some of the above extensions were motivated by finite
differences, they are \emph{not} discrete because $\cp(\x)$ and
$\cpbar(\x)$ are not generally grid points (due to the curved boundary
$\gamma$).
One way to discretize is to use a polynomial interpolation scheme to
approximate $u(\cpbar(\x))$ using a \emph{stencil} of grid points
neighboring $\cpbar(x)$.  The typical choice is a $4 \times 4$ grid
which allows bicubic interpolation \cite{ruuth2008}.  Equivalently, we
can use the sample values of $u$ at those same 16 points to build a
bicubic polynomial which approximates $u$; we then compute the exact
extension of that polynomial.

Some of these stencils will contain points outside of $\pdedomain$.
This is not a problem because all functions will be defined over
$B(\pdedomain)$.  That is, we do not really have $u$ and $v$, only
$v : B(\pdedomain) \to \Real$.  What is crucial however is that all
discrete stencils lie inside $B(\pdedomain)$; this is how we define
the computational domain: the set of all grid points $\x$ such that
the stencil around $\cpbar(\x)$ is contained in the
set~\cite{macdonald2009}.

\subsection{Imposing boundary conditions with a penalty}

We wish to spatially discretize the PDE \eqref{MFPT_timedep_rev} using
finite differences and standard time-stepping schemes.  A systematic
procedure is needed to ensure that $v$ remains an appropriate
extension so that such a computation respects the boundary conditions.
The approach of \cite{vonGlehnMarzMacdonald:cpmol} modifies the
problem by introducing a penalty for change that does not satisfy the
extension.  Ignoring the time-periodic condition $u(\x, 0) = u(\x, T)$
for the moment, the idea is that we want to solve
\begin{subequations}   \label{eqn_and_const}
\begin{align}  \label{eqn_and_consta}
  v_t &= D \nabla^2  v + 1, \qquad \x \in \pdedomain,
\end{align}
subject to the constraint that
\begin{align}  \label{eqn_and_constb}
  v &= E v + g, \qquad \x \in B(\pdedomain), \,\,
      \mbox{and for all relevant $t$.}
\end{align}
\end{subequations}
This system can be achieved by extending the right-hand side of
\eqref{eqn_and_consta}, introducing a parameter $\penaltyparam$, and
combining the two equations \cite{vonGlehnMarzMacdonald:cpmol} to give
\begin{align}  \label{penalty_eqn}
  v_t &= \bar{E} D \nabla^2  v + 1
        -\penaltyparam (v - E v - g), \quad \x \in B(\pdedomain), \,\,\,
        \text{and for all relevant $t$,}
\end{align}
where $\bar{E}$ is the closest point extension \eqref{cpext}.

\subsubsection{Method of lines discretization}

The extension operators can be discretized into matrices by collecting
the coefficients of the polynomial interpolant, e.g., using
Barycentric Lagrange Interpolation \cite{macdonald2009}.  This allows
us to write \eqref{eq:extension_bcs} as
\begin{align*}
  \v{v} := \matrix{E}_h \v{u} + \v{g}\,,
\end{align*}
where $\v{v}$ is a long vector of the pointwise samples of the
function $v$ at the grid points in the computational domain.  We use a
uniform grid of $B(\pdedomain)$ with grid spacing $h = \Delta x$.  The
Laplacian operator is replaced by a square matrix $\matrix{L}_h$ where
each row consists of \newcommand{\hack}{\rule{0pt}{1.3ex}}
$\left\{\tfrac{1}{h^2\hack}, \tfrac{1}{h^2\hack},
  \tfrac{-4}{h^2\hack}, \tfrac{1}{h^2\hack},
  \tfrac{1}{h^2\hack}\right\}$ and many zeros.  Combining these
spatial operators, we then discretize \eqref{penalty_eqn} using the
method of lines to obtain an ODE system
\begin{align} \label{mol_penalty_eqn}
  \v{v}_t &= \bar{\matrix{E}}_h D
  \matrix{L}_h \v{v} + \mathbf{1} -\frac{4D}{h^2} \Big( \v{v} -
  \matrix{E}_h \v{v} - \v{g} \Big)\,, \qquad \text{for all relevant
    $t$\,,}
\end{align}
where we have used $\penaltyparam = \frac{2 \text{dim}}{h^2}D$ as recommended
by \cite{vonGlehnMarzMacdonald:cpmol}.  We can then apply forward
Euler, backward Euler or some other time-stepping scheme to
\eqref{mol_penalty_eqn} using discrete time-step size of $\Delta t$.
For example, backward Euler would be
\begin{align}  \label{backward_euler}
  \frac{\v{v}^{n+1} - \v{v}^n}{\Delta t}
  &=
  \left[
    D \bar{\matrix{E}}_h \matrix{L}_h
    - \frac{4D}{h^2} \Big( \matrix{I} - \matrix{E}_h \Big)
  \right] \v{v}^{n+1}
  + \frac{4D}{h^2}  \v{g} + \mathbf{1}\,,
\end{align}
where $\v{v}^{n}$ is a vector of the approximate solution at each grid
point at time $t = n \Delta t$.

\subsubsection{Elliptic solves}

The elliptic problem \eqref{MFPT_EllipticalMFPT} can be discretized in
a similar way \cite{ChenMacdonald:ellipticCPM} using the penalty
approach.  We obtain the discretization
\begin{subequations} \label{elliptic_penalty_eqn}
\begin{equation}
   D \bar{\matrix{E}}_h \matrix{L}_h \v{v}
  -\frac{4D}{h^2} \Big( \v{v} - \matrix{E}_h \v{v} - \v{g} \Big)
  + \Big(
     \matrix{S}_1 \matrix{D}_h^x \v{v} +
     \matrix{S}_2 \matrix{D}_h^y \v{v}
  \Big)
  + \mathbf{1} = \v{0}\,,
\end{equation}
where $\matrix{D}_h^x$ and $\matrix{D}_h^y$ are centered differences
using weights $\left\{-\tfrac{1}{2h}, 0, \tfrac{1}{2h}\right\}$, and
$\matrix{S}_1$ and $\matrix{S}_2$ are diagonal matrices with the local
advection vector coefficients $s_1(x,y)$ and $s_2(x,y)$, extended by
\eqref{cpext}, on the diagonal.
For our specific problem~\eqref{MFPT_EllipticalMFPT}, we have
\begin{equation}
  s_1(x,y) = \omega r \cos \theta,  \,\,
  s_2(x,y) = -\omega r \sin \theta, \,\, \mbox{where} \,\,
  r^2 = x^2 + y^2, \,
  \theta = \tan^{-1}\!\left(\tfrac{y}{x}\right).
\end{equation}
\end{subequations}
If $\omega$ is large, upwinding differences should be used for the
advection.

\subsection{Relaxation to a time-periodic solution}\label{sec:relax}

In our moving trap problem \eqref{MFPT_timedep_rev}, the traps
$\Omega_i(\x_i(t))$ are moving, and thus the domain $\pdedomain$ is
changing over time.  This means the discretization operators
$\matrix{E}_h$ and $\bar{\matrix{E}_h}$ are changing at each time step.
At least in principle
the grid itself could also change although for simplicity of
implementation we include all grid points in the interior of the small
traps (even if not strictly needed).  We assume that the
traps do not move too far per timestep---not more than one or two
grid points---to avoid large discretization errors.

In our moving domain problems, the period $T = {2 \pi/\omega}$ of the
motion is known and we look for solutions which satisfy the
time-periodic boundary condition $u(\x,0) = u(\x,T)$.  An
``all-at-once'' discretization of both space and time simultaneously
could be prohibitive in terms of memory usage.
Instead, we approach this problem using a ``shooting method'':
we solve an initial value problem from a somewhat arbitrary initial
guess at $t=0$ \emph{for many periods}.  Due to the dissipative nature
of the PDE, we expect this procedure to converge to a time-periodic
solution.

\subsubsection{Stopping criterion}\label{sec:stopping_criterion}

At the end of the $N$th period we compare the numerical solution at
$t = NT$ with that from $t = (N-1)T$.  We define a tolerance
\tol and stop the calculation when
  $\|\v{v}(NT) - \v{v}((N-1)T)\| \le \tol$,
in some norm; typically we use the change in the average MFPT as our
stopping criterion.

\subsection{Feature extraction}
\label{sec:viz}

Visualizing the solution can be accomplished by
coloring all grid points according to the numerical solution value,
with grid points outside the physical domain simply omitted.
We also need to extract features of the solution, such as the maximum
value, or the average over space and time from \S~\ref{sec:features}.
Spatial integrals of the solution can be extracted using quadrature
although care must be taken near the edges of the domain to ensure
second-order accuracy.  We use a modified quadrature weight
\cite{engquist2005discretization} to integrate the numerical solution over
a non-rectangular domain.  Temporal integration is done using
Trapezoidal Rule.

\section{Numerical computations for stationary trap problems}\label{StationaryTrap_Section}

In this section, the CPM is used to compute solutions for
some MFPT problems in 2-D domains with stationary traps. Moreover,
some stationary trap configurations that optimize the average MFPT are
identified numerically.

\subsection{MFPT for a concentric stationary trap in a disk}
We use the CPM to compute the MFPT for a Brownian particle
in the unit disk with a concentric stationary trap of radius
$\varepsilon = 0.05$.  The result is shown in
Figure~\ref{fig:convg_static_trap_with_epsilon:a}.  Based on the
figure colormap we observe the intuitive result that the MFPT is
smaller for particles that start closer to the trap than for those
that start farther away.

\subsection{Convergence Study}

We use the exact solution
$u(r) = \frac{1}{4}(\varepsilon^2 - r^2) + \frac{1}{2}\log(r/\varepsilon)$
for the MFPT to perform a convergence study of our
numerical method.
For several values of the trap radius $\varepsilon$, and various grid
spacings $\dx$, we numerically compute the MFPT.
The resulting $L_\infty$ error is shown in
Figure~\ref{fig:convg_static_trap_with_epsilon}.
As $\varepsilon$ decreases, the exact solution has a stronger gradient
owing to the logarithmic term.
This leads to a poorer convergence of the numerical solution.
Nevertheless, we observe second-order convergence of
the numerical solution as $\dx \to 0$, as expected from
\S~\ref{subsubsection:2nd_order_boundary_extension}.

\begin{figure}[tbp]
  \centering
  \makebox{%
    \raisebox{0.5ex}{\small{(a)}}
    \includegraphics[width=0.37\textwidth]{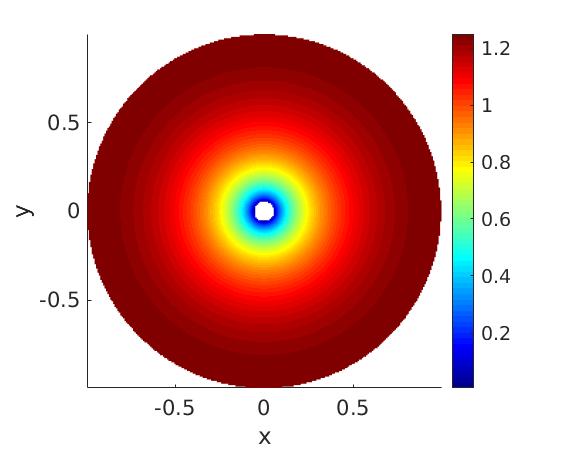}
    \phantomsubcaption
    \label{fig:convg_static_trap_with_epsilon:a}
  } \qquad \makebox{%
    \raisebox{0.5ex}{\small{(b)}}
    \includegraphics[width=0.36\textwidth]{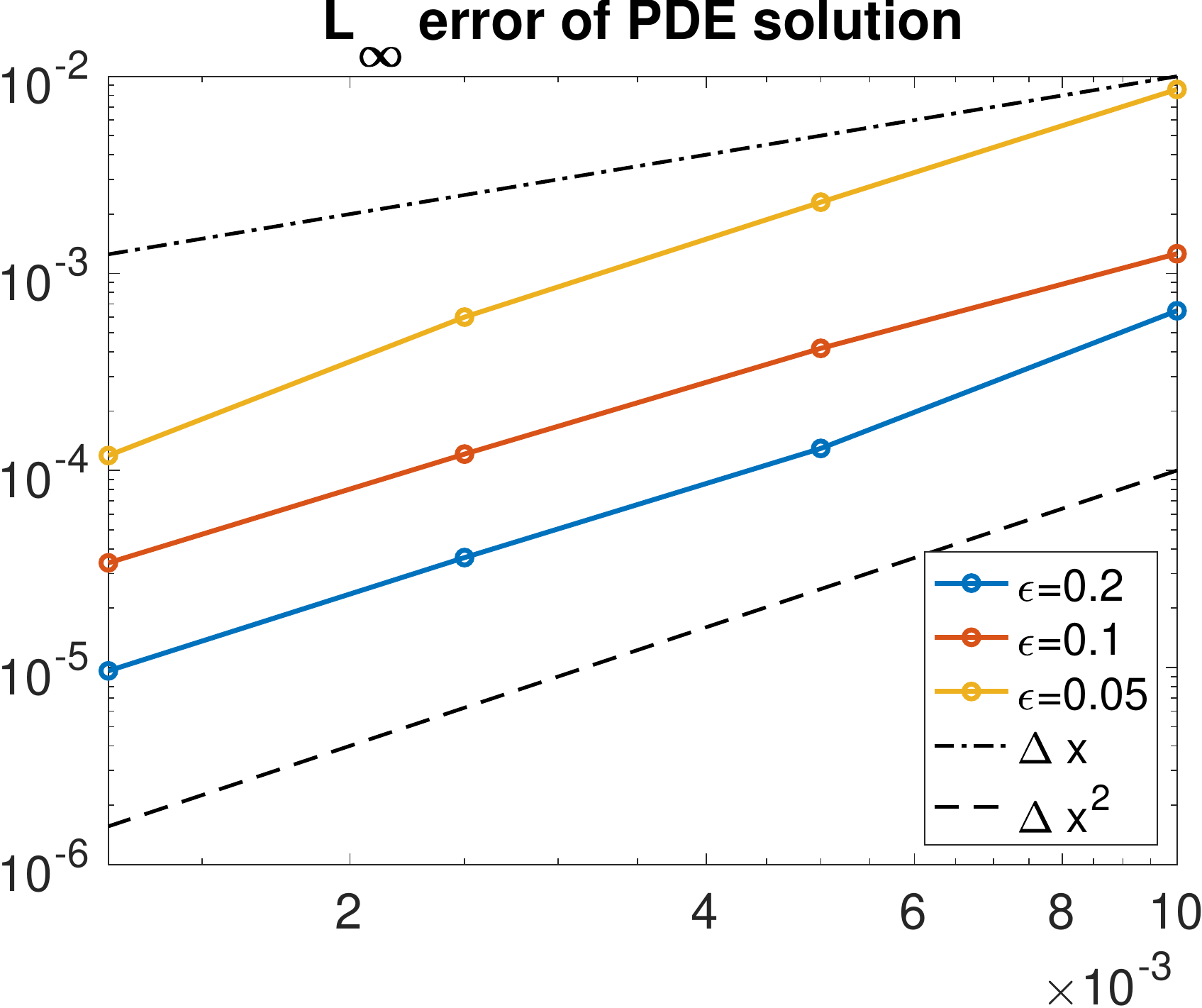}
    \phantomsubcaption
    \label{fig:convg_static_trap_with_epsilon:b}
  }
  \vspace*{-1ex}
  \caption{Convergence studies on the punctured unit disk
      for various values of the trap radius $\varepsilon$, confirming
      second-order convergence of our elliptic solver. (a) MFPT, with
      colormap indicating the time for capture starting at $\x$. (b)
      $L_\infty$ error versus $\Delta x$.
    }
  \label{fig:convg_static_trap_with_epsilon}
\end{figure}

Next, we study the accuracy of the numerical quadrature
$I_h = \sum_{i,j} \omega_{i,j} u^h_{i,j}$ of the numerical solution
$u^h$ on rectangular grid.  The trivial weight $\omega_{i,j} = 1$ is
only first order accurate.  We compare it with second-order accurate
modified weight~\cite{engquist2005discretization} by computing the
area of the perforated domains shown in
Figure~\ref{fig:convg_integration}.  The convergence study in
Figure~\ref{fig:convg_integration:c}, shows that the convergence rate
using the trivial weight is only first order, with an error
significantly larger than the mesh size $\dx$.  However, by using the
modified weight for numerical integration, we observe a second-order
convergence rate in both examples.

\begin{figure}[tbp]
  \centerline{%
    \scriptsize{(a)}%
    \includegraphics[width=0.25\textwidth]{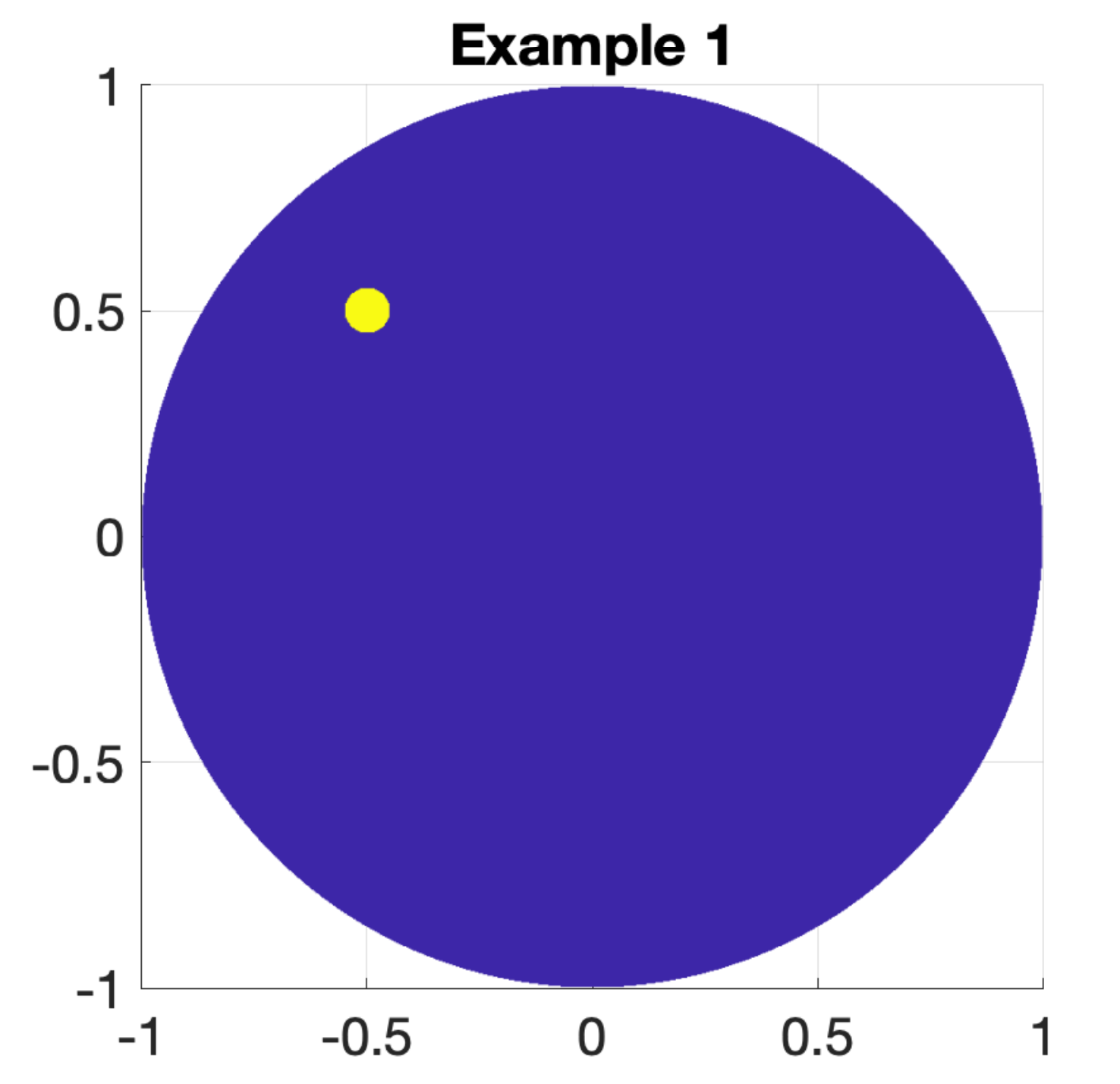}
    \phantomsubcaption
    \label{fig:convg_integration:a}
    \hfill
    \scriptsize{(b)}%
    \includegraphics[width=0.25\textwidth]{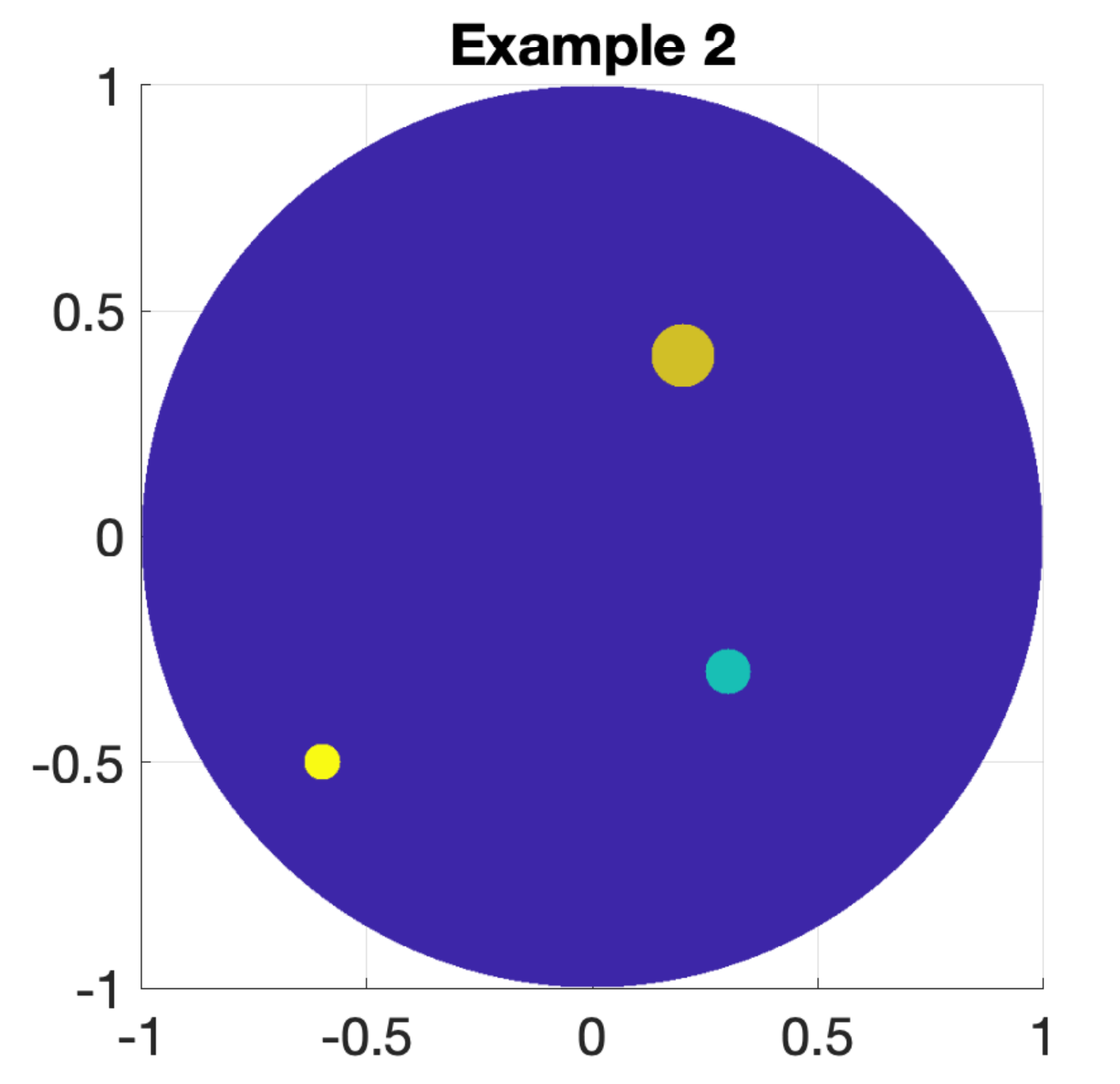}
    \phantomsubcaption
    \label{fig:convg_integration:b}
    \hfill
    \scriptsize{(c)}%
    \raisebox{-2ex}{%
    \includegraphics[width=0.34\textwidth]{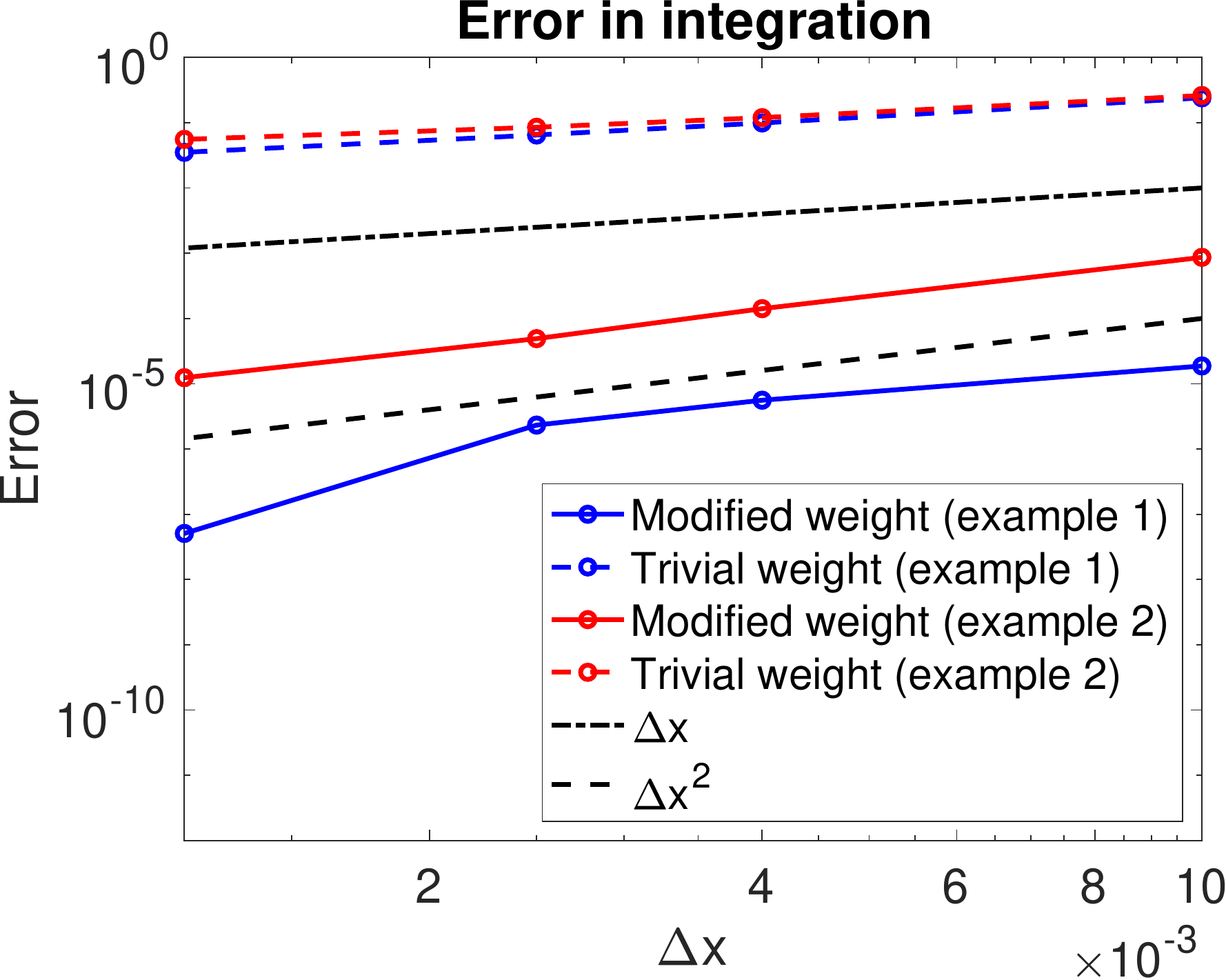}}
    \phantomsubcaption
    \label{fig:convg_integration:c}
  }
  \vspace*{-1ex}
  \caption{Two examples of the unit disk perforated by circular
    traps. (a) one trap centered at $\mathbf{x}_1 = (-0.5,0.5)$
    with radius $\varepsilon_1=0.05$. (b) three traps centered at
    $\mathbf{x}_1 = (0.3,-0.3)$, $\mathbf{x}_2 = (0.2,0.4)$, and
    $\mathbf{x}_3 = (-0.6,-0.5)$, with radii
    $\varepsilon_1=0.05, \, \varepsilon_2=0.07$, and
    $\varepsilon_3=0.04$, respectively. (c) accuracy of the numerical
    integration to compute the trap-free areas for (a) and (b), using
    trivial and modified weights.
  }
  \label{fig:convg_integration}
\end{figure}

Having confirmed the numerical accuracy and convergence of the CPM, we
now consider more intricate problems where analytic solutions are not
available.  In certain cases, the novel asymptotic approaches
developed later in \S~\ref{sec:analysis} are used to compare with our
computational results.

\subsection{MFPT in a disk with traps arranged on a ring}
\label{Ring_NTraps}

We consider a pattern of $m \geq 2$ circular traps that are
equally-spaced on a ring of radius $0 < r < 1$, concentric within the
unit disk.  In \cite{kolokolnikov2005optimizing} it was shown using
asymptotic analysis that for each $m\geq 2$ there is a unique ring
radius $r_c$ that minimizes the average MFPT for this pattern. We now
validate this result numerically. To do so, we solve
\eqref{MFPT_DiskSation} for a given $m$ with many different possible
radii $r$.  The numerically optimal ring radius $r_c$ is taken as the
value of $r$ for which the average MFPT is minimized.  Specifically,
we discretized the ring radius $r$ with a resolution of
$\Delta r = 0.0001$.  For each discrete value of $r$, we solved for
the average MFPT using the CPM with numerical grid spacing
$\dx = 0.004$.  We then took $r_c$ as the minimum value over the
resulting discrete set.

\begin{figure}[htbp]
    \centering
    \begin{subfigure}[b]{0.45\textwidth}
      \centering
        \includegraphics[width=0.9\textwidth]{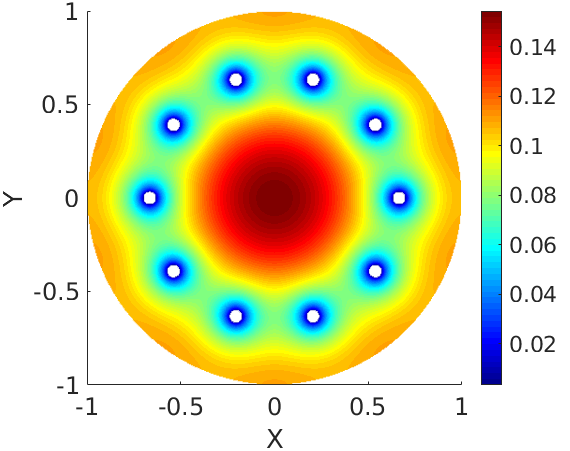}
        \caption{MFPT for the optimal 10 trap ring.}
        \label{Ring10}
    \end{subfigure}
    \hfill
    \begin{subfigure}[b]{0.45\textwidth}
      \centering
      \begin{tabular}{ c  c  c }
        \hline
    $m$ & Asymptotics & Numerics \\ \hline
	 2 & 0.4536 & 0.4533 \\
	 3 & 0.5517 & 0.5480 \\
	 4 & 0.5985 & 0.5987 \\
	 5 & 0.6251 & 0.6275 \\
	 6 & 0.6417 & 0.6411 \\
	 7 & 0.6527 & 0.6467 \\
	 8 & 0.6604 & 0.6609 \\
	 9 & 0.6662 & 0.6689 \\
	10 & 0.6706 & 0.6708 \\ \hline
  \end{tabular}
        \caption{Optimal ring radius $r_c$ for $m$ traps.}
        \label{RingRes}
      \end{subfigure}
      \vspace*{-3ex}
    \caption{The optimal ring radius $r_c$ for $m$ circular traps of
      radius $\varepsilon = 3 \times 10^{-3}$ that are equally-spaced
      on a ring concentric within a reflecting unit disk.  For each
      $m\geq 2$, the optimal radius $r_c$ minimizes the average MFPT
      for such a ring pattern of traps. (a) Optimal MFPT computed
      from the CPM with $m=10$. (b) Comparison of our
      numerical results with the asymptotic results obtained in
      \cite{kolokolnikov2005optimizing}. } 
    \label{Ntraps_Ring_config}
\end{figure}

Figure~\ref{Ring10} shows the MFPT for $m = 10$ traps on a ring with
the optimal radius $r_c = 0.6708$ computed by the procedure above.
The table in Figure~\ref{RingRes} shows a close comparison of our
numerical results with the asymptotic results obtained in
\cite{kolokolnikov2005optimizing}.

\subsection{Two stationary traps in an elliptical domain}\label{Static_2Trap_Ellipse}

Next, we consider the MFPT for a family of elliptical domains with
semi-minor axis $b$, with $b<1$, and semi-major axis $a={1/b}>1$ that
contains two circular absorbing traps of radius $\varepsilon$ centered
on the major axis. As $b$ is decreased from unity, an initial circular
domain gradually deforms into an elliptical region of increasing
eccentricity, with the area of the domain fixed at $\pi$. As $b$ is
varied, we will compute the optimal location of the traps that
minimize the average MFPT. For each fixed $b<1$, the centers of the
two traps are varied on the major axis with a step size of $0.01$, and
for each such configuration the average MFPT is computed.  The optimal
trap locations at the given $b$ correspond to where the average MFPT
is smallest.  The computations were done with a numerical grid spacing
of $\dx = 0.005$, and the semi-minor axis was decreased in steps of
$\Delta b = 0.02$.  Our numerical simulation predicts, as expected,
that the optimal locations of the two traps must be symmetric about
the minor axis. For the unit disk where $b=1$, our numerical results
yield that the optimal locations of the traps is at a distance
$x_0 = 0.450$ from the center of the disk.  This agrees with
computations in \S~\ref{Ring_NTraps} (see Figure~\ref{RingRes}) of a
two-trap ring pattern in a unit disk.

\begin{figure}[htbp]
    \centerline{
    \begin{subfigure}[b]{0.33\textwidth}
      \includegraphics[width=\textwidth]{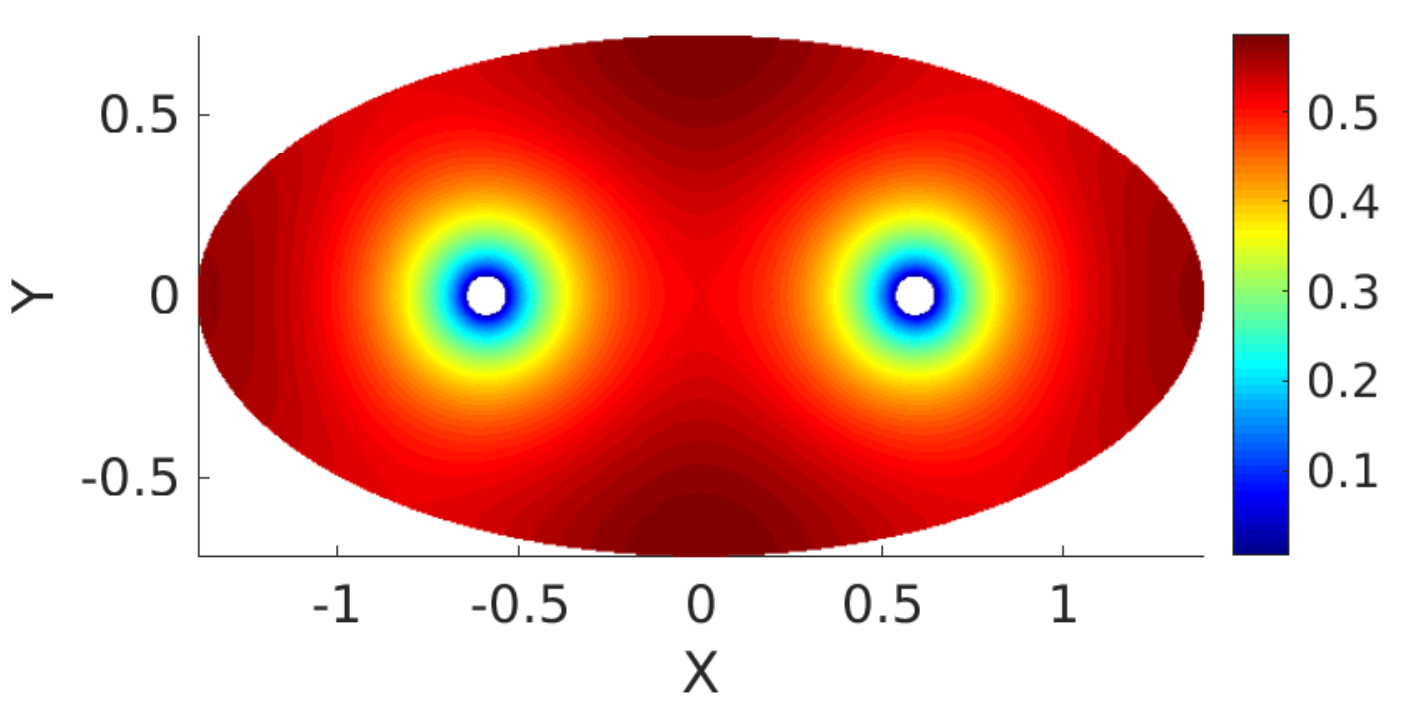}
      \caption{MFPT for optimal traps}
      \label{Ellpt_MFPT2}
    \end{subfigure}
    \begin{subfigure}[b]{0.33\textwidth}
        \includegraphics[width=\textwidth]{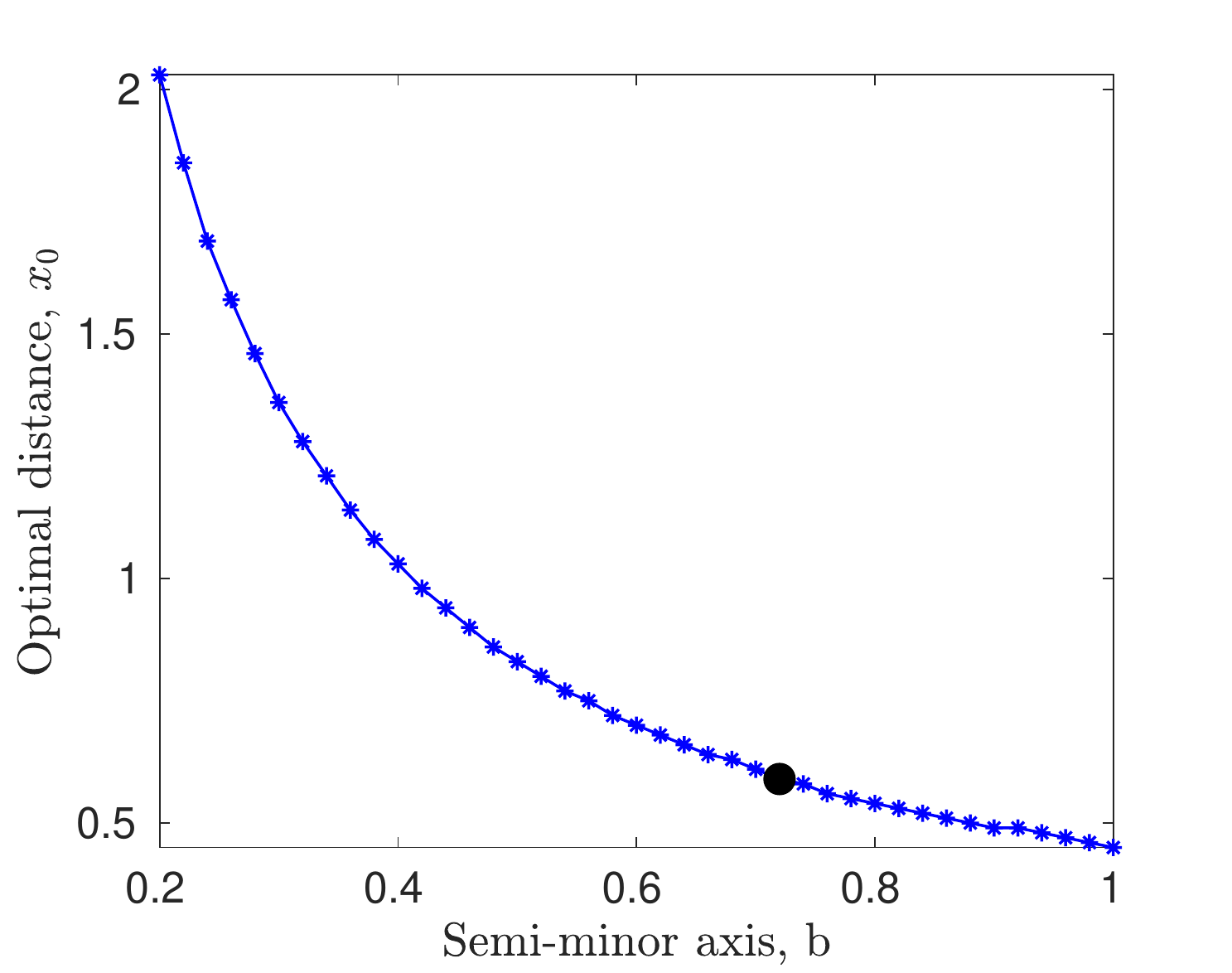}
        \caption{Optimal location of traps}
        \label{Ellpt_Opt_X0} 
    \end{subfigure}
    \begin{subfigure}[b]{0.33\textwidth}
        \includegraphics[width=\textwidth]{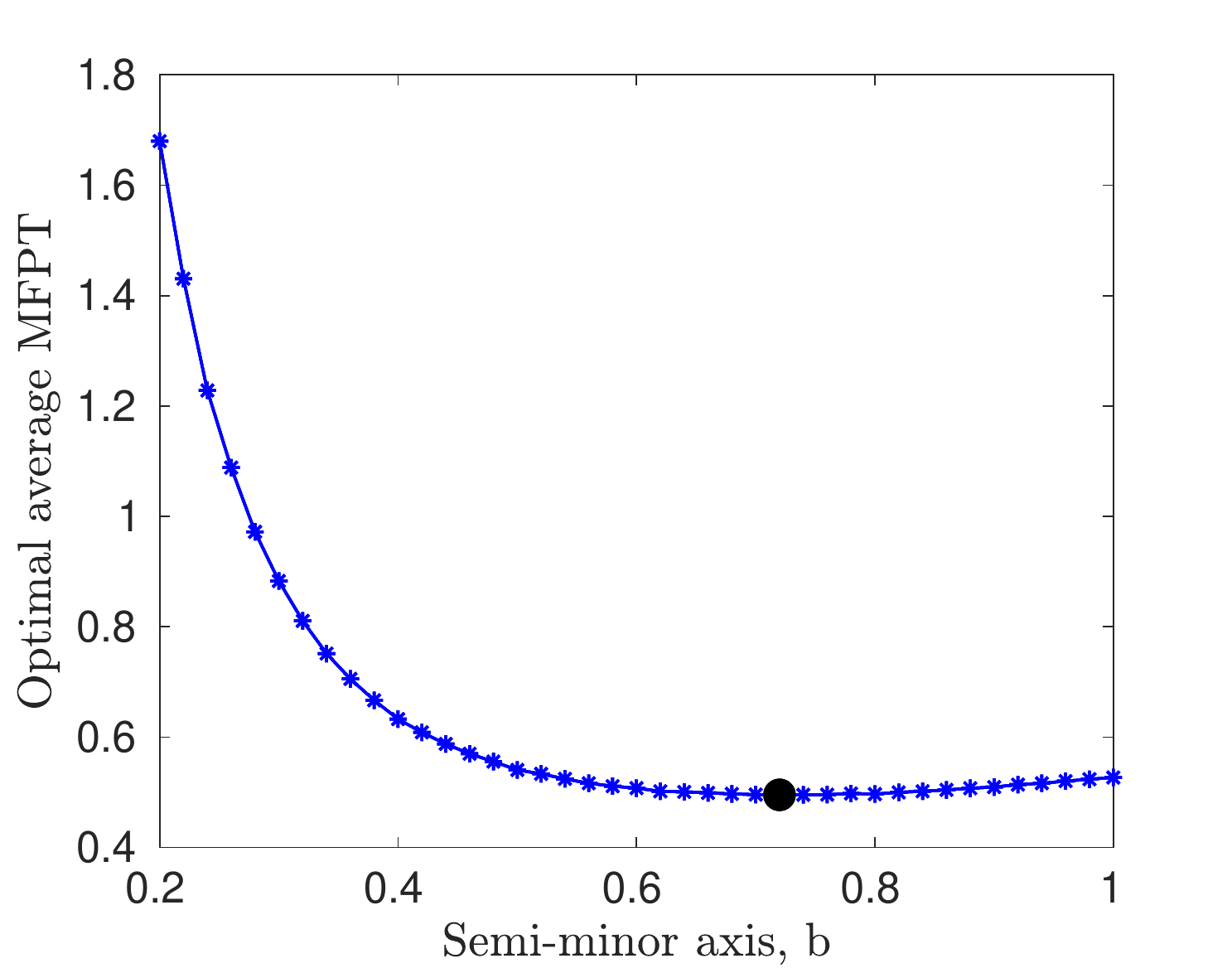}
        \caption{Optimal average MFPT}
        \label{Ellpt_Opt_MFPT}
    \end{subfigure}
    }
    \vspace*{-3ex}
    \caption{Two traps of radii
        $\varepsilon = 0.05$ on the major axis of an
        elliptical domain.  Left: with semi-major axis
        $a \approx 1.3889$ and semi-minor axis $b=1/a=0.72$, the optimal location
        for the traps are
        $(\pm 0.59,0)$. Middle: the optimal trap locations change as
        we shrink the minor axis.  Right: the average MFPT for
        optimal trap locations as the semi-minor axis is varied.  The
        dot is the globally minimal average MFPT
        $\overline{u}_{\textrm{opt}} = 0.4954$, over all ellipses of
        area $\pi$; it occurs in the configuration shown in~(a).
      }
    \label{Ellpt_MFPT_Opt}
\end{figure}

Figure~\ref{Ellpt_MFPT2} shows the MFPT for an elliptical region of
semi-major axis $a=1.3889$ and semi-minor axis $b=0.72$, with two
circular traps of radius $\varepsilon=0.05$ on its major axis centered
at $(\pm 0.59,0)$. These are the optimal locations of the traps for
this particular elliptical region.  Figures~\ref{Ellpt_Opt_X0}
and~\ref{Ellpt_Opt_MFPT} show the optimal locations of the traps and
the optimal average MFPT, respectively, as the semi-minor axis, $b$,
is decreased.  We observe from this figure that the optimal traps move
away from each other as $b$ decreases. This is because, as the
eccentricity of the ellipse increases, narrow regions at the two ends
of the major axis are created in which a Brownian particle can
``hide'' from the traps.  This effective ``pinning'' of particles by
the domain geometry increases their escape time.  In order to reduce
the escape time of such pinned particles---and thus the overall
average MFPT for the region---the traps need to move closer to the
ends of the major axis.

Figure \ref{Ellpt_Opt_MFPT} shows that as $b$ is decreased the optimal
average MFPT initially decreases until a global minimum
$\overline{u}_{\textrm{opt}} = 0.4954$ is reached at $b \approx
0.72$. This corresponds to traps that are at a distance $x_0 = 0.59$
from the center of the ellipse (see Figure~\ref{Ellpt_MFPT2} for the
MFPT of this pattern). This result suggests that the geometry that
gives the global minimum MFPT for the two-trap pattern is an
elliptical region with semi-major axis $a= 1.3889$ and semi-minor axis
$b=0.72$, and most notably is not the unit disk.  In
\S~\ref{sec:asymp_perturbed_unit_disk} we perform an asymptotic
analysis to determine the optimal MFPT and trap locations in near-disk
domains, which verifies that the global minimum of the MFPT is
\emph{not} attained by the unit disk but rather for a specific
elliptical domain. Moreover, in \S~\ref{sec:skinnyellipse} an
asymptotic approach based on thin domains is used to predict the
optimal trap locations and optimal average MFPT when $b\ll 1$.

\subsection{Three stationary traps in an ellipse}\label{Static_3Trap_Ellipse}

From \cite{kolokolnikov2005optimizing} a ring pattern of
three equally-spaced traps provides the optimal three-trap
configuration to minimize the average MFPT in the unit
disk. However, it is more intricate to determine the optimal
three-trap pattern in an elliptical domain. To do so numerically, we
employ the \textsc{Matlab} built-in function \texttt{particleswarm} for particle
swarming optimization (PSO)~\cite{kennedy2010}, to compute a local
minimum of the MFPT for an elliptical domain
$\frac{x^2}{a^2} + \frac{y^2}{b^2} = 1$ with $a=1.1$ and
$b={10/11}$. This optimal configuration is shown in the left panel
of Figure~\ref{fig:three_traps_ellipse_pso}. We use this
optimization result to initialize the numerical computation of local
minima of MFPT with the \textsc{Matlab} built-in function \texttt{fmincon} for
other values of $a$.
For $1.1 \leq a \leq 2$, and fixing the area of the ellipse at $\pi$,
in the right panel of Figure~\ref{fig:three_traps_ellipse_pso}
we plot the area of the triangle
formed by the numerically optimized locations of the three
traps. This figure shows that the three traps becomes colinear as
$a$ is increased.  In \S~\ref{sec:long_thin}, an asymptotic analysis,
tailored for long thin domains, is used to predict the optimal
locations of these three colinear traps for $a\gg 1$.

\begin{figure}[htbp]
\includegraphics[height=18ex]{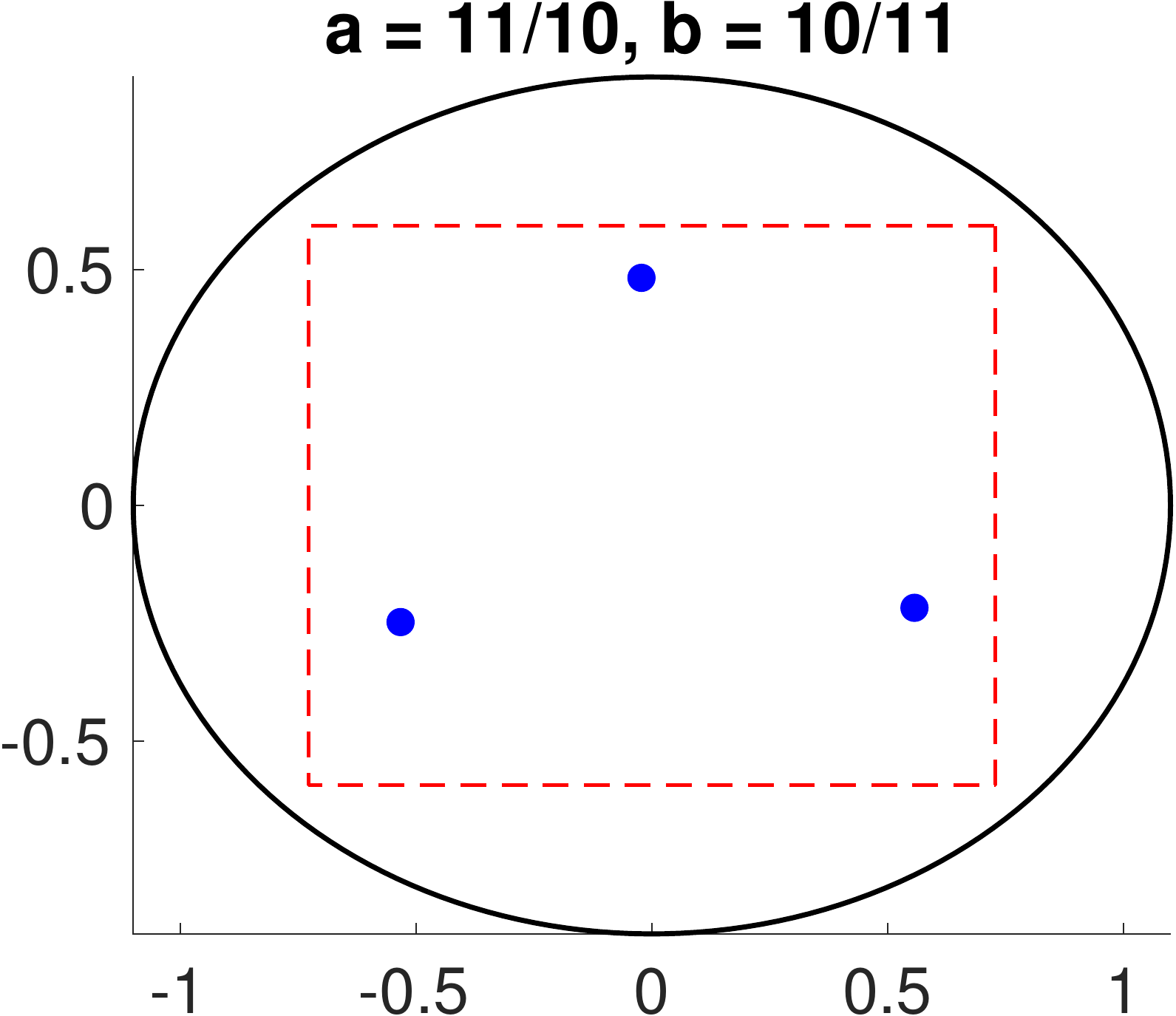}
\hfill
\includegraphics[height=14ex]{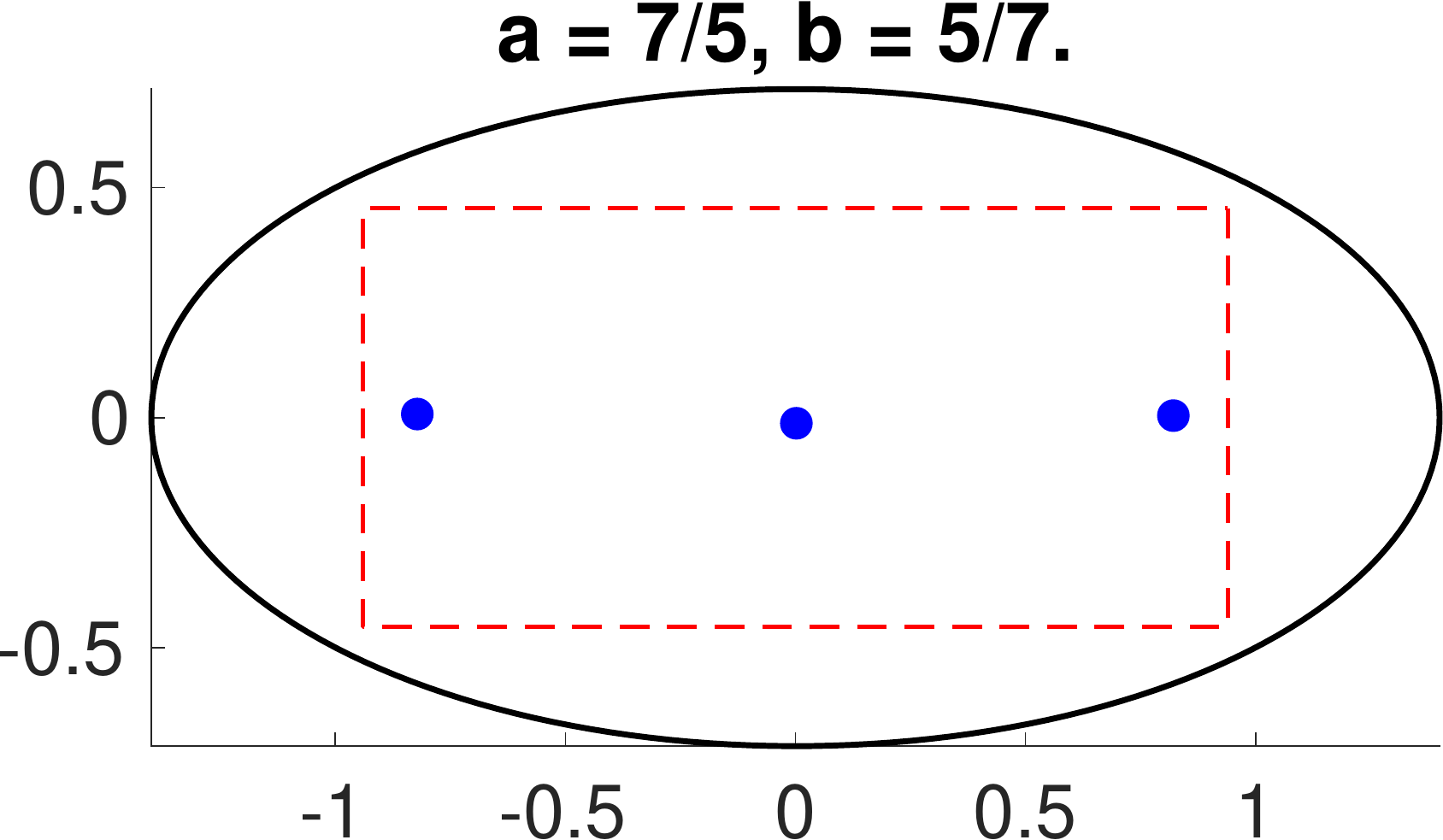}
\hfill
\raisebox{-1.4ex}{%
  \includegraphics[width=0.32\textwidth]{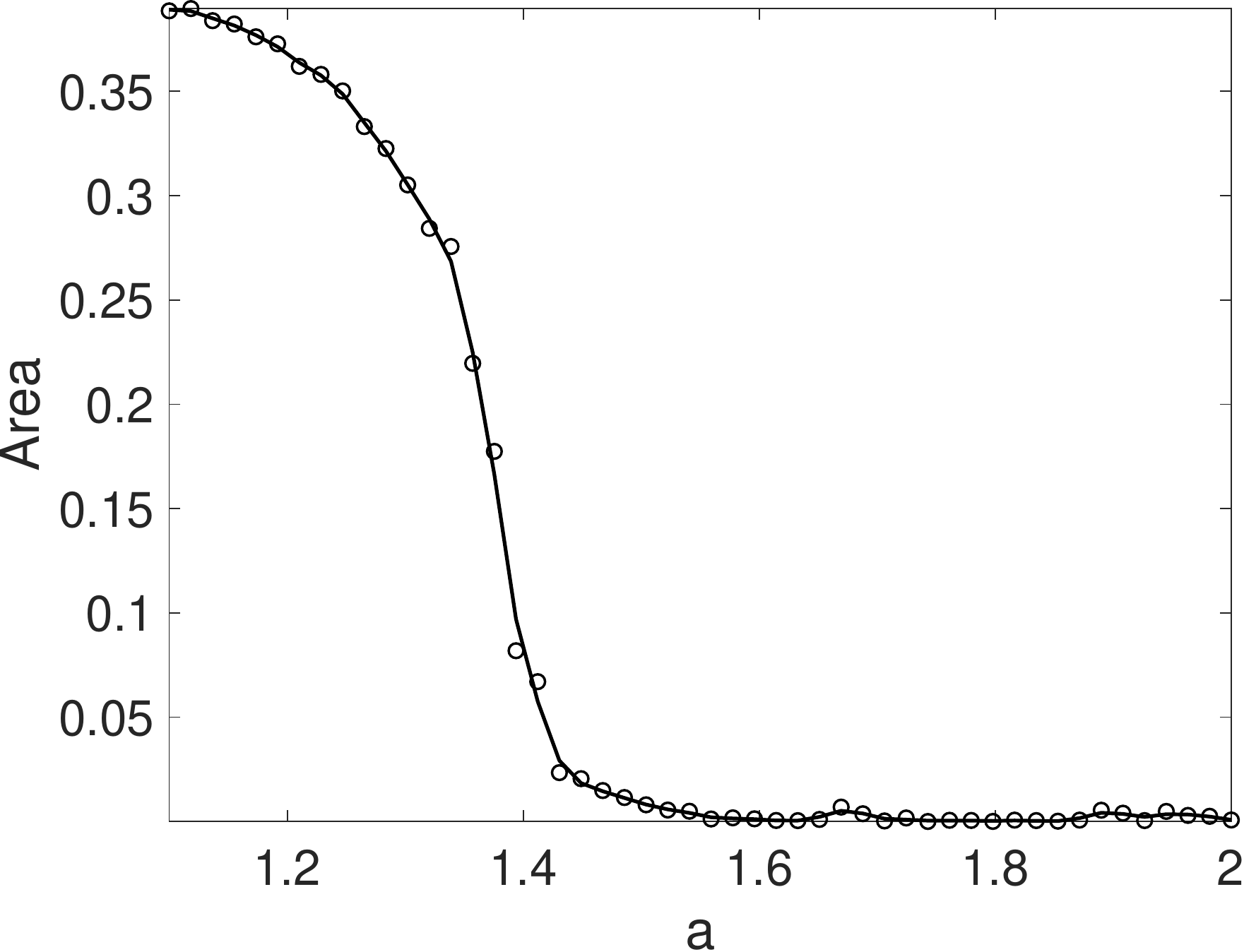}
}
\vspace*{-1ex}
\caption{The CPM and PSO is used to numerically compute local
  minimizers of the MFPT for three trap patterns in a one-parameter
  family of ellipses $(\frac{x}{a})^2 + (\frac{y}{b})^2 = 1$ with trap
  radius $\eps=0.05$, $1.1 \le a \le 2$ and $b={1/a}$. The right panel is for the area of the triangle formed by the three traps, which shows that the
  optimal traps become colinear as $a$ increases. The red dashed
  rectangles show the bounds used for PSO.}
\label{fig:three_traps_ellipse_pso}
\end{figure}

\subsection{Traps in star-shaped domains}\label{ThreeStarShapedDomain}

We briefly investigate the MFPT for multiple static traps in
a star-shaped domain, defined as the region bounded by
\begin{equation}
  r = 1 + \sigma \cos(\mc{N} \theta)\,, \quad 0 < \theta < 2\pi\,, \quad
  0<\sigma < 1\,,
\end{equation}
where $(r,\theta)$ are polar coordinates. Here $\mc{N}$ is a positive
integer that determines the number of folds in the domain boundary. We
use the CPM together with particle swarm optimization
\cite{kennedy2010} to numerically compute a local minimizer of the
MFPT for two specific examples.
In Figure~\ref{fig:3and4star} we show the optimal MFPT and trap locations
for a three-trap pattern in a three-fold star-shaped domain ($\mc{N}=3$) and
for a four-trap pattern in a four-fold star-shaped domain ($\mc{N}=4$).
In our asymptotic analysis of the optimal MFPT
in near-disk domains in \S~\ref{sec:asymp_perturbed_unit_disk} we will
predict the optimal trap locations when $m=\mc{N}$ and $\sigma\ll
1$. For $\sigma\ll 1$, we will show that the optimal trap locations
are aligned on rays where the boundary deflection is at a maximum.

\begin{figure}[htbp]
  \includegraphics[width=0.24\textwidth]{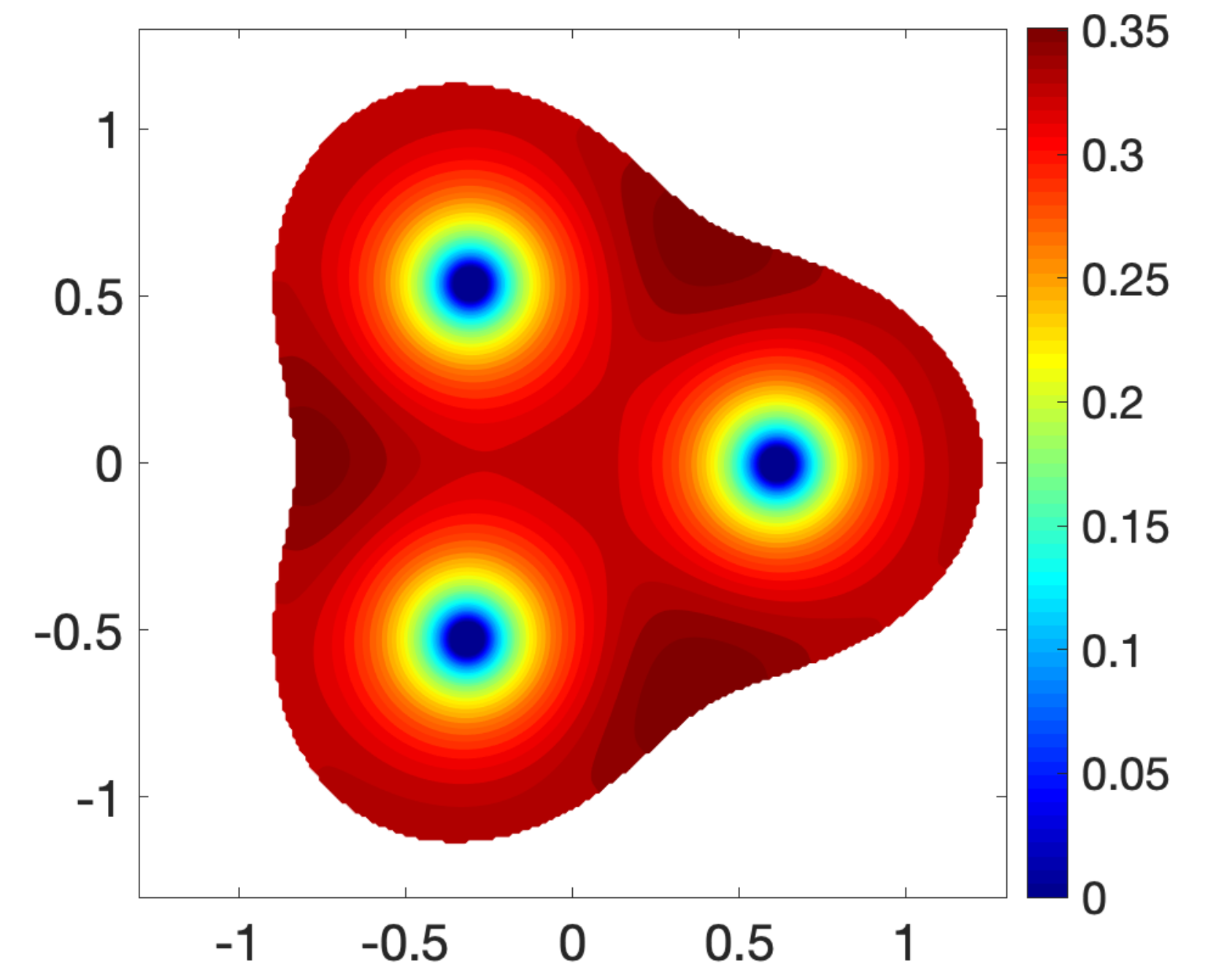}
  \includegraphics[width=0.24\textwidth]{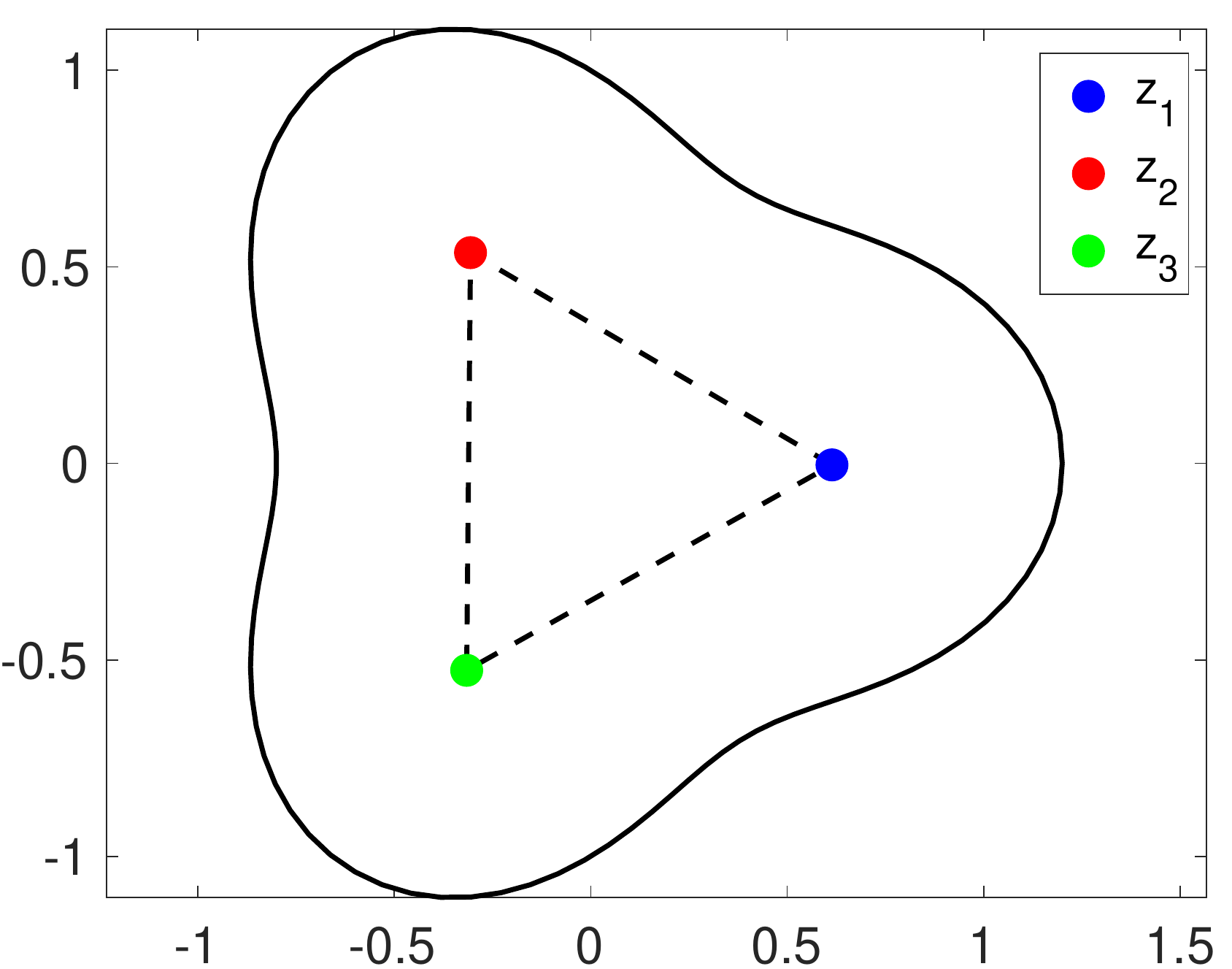}
  \hfill
  \includegraphics[width=0.24\textwidth]{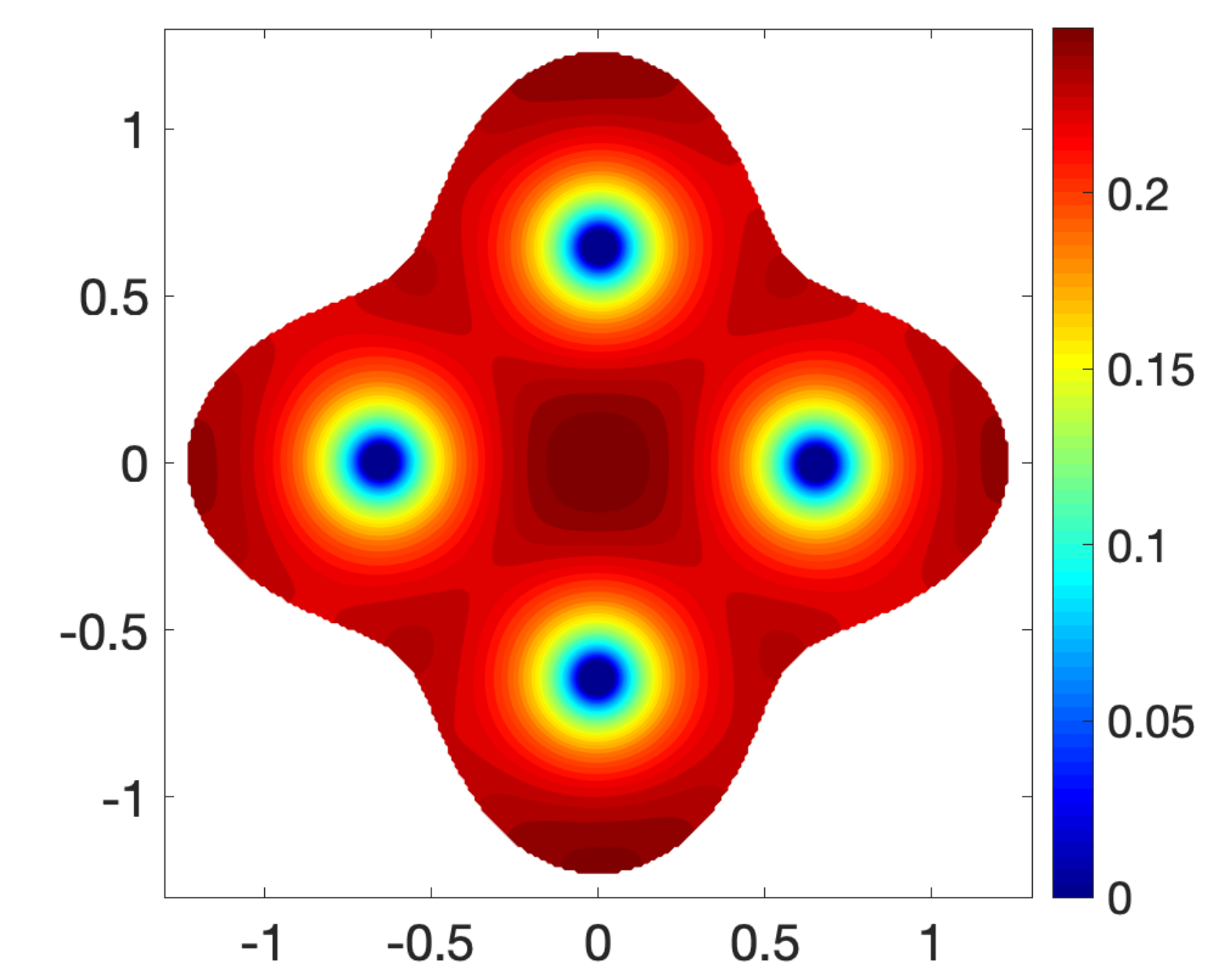}
  \includegraphics[width=0.24\textwidth]{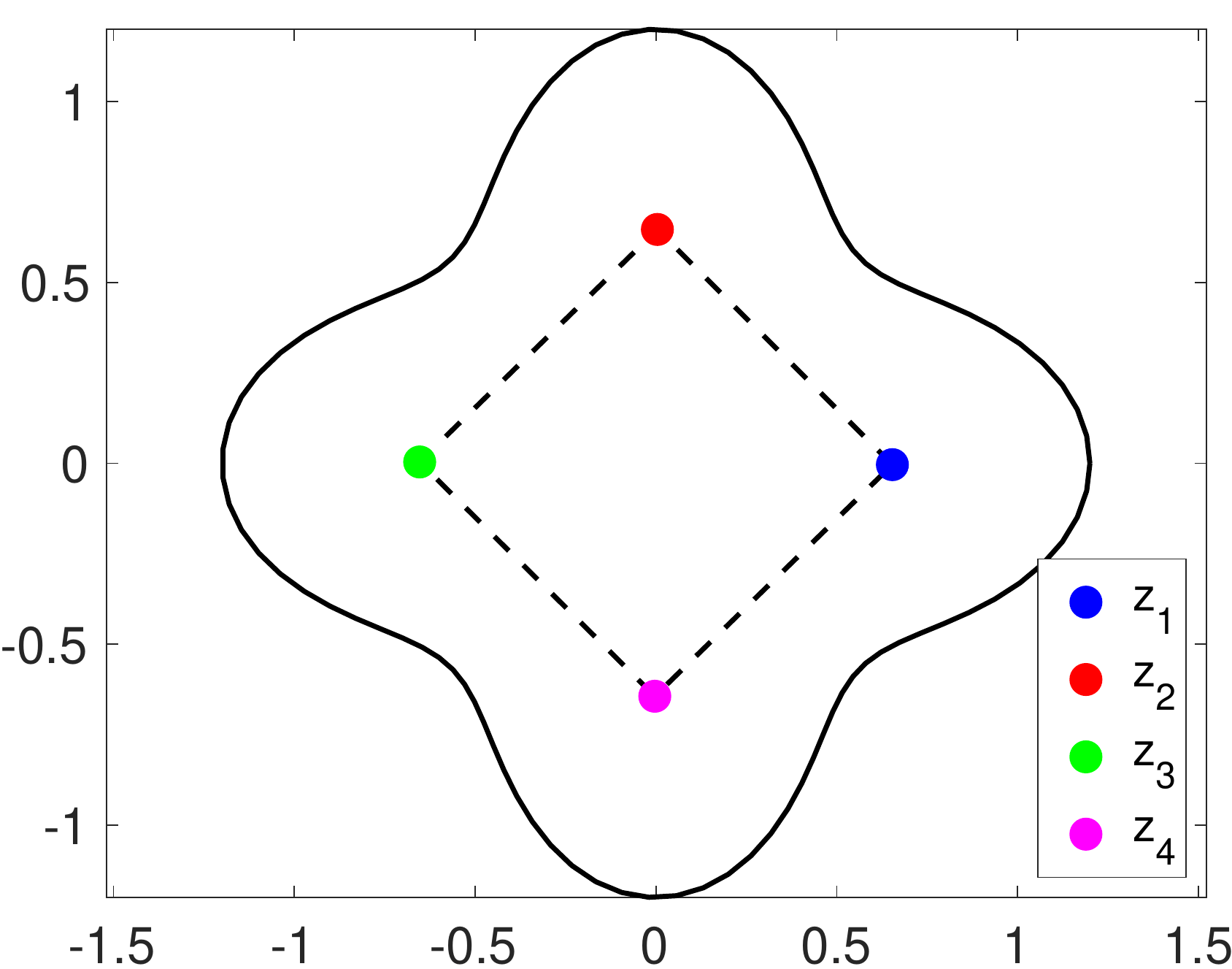}
  \caption{Numerically computed optimal $\mc{N}$-trap patterns
    in $\mc{N}$-fold star-shaped domains, found by PSO.
    Left two: PDE solution and optimal locations for $\mc{N}=3$;
    the optimal locations form an equilateral triangle on the
    circle of radius approximately $0.615$, to within a numerical
    error of $0.005$.
    Right two: $\mc{N}=4$; the square has vertices on the circle of radius approximately $0.65$.
    Here $\sigma=0.2$ and trap radii are $\varepsilon=0.05$.
  }
  \label{fig:3and4star}
\end{figure}

\section{Numerical computation for moving trap problems}
\label{MovingTrap_Section}

In this section, we will consider several problems for a Brownian
particle in a domain with moving traps.

\subsection{Convergence study}
\label{sec:moving_conv_study}

We first study the rate of convergence of our time relaxation approach
discussed in \S~\ref{sec:relax}. Consider the unit disk with a trap
moving in a circular path concentric within the disk at a fixed radius
$r_0 = 0.6$ from the origin.  At period $N$ of the algorithm, using
the notation in \S~\ref{sec:stopping_criterion}, we compute residual
$ \|\v{v}(NT) - \v{v}((N-1)T)\|_{L_2}$.  We study the rate of
convergence of the residual under different choices of mesh size
$\Delta x$, the radius of the trap $\varepsilon$, and the rotation
speed $\omega$.  In Figure~\ref{fig:convg_time_relaxation} we show that
the number of cycles for convergence is of $\mathcal{O}(1)$ and, in
particular, is independent of the mesh size $\Delta x$.  This figure
shows that the key factors that determine the rate of convergence are
the trap radius $\varepsilon$ and the angular frequency $\omega$ of
the circular trajectory of the trap. We use Forward Euler timestepping in these
numerical convergence studies.

\begin{figure}[htbp]
  \scriptsize{(a)}%
  \includegraphics[width=0.27\textwidth]{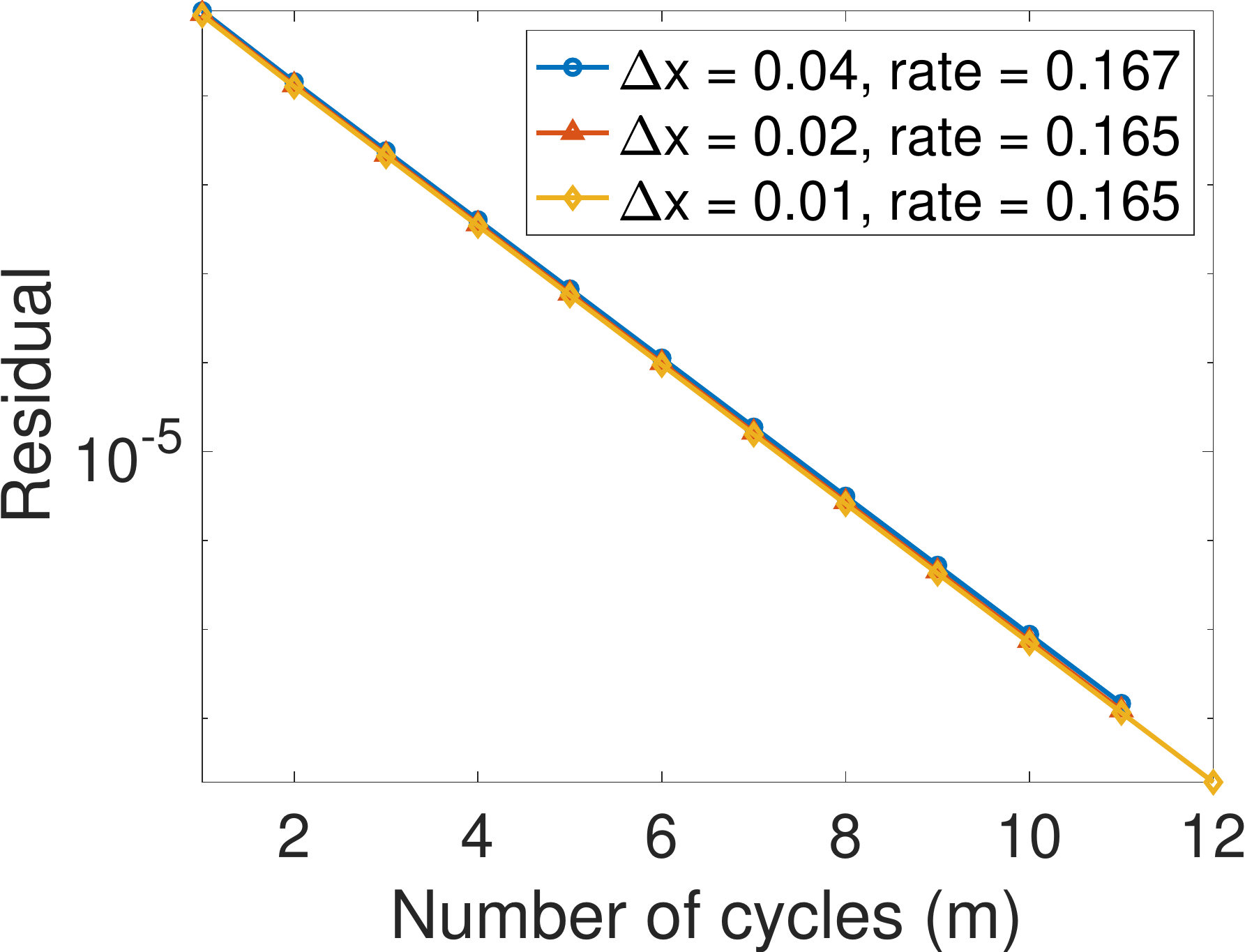}
  \hfill
  \scriptsize{(b)}%
  \includegraphics[width=0.27\textwidth]{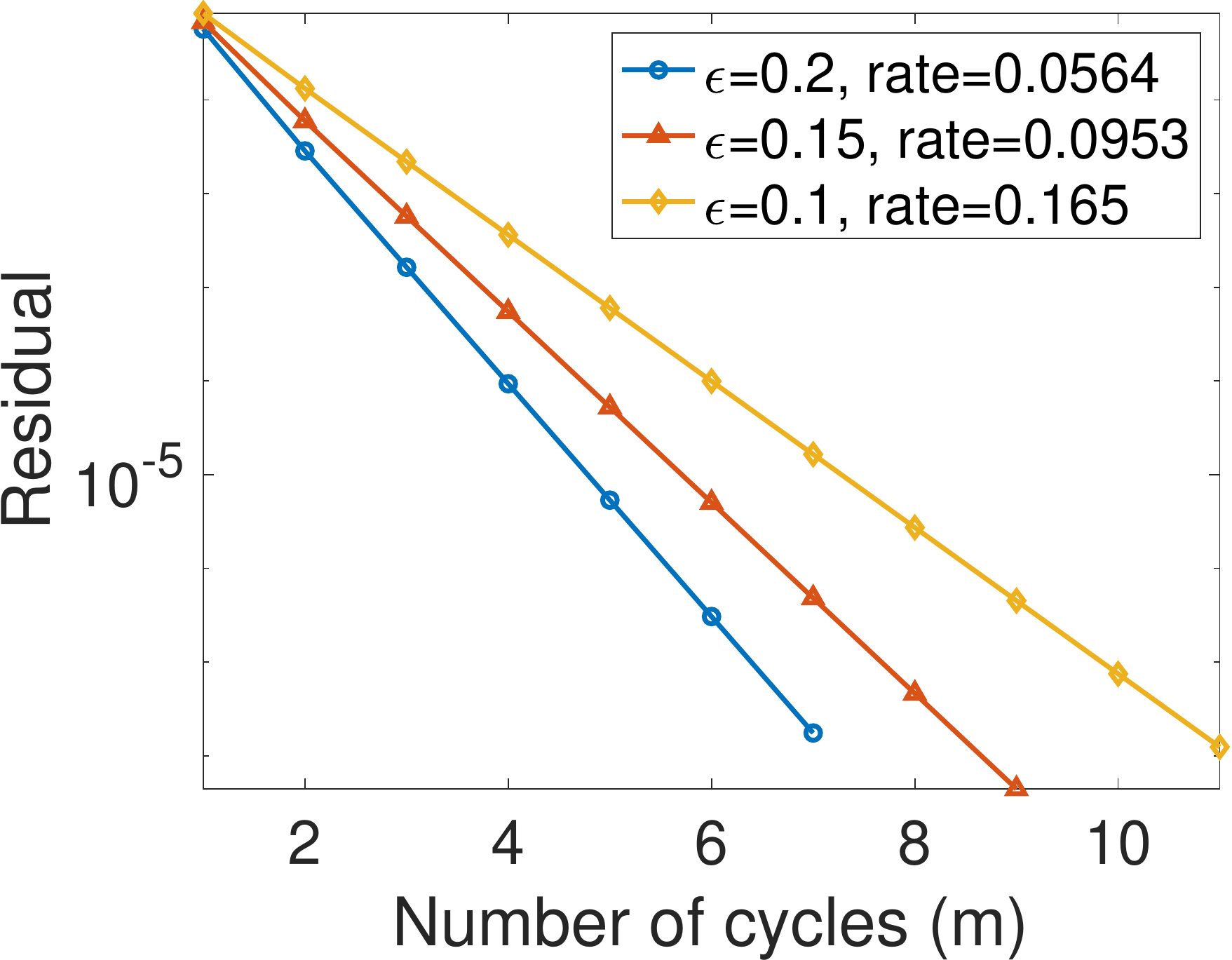}
  \hfill
  \scriptsize{(c)}%
  \includegraphics[width=0.27\textwidth]{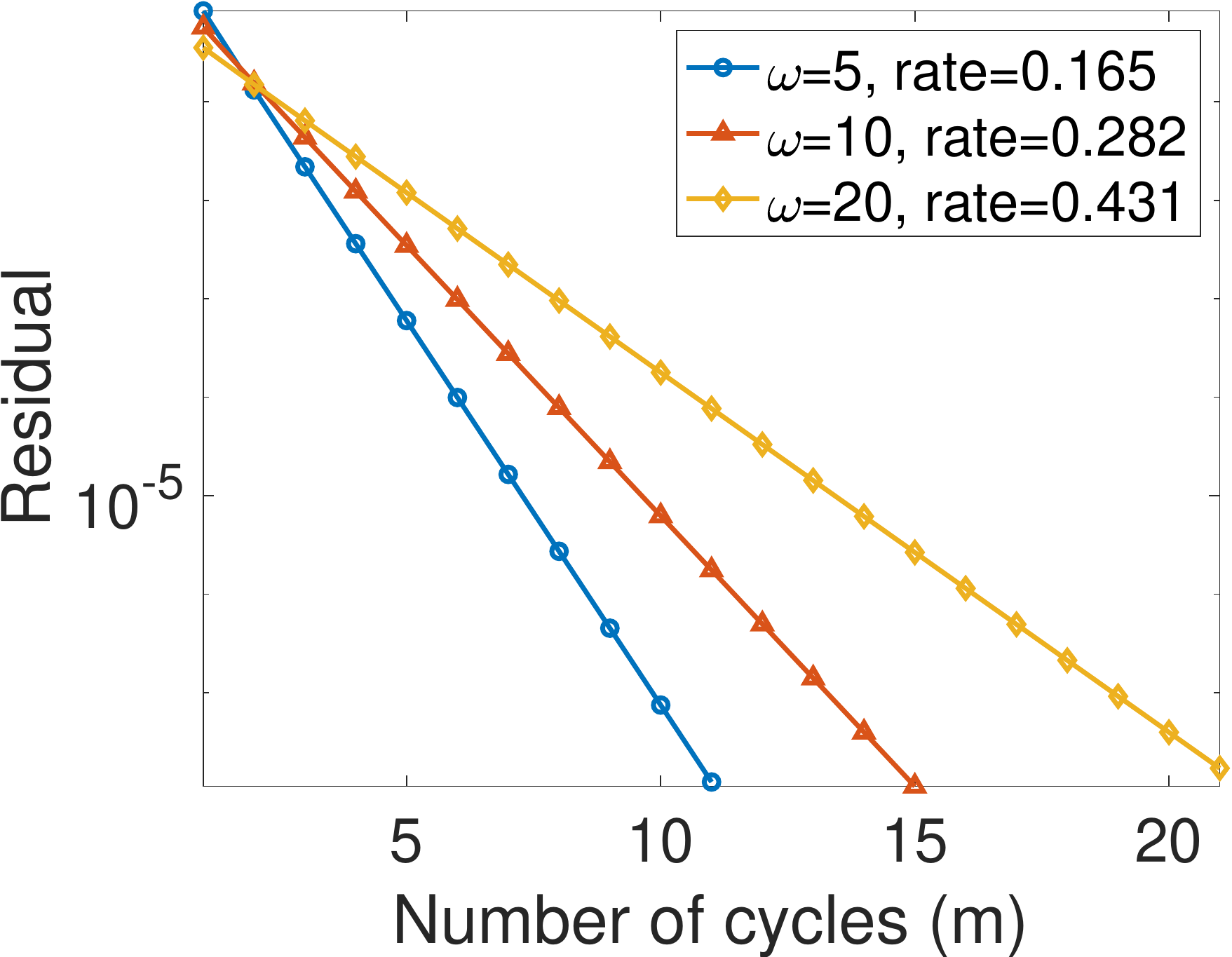}
  \vspace*{-1ex}
  \caption{Convergence studies for our time relaxation strategy for a
    trap moving on a ring of radius $r_0=0.6$ within the unit disk.  In
    (a) we fix the trap radius $\varepsilon=0.1$ and angular frequency
    $\omega=5$, and vary the mesh size with $\Delta x = 0.04, 0.02$ and
    $0.01$; the rate of convergence is almost independent of the mesh
    size.  In (b) we fix the angular frequency $\omega=5$ and mesh size
    $\Delta = 0.02$, and test three choices of trap radius
    $\varepsilon = 0.2, 0.15$ and $0.1$; smaller trap radii lead to
    slower convergence.  In (c) we fix the trap radius $\varepsilon=0.1$
    and mesh size $\Delta x = 0.02$, and consider three angular
    frequencies $\omega = 5, 10$ and $20$; larger angular frequencies
    lead to slower convergence.
  }
\label{fig:convg_time_relaxation}
\end{figure}

\subsection{Optimizing the radius of rotation of a moving trap in a disk}\label{sec:opt_disk}

Consider an absorbing circular trap of radius $\varepsilon = 0.05$
that rotates on a ring of radius $r$ about the center of a reflecting
unit disk at a constant angular frequency $\omega$, as illustrated in
Figure~\ref{particle_disk_examples:b}.  For any fixed $\omega$ and $r$
value, we can compute the MFPT using our time relaxation strategy with
mesh size $\dx = 0.01$, and forward Euler time-stepping with
$\Delta t = {\dx/f(\omega)}$, where $f(\omega)$ is a linear functions
of the angular frequency $\omega$.  The iteration proceeds over many
cycles until the tolerance from \S~\ref{sec:moving_conv_study} is
satisfied.  A typical result is shown, at a fixed instant in time, in
Figure~\ref{RotateTrap_diska}.

To estimate numerically the radius $r_{\textrm{opt}}(\omega)$ of
rotation of the trap that minimizes the average MFPT as a function of
$\omega$, we choose a discrete set of $\omega$ values and for each
such value estimate $r_{\textrm{opt}}$ by computing the average MFPT
for different discrete radii of rotation of the trap. We then record
the $r$ value that gives the minimum average MFPT as
$r_{\textrm{opt}}$.  In choosing the discrete radii set, various
values of $\Delta r$ were used, depending on $\omega$.  The results
are shown in Figure~\ref{RotateTrap_diskb}.  The use of discrete sets
of $r$ values induces some mild stair-casing artifacts into the plot.
In Figure~\ref{RotateTrap_disk} (and elsewhere), we have added a
heuristic fit to the data points.
\begin{figure}[htbp]
  \centering
  \makebox{
    \raisebox{0.5ex}{\small{(a)}}
    \includegraphics[width=0.37\textwidth]{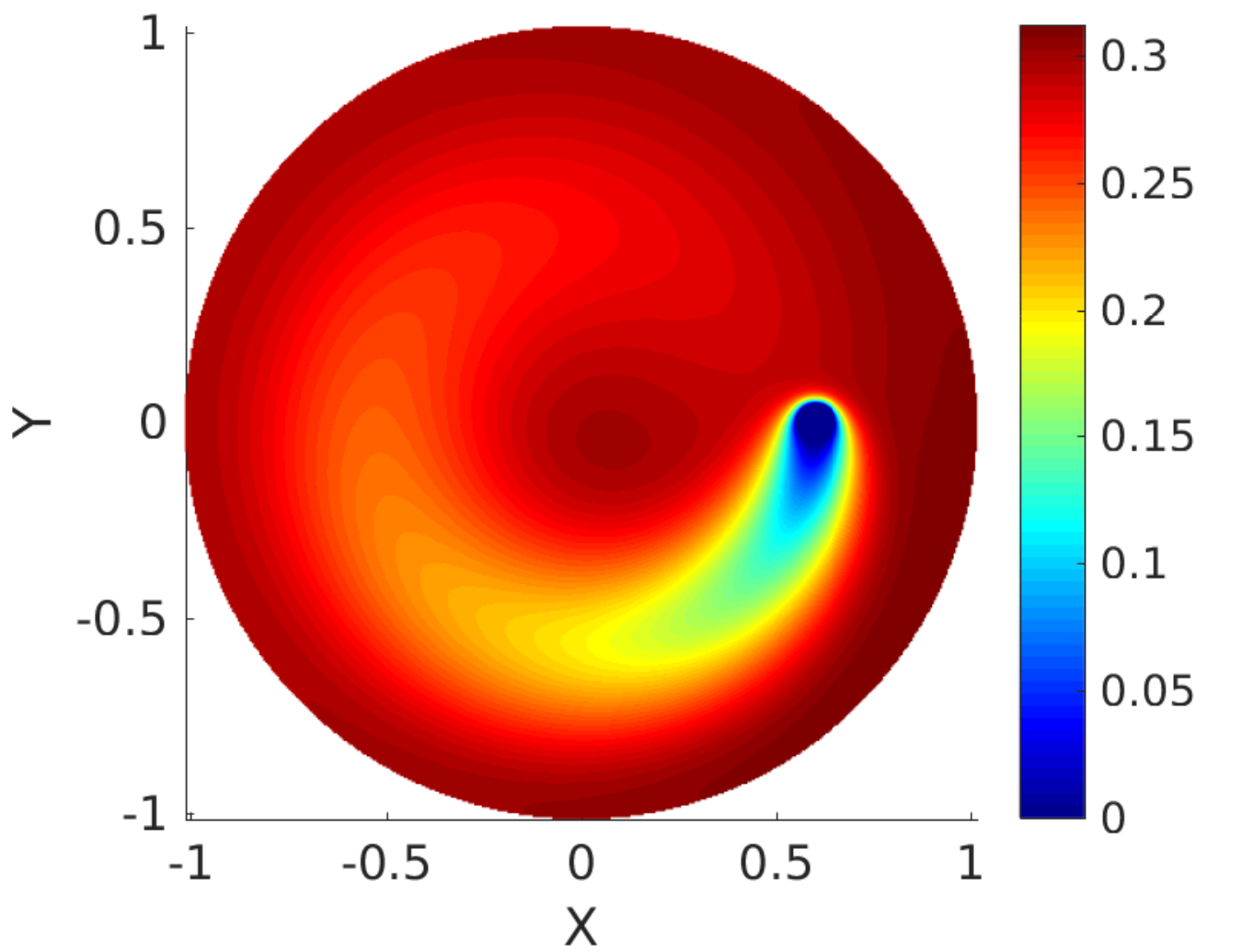}
    \phantomsubcaption
    \label{RotateTrap_diska}
  }
  \qquad
  \makebox{
    \raisebox{0.5ex}{\small{(b)}}
    \includegraphics[width=0.38\textwidth]{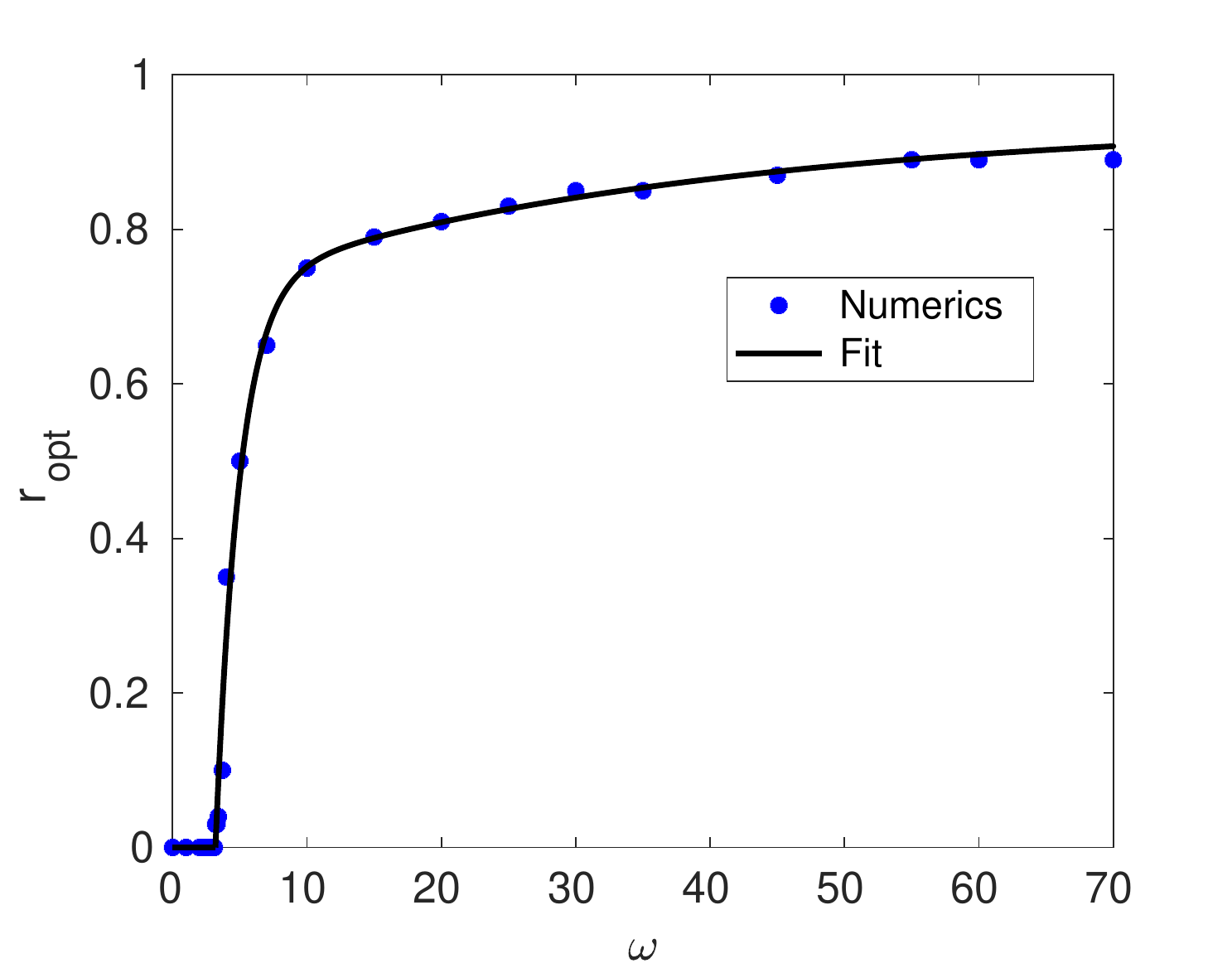}
    \phantomsubcaption
    \label{RotateTrap_diskb}
  }
  \vspace*{-1ex}
  \caption{ Left: the MFPT at a given time for a circular
      trap of radius $\varepsilon = 0.05$ rotating at an angular
      frequency of $\omega = 100$ about the center of a unit disk on a
      ring of radius $r=0.6$.  Right: the optimal radius of
      rotation $r_{\textrm{opt}}(\omega)$ that minimizes the average
      MFPT at a given rotation frequency $\omega$.}
  \label{RotateTrap_disk}
\end{figure}

From Figure~\ref{RotateTrap_diskb} we observe that there is a critical
rotation frequency $\omega_b$, estimated numerically as
$\omega_b \approx 3.131$, where the optimal radius of rotation changes
from a zero to a positive value. When $\omega < \omega_b$, the
location of the trap that minimizes the average MFPT is at the center
of the unit disk. Alternatively, when $\omega > \omega_b$, the optimal
trap moves away from the center of the domain.  This problem has
previously been studied analytically in \cite{tzou2015mean} using
asymptotic analysis valid in the limit of small trap radius.  In
\cite{tzou2015mean}, the critical value of $\omega_b$ was calculated
asymptotically as $\omega_b \approx 3.026$, which is close to what we
obtained numerically.

\subsection{Optimizing the trajectory of a trap in an elliptical region}

Next, we consider a circular absorbing circular trap of radius
$\varepsilon=0.05$ that is rotating at constant angular frequency on
an elliptical orbit about the center of an elliptical region as shown
in Figure~\ref{MFPT_rotateEll}.  The elliptical path for the trap is
taken as $(x,y) = (\alpha \cos(\omega t), \beta \sin(\omega t)$, where
$\alpha=ra$, $\beta=rb$, and $a$ and $b$ are the semi-major and
semi-minor axis of the elliptical region, respectively. We choose
$a=4/3$ and $b={1/a} = 3/4$, so that the area of the ellipse is the
same as that for the unit disk.  The parameter
$0 < r < (1-\varepsilon)$, referred to as the radius of rotation, is
used to stretch or shrink the orbit of the trap.  This
parameterization ensures that the eccentricity of all elliptical paths
of the trap is the same as that of the domain boundary.

Similar to that done in \S~\ref{sec:opt_disk}, for various angular
frequencies $\omega$ we numerically determine the optimal radius of
rotation $r_{\textrm{opt}}(\omega)$ that minimizes the average
MFPT. The results are shown in Figure \ref{Optimal_Ell}. As similar to
the case of the unit disk, we observe for the elliptical domain that
there is a critical value of $\omega$ where the optimal radius
bifurcates from the origin. We estimate this numerically as
$\omega_b \approx 2.65$.

\begin{figure}
    \centering
    \makebox{
      \raisebox{0.5ex}{\small{(a)}}
      \includegraphics[width=0.4\textwidth]{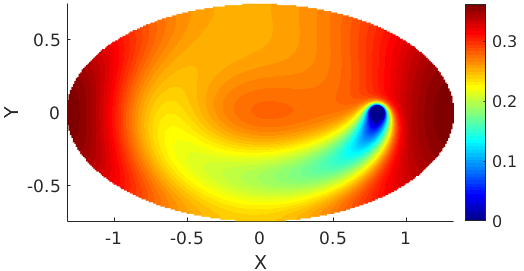}
      \phantomsubcaption
      \label{MFPT_rotateEll}
    }
    \qquad
    \makebox{
      \raisebox{0.5ex}{\small{(b)}}
      \includegraphics[width=0.38\textwidth]{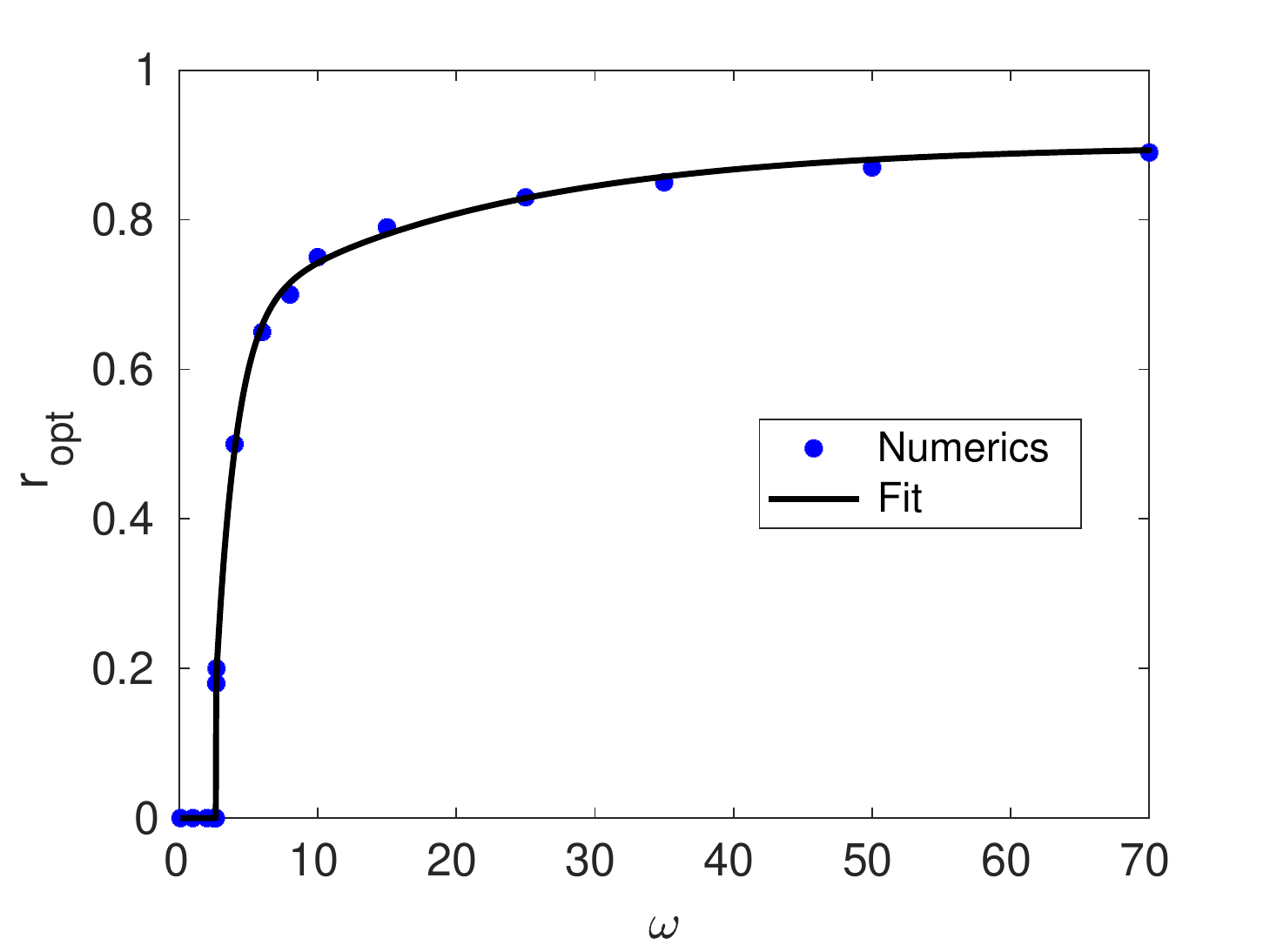}
      \phantomsubcaption
      \label{Optimal_Ell}
    }
    \vspace*{-1ex}
    \caption{The MFPT for a moving trap of radius $\varepsilon=0.05$ in an ellipse.
        The trap rotates on an elliptical path with semi-major axis
        $\alpha=ra$ and semi-minor axis $\beta=rb$ in an elliptical
        region with semi-major axis $a=4/3$ and semi-minor axis
        $b=3/4$.
        (a) MFPT at an instant in time with $\omega = 100$ and $r = 0.6$.
        (b) The optimal radius $r_{\textrm{opt}}(\omega)$ which minimizes the average MFPT
        for each $\omega$.}
    \label{RotateTrap_Ell}
\end{figure}

\subsection{Optimizing one rotating trap and one fixed trap in a disk}\label{sec:TwoTrapsDisk}

Next, we consider the unit disk in which there are two circular
absorbing traps each of radius $\varepsilon=0.05$. One of the traps is
fixed at the center of the disk while the other one is rotating at
constant angular frequency $\omega$ about the center of the disk on a
ring of radius $r$ concentric within the disk.
\begin{figure}
    \centering
    \makebox{
      \raisebox{0.5ex}{\small{(a)}}
      \includegraphics[width=0.35\textwidth]{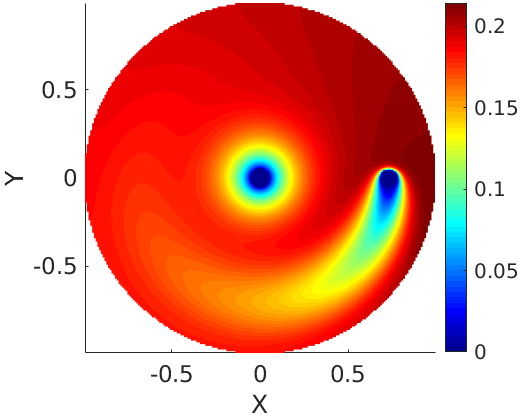}
      \phantomsubcaption
      \label{MFPT_disk_2traps}
    }
    \qquad
    \makebox{
      \raisebox{0.5ex}{\small{(b)}}
      \includegraphics[width=0.38\textwidth]{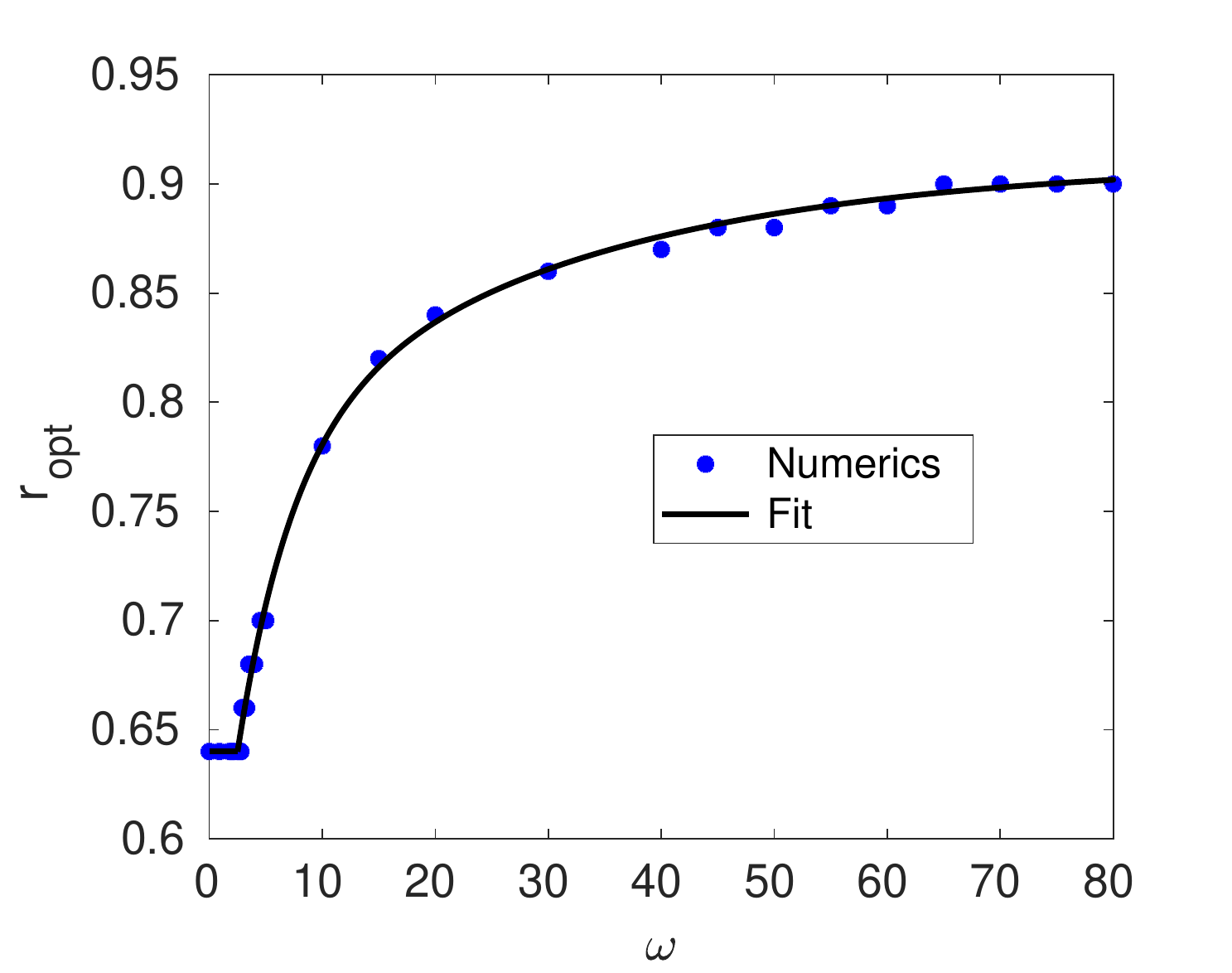}
      \phantomsubcaption
      \label{Opt_disk_2traps}
    }
    \vspace*{-1ex}
    \caption{The average MFPT for a unit disk with a
      trap at the center and a trap rotating
      with angular frequency $\omega$ around the center at radius $r$.
      The traps have radii $\varepsilon=0.05$.
      (a) MFPT at an instant in time with $r = 0.6$ and $\omega = 100$.
      (b) The optimal radius $r_{\textrm{opt}}(\omega)$ for the moving trap, which minimizes the average MFPT
      for each $\omega$.
      These values were found using a discrete search with $\Delta r = 0.01$.}
    \label{Disk_2traps}
\end{figure}
As a function of $\omega$, we proceed similarly to \S~\ref{sec:opt_disk}
to estimate numerically the radius of rotation of
the moving trap that minimizes the average MFPT. The results for the
optimal radius are shown in Figure~\ref{Opt_disk_2traps}. From this
figure, we observe that there is a specific angular frequency
$\omega_b$, estimated as $\omega_b\approx 2.5$, at which the optimal
radius first begins to increase from the fixed value
$r_{\textrm{opt}} = 0.64$ when $\omega$ increases beyond $\omega_b$.
This critical frequency is lower than that computed in \S~\ref{sec:opt_disk}
for a single rotating trap in the unit disk. An
analysis to predict the optimal radius in the fast rotation limit
$\omega\gg 1$ for this problem is given in \S~\ref{sec:fastrottrap}.

\section{Analysis}
\label{sec:analysis}

In this section, we provide some new analytical results to confirm
some of our numerical findings. First, in
\S~\ref{sec:asymp_perturbed_unit_disk} we use strong localized
perturbation theory (cf.~\cite{ward2018spots}, \cite{ward1993strong}),
to confirm some of our predictions on the optimum locations of steady
traps in perturbed disk-shaped domains. Next, in
\S~\ref{sec:skinnyellipse} we use a novel singular perturbation approach
to estimate optimal locations of colinear traps in long thin
domains. Finally, in \S~\ref{sec:fastrottrap}, we develop an
analytical approach to study the moving trap problem in a disk in the
limit of fast rotation.  For these three problems we will focus on
summarizing our main analytical results: a detailed derivation of them
is given in the Supplementary Material.

\subsection{Asymptotic analysis of the MFPT for a perturbed unit disk}
\label{sec:asymp_perturbed_unit_disk}

We begin by calculating the MFPT for a slightly perturbed unit disk
that contains $m$ traps. In the unit disk, and for small values of
$m$, the optimal trap configuration consists of equally-spaced traps
on a ring concentric within the disk
\cite{kolokolnikov2005optimizing}. When the disk is perturbed into a
star-shaped domain with $\mc{N}$ folds, we will develop an asymptotic
method to determine how the optimal trap locations and optimal average
MFPT associated with the unit disk are perturbed. For the special case
where $m=\mc{N}$ explicit results for these quantities are
derived. The results from this analysis are used to confirm some of
the numerical results in \S~\ref{Static_2Trap_Ellipse} and
\S~\ref{ThreeStarShapedDomain}.

For $\sigma \ll 1$, we use polar coordinates to define the perturbed
unit disk as
\begin{equation}\label{PerturbPar}
  \Omega_{\sigma} = \Big{ \{ }(r,\theta)\, \Big{|}\, 0 < r \leq 1 +
  \sigma \cos(\mc{N} \theta),  \,\,  0 \leq \theta \leq 2 \pi  \Big{ \} }.
\end{equation}
Observe that $\Omega_{\sigma}$ is a star-shaped domain with $\mc{N}$
folds for any $\sigma > 0$, and it tends to the unit disk, denoted by
$\Omega$, as $\sigma \to 0$.  From \eqref{MFPT_DiskSation} the
MFPT for a Brownian particle starting at a point
$\x \in \pdedomain_{\sigma}$ to be absorbed by a trap satisfies
\begin{equation}\label{Ellip_Model}
\begin{split}
D\, \nabla^2 u  = -1, &\quad \x \in \pdedomain_{\sigma};\\
\partial_n u = 0, \quad \x \in \partial \Omega_{\sigma}; &
\qquad u = 0, \quad \x \in \partial \Omega_{\varepsilon j},
\quad j = 0, \dots,m-1,
\end{split}
\end{equation} 
where
$\pdedomain_{\sigma} \equiv \Omega_{\sigma} \setminus \cup_{j=1}^{m}
\Omega_{\varepsilon j}$ is the perturbed domain with the trap set
deleted, while
$\Omega_{\varepsilon j} = \{\x : |\x - \x_j| \leq \varepsilon \}$ is
the $j^{\text{th}}$ absorbing trap centered at
$\x_j = r_c \exp \big{(} i ( 2\pi j/m + \psi) \big{)}$ with
$\psi > 0$, for $j = 0, \dots, m-1$ on the ring of radius $r_c$. A
simple calculation shows that the area of the star-shaped domain is
$|\Omega_{\sigma}| = |\Omega| + {\mathcal O}(\sigma^2)$. Our goal is
to use perturbation methods to reduce the MFPT problem for the
perturbed disk \eqref{Ellip_Model} to problems involving the unit disk.
Using the parameterization
$\x \equiv (x,y)= (r \cos(\theta),r \sin(\theta))$, the Neumann
boundary condition in \eqref{Ellip_Model} can be written as
\begin{equation}\label{PolarBC}
\begin{split}
  u_r -  &\frac{ \sigma h_{\theta}}{(1 + \sigma h)^2} u_{\theta}= 0 \quad
  \text{on} \quad r = 1 +\sigma h,\quad \text{where}\quad
  h(\theta) = \cos(\mc{N} \theta).
\end{split}
\end{equation} 

We begin by expanding the MFPT $u$ in terms of $\sigma\ll 1$ as
\begin{equation}\label{U_Sigma_Expand}
\begin{split}
  u(r,\theta; \sigma) = u_0(r,\theta) + \sigma u_1(r,\theta) +
  \sigma^2 u_2(r,\theta) + \ldots.
\end{split}
\end{equation} 
Upon substituting \eqref{U_Sigma_Expand} into \eqref{Ellip_Model} and
\eqref{PolarBC}, and collecting terms in powers of $\sigma$, we derive
that the leading-order MFPT problem satisfies
\begin{equation}\label{LeadingOrder}
\begin{split}
D\,  \nabla^2 u_0  = -1, &\quad \x \in \pdedomain;\\
\partial_n u_0 = 0, \quad \text{on} \quad r = 1; & \qquad u_0 = 0,
\quad \x \in \partial \Omega_{\varepsilon j}, \qquad j = 0, \dots,m-1,
\end{split}
\end{equation}
where
  $\pdedomain \equiv \Omega \setminus \cup_{j=1}^{m}
  \Omega_{\varepsilon j}$. At next order, the $\mathcal{O}(\sigma)$
  problem is
\begin{equation}\label{OrderSigma}
\begin{split}
  \nabla^2 u_1 = 0 , &\quad \x \in \pdedomain; \qquad \partial_r u_1 =
  -h u_{0rr} + h_{ \theta} u_{0 \theta}, \quad \text{on} \quad r = 1;
  \\ \quad u_1 &=0, \quad \x \in \partial \Omega_{\varepsilon j},
  \qquad j = 0, \dots,m-1,
\end{split}
\end{equation}
with $h \equiv h(\theta)$ as given in \eqref{PolarBC}. We emphasize
that the leading-order problem \eqref{LeadingOrder} and the
$\mathcal{O}(\sigma)$ problem \eqref{OrderSigma}, are
formulated on the unit disk and not on the perturbed disk.
Assuming $\varepsilon^2 \ll \sigma $, we use \eqref{AveMFPT} and
$|\Omega_{\sigma}| = |\Omega| + {\mathcal O}(\sigma^2)$ to derive an
expansion for the average MFPT for the perturbed disk in terms of the
unit disk as
\begin{equation}\label{AveMFPT_Perturb}
\begin{split}
  \overline{u} = \frac{1}{|\Omega|}\!\int_{\Omega } \! u_0(\x)
  \,\text{d}\x + \sigma \! \left[ \frac{1}{|\Omega|}\!\int_{\Omega }
   \! u_1(\x) \,\text{d}\x + \frac{1}{|\Omega|}\!\int_{0}^{2\pi}
   \! h(\theta)\,u_0|_{r=1} \, \text{d}\theta \right] +
  \mathcal{O}(\sigma^2,\varepsilon^2),
\end{split}
\end{equation}
where $|\Omega|=\pi$, $h(\theta) = \cos(\mc{N} \theta)$, and
$u_0|_{r=1}$ is the leading-order solution $u_0$ evaluated on $r=1$.
In the Supplementary Material we show how to calculate $u_0$
  and $u_1$, which then yields $\overline{u}$ from
  \eqref{AveMFPT_Perturb}. This leads to the following main result:

\begin{Prop}\label{u_bar_prop_main}
    Consider a near-disk domain with boundary
    $r = 1 + \sigma \cos(\mc{N} \theta)$, with $\sigma \ll 1$, that
    has $m$ traps equally-spaced on a ring of radius $r_c$, centered
    at $\x_{j} = r_c e^{i\theta_j}$, where
    $\theta_j = {2\pi j/m} + \psi$ for $j=0,\ldots,m-1$. Then, if
    ${\mc{N}/m}\in \mathbb{Z}^{+}$, where $\mathbb{Z}^{+}$ is the set
    of positive integers, we have in terms of the ring radius $r_c$
    and the phase shift $\psi$ that the average MFPT satisfies
\begin{subequations}\label{u_bar_prop}
\begin{gather}
  \overline{u} \sim \overline{u}_0  + \sigma \overline{U}_1 + \ldots \,,
  \label{u_bar_propA} \\
  \overline{u}_0 = \frac{1}{2m\nu D} + \frac{\pi
    \kappa_1}{m D}, \;\; \overline{U}_1 = - \frac{r_c^{\mc{N}}}{\mc{N} D}
  \cos(\mc{N}\psi) \! \left( \frac{2 + (\mc{N} - 2)r_c^{2 m}}{1 - r_c^{2 m}} -
     \frac{\mc{N}}{2} (k-1)\!\right)\!,  \label{u_bar_propA_2} \\
  \mbox{and} \quad
  \kappa_1 = \frac{1}{2\pi}\left[ -\log(m r_c^{m-1}) - \log(1 - r_c^{2m}) + m
    r_c^2  - \frac{3}{4}m \right],  \label{u_bar_propA_3}
\end{gather}
where $k\equiv {\mc{N}/m}$ and $k\in \mathbb{Z}^{+}$. Alternatively,
if ${\mc{N}/n}\notin \mathbb{Z}^{+}$, then
$\overline{u}\sim \overline{u}_0 + {\mathcal O}(\sigma^2)$.
\end{subequations}
\end{Prop}

{This result shows that there are two distinct cases:
  ${\mc{N}/m}\in \mathbb{Z}^{+}$ and
  ${\mc{N}/m}\notin \mathbb{Z}^{+}$. In the latter case, the
  correction to the average MFPT at ${\mathcal O}(\sigma)$ vanishes,
  and a higher-order asymptotic theory would be needed to determine
  the correction term at ${\mathcal O}(\sigma^2)$. We do not pursue
  this here.}

{In the analysis below we will focus on the case where $\mc{N}=m$ and
will use our result in \eqref{u_bar_prop}} to optimize the average MFPT
with respect to the radius $r_c$ of the ring and the phase shift
$\psi$. We observe from \eqref{u_bar_propA_2} that $\overline{u}$ is
minimized when $\psi = 0$.  Therefore, the optimal traps on the ring
are on rays from the origin that coincide with the maxima of the
boundary perturbation given by
$\max(1 + \sigma \cos(\mc{N}\theta)) \equiv 1 + \sigma$.
To optimize $\overline{u}$ with respect to $r_c$, we write
$\overline{u}_0=\overline{u}_0(r_c)$ and
$\overline{U}_1=\overline{U}_1(r_c)$ and expand
\begin{equation}\label{r_expand}
r_{c\, \textrm{opt}}= r_{c_0} + \sigma\, r_{c_1} + \ldots\,.
\end{equation}
	Here $r_{c_0}$ is the leading-order optimal ring-radius obtained by
  setting $\overline{u}_0^{\prime}(r_{c})=0$ in \eqref{u_bar_propA_2}. In this
  way, for any $m\geq 2$, we obtain $r_{c_0}$ is the unique root on
  $0<r_{c_0}<1$ to
 \begin{align}\label{TransEqua_rc}
\frac{r_c^{2m}}{(1 - r_c^{2m})} = \frac{m-1}{2m} - r_c^2.
\end{align} 
	Numerical values for this root for various $m$ were given in
  the table in Figure~\ref{Ntraps_Ring_config}.

	Next, we substitute \eqref{r_expand} into the expansion in
  \eqref{u_bar_propA}, and collect terms in powers of $\sigma$. In
  this way, the optimal average MFPT is given by
\begin{equation}\label{ubar_opt}
\begin{split}
  \overline{u}_{\textrm{opt}} &\sim \overline{u}_0(r_{c_0})  + \sigma
  \overline{U}_1(r_{c_0}) + \ldots, 
\end{split}
\end{equation}
	where $\overline{u}_0$ and $\overline{U}_1$ are as defined in
\eqref{u_bar_propA_2}. Moreover,
by setting $\overline{u}^{\prime}(r_c)=0$ and expanding $r_c$ as in
\eqref{r_expand}, we obtain that
 $r_{c_1}=-{\overline{U}_1^{\prime}(r_{c_0})/\overline{u}_0^{\prime \prime}(r_{c_0})}$.
This yields that
\begin{align}\label{rc_Opt}
  r_{c_1} = \frac{1}{\pi} \frac{\chi^{\prime}(r_{c_0})}{\kappa_1^{\prime\prime}
  (r_{c_0})}\,; \quad \chi^{\prime}(r_{c_0}) = -
  \frac{m r_{c_0}^{m-1}}{(1 - r_{c_0}^{2m})^2}
  \Big{[} (m-2)r_{c_0}^{4m} + (4 -3m)r_{c_0}^{2m} -2  \Big{]},
\end{align}
	and $\kappa_1^{\prime\prime}(r_{c_0})$ is the second
  derivative of $\kappa_1(r_c)$ as defined in \eqref{u_bar_propA_3},
  evaluated at the leading-order optimal radius $r_{c_0}$.  Since
  $r_{c_0}$ is a minimum point of $\kappa_1(r_c)$, then
  $\kappa_1^{\prime\prime}(r_{c_0}) > 0$.  Also, it can easily be
  shown that $\chi^{\prime}(r_{c_0}) > 0$ for $0< r_{c_0}< 1$. Thus,
  $r_{c_1} > 0$, which implies that the centers of the traps bulge outwards
  towards the maxima of the domain boundary perturbation.  This result
  is summarized as follows:

\begin{Prop}\label{rc_prop_main}
  In the near disk case with boundary
  $r = 1 + \sigma \cos(\mc{N} \theta)$ and $\sigma \ll 1$, and for a
  ring pattern with $m=\mc{N}$ traps equally spaced on a ring of
  radius $r_c$, the optimal radius $r_{c \, \textrm{opt}}$ of the ring
  is given by
\begin{subequations}\label{rc_prop}
\begin{gather}
  r_{c\, \textrm{opt}} \sim r_{c_0} + \,\frac{\sigma}{\pi}
  \frac{\chi^{\prime}(r_{c_0})}{\kappa_1^{\prime\prime}(r_{c_0})} +
  \ldots\,, \label{rc_propA}\\
  \text{where} \qquad \kappa_1^{\prime\prime}(r_{c_0}) = \frac{m}{\pi
    r_{c_0}^2} \left[\frac{(m-1)}{2m} + r_{c_0}^2 +
    \frac{r_{c_0}^{2m}}{(1 - r_{c_0}^{2m})^2} \Big{(} 2m-1 +
    r_{c_0}^{2m} \Big{)} \right]\,. \label{rc_propB}
\end{gather}
\end{subequations}
	Here $\chi^{\prime}(r_{c_0})$ is given in \eqref{rc_Opt} in
  terms of the unique solution $r_{c_0}$ to \eqref{TransEqua_rc}.
\end{Prop}

We first apply our results to an ellipse of area $\pi$ that contains
two circular traps each of radius $\varepsilon=0.05$ centered on the
major axis. This corresponds to the early stage of deformation of the
unit disk in the optimal MFPT problem studied in
\S~\ref{Static_2Trap_Ellipse} (see Figure \ref{Ellpt_MFPT_Opt}). The
boundary of the ellipse is parameterized for $\sigma\ll 1$ by
$(x,y)=(a \cos(\theta)\,, b \sin(\theta))$, for $0\leq \theta<2\pi$,
where $a = 1 + \sigma$ and $b=1/(1+\sigma)$ are the semi-axes
chosen so that $ab=1$ for any $\sigma>0$. For $\sigma\ll 1$, we
readily calculate that the domain boundary in polar coordinates is
$r=1+\sigma \cos(2\theta) + {\mathcal O}(\sigma^2)$.

Upon setting $m=2$ and $\mc{N}=2$ in \eqref{rc_prop}, and then using
$\sigma=(b^{-1}-1)$ as $b\to 1^{-}$, we obtain that the optimal ring radius
satisfies 
\begin{subequations}\label{rc_Opt_m2}
\begin{align}
  r_{c\, \textrm{opt}} \sim r_{c_0} +  \frac{1}{\pi} \left(  \frac{1}{b} - 1
  \right)\, \frac{\chi^{\prime}(r_{c_0})}{\kappa_1^{\prime\prime}(r_{c_0})}, 
 \label{rc_prop_m2A}
\end{align}
where $r_{c_0}\approx 0.4536$ is the unique root of
\eqref{TransEqua_rc} when $m=2$. Here, from
\eqref{rc_propB} and \eqref{rc_Opt} with $m=2$, we have that
\begin{equation}
  \chi^{\prime}(r_{c_0})  = \frac{4 r_{c_0}(r_{c_0}^{4} + 1 )}{(1 - r_{c_0}^{4})^2}\,,
   \quad \mbox{and} \quad \kappa_1^{\prime\prime}(r_{c_0})  =
  \frac{2}{\pi\, r_{c_0}^2} \left[\frac{1}{4} + r_{c_0}^2 +
 \frac{r_{c_0}^{4} (3 + r_{c_0}^{4}  )}{(1 - r_{c_0}^{4})^2}   \right].
                           \label{rc_prop_m2B}
\end{equation}
\end{subequations}
By setting $r_{c_0} = 0.4536$ in \eqref{rc_Opt_m2}, \eqref{ubar_opt}, and
\eqref{u_bar_prop} we obtain for a trap radius of $\varepsilon=0.05$
that the optimal ring radius and the optimal average MFPT are 
\begin{align}\label{rc_mfpt_optimal}
  r_{c\,\textrm{opt}}(b) \sim 0.4536 +  \left(  \tfrac{1}{b} - 1  \right) 0.3559,
  \quad
  \overline{u}_{\textrm{opt}} \sim \frac{1}{D} \Big[
  0.5120 - \left(  \tfrac{1}{b} - 1  \right) 0.2149 \Big],
\end{align}
as $b\to 1^{-}$. This perturbation result characterizes the optimal
trap locations and optimal average MFPT for a slight elliptical
perturbation of the unit disk. 

\begin{figure}
    \centering
    \begin{subfigure}[b]{0.44\textwidth}
        \includegraphics[width=\textwidth]{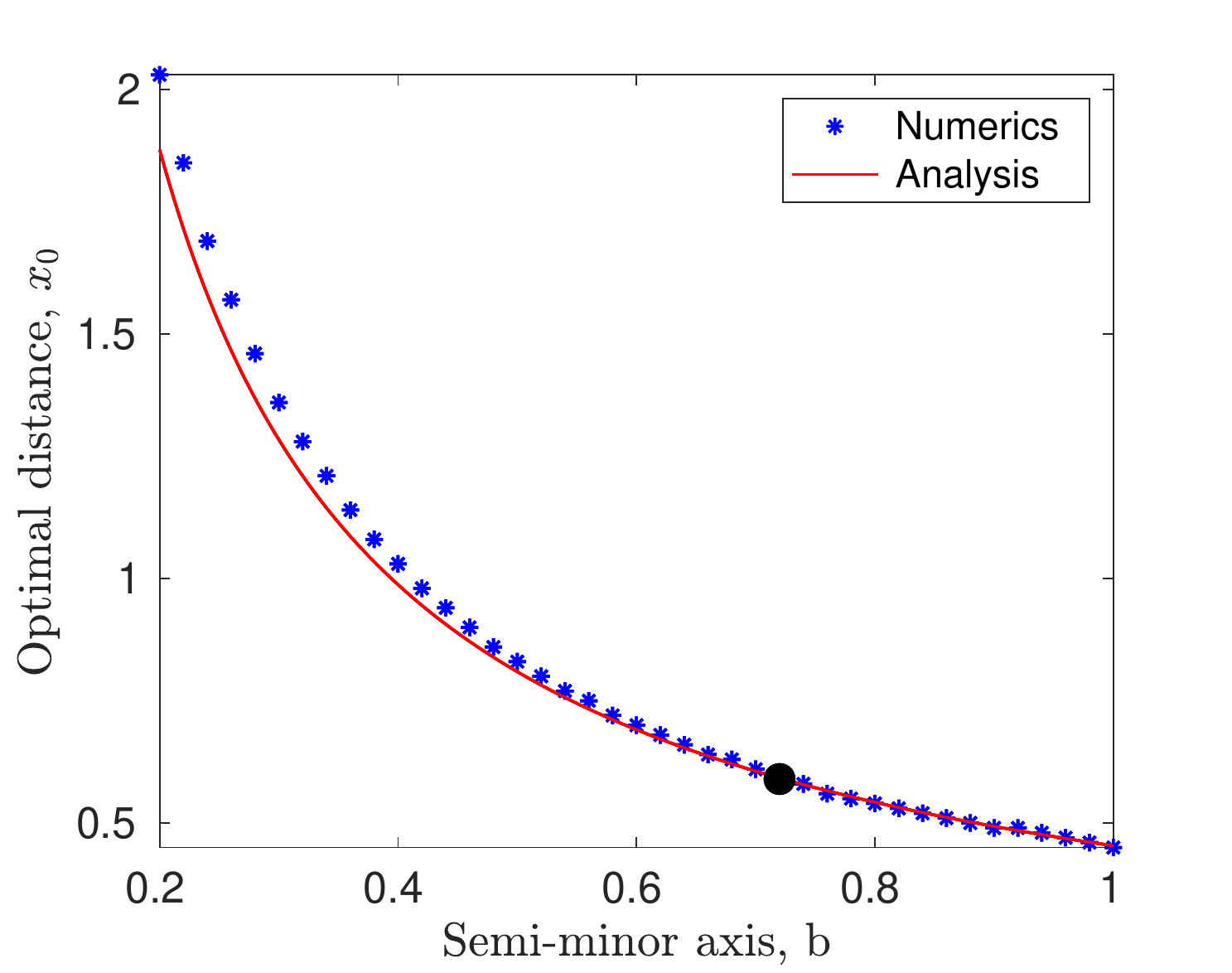}
        \caption{Optimal location of traps}
        \label{Ellpt_Opt_X0_Analysis}
    \end{subfigure} \qquad 
    \begin{subfigure}[b]{0.46\textwidth}
        \includegraphics[width=\textwidth]{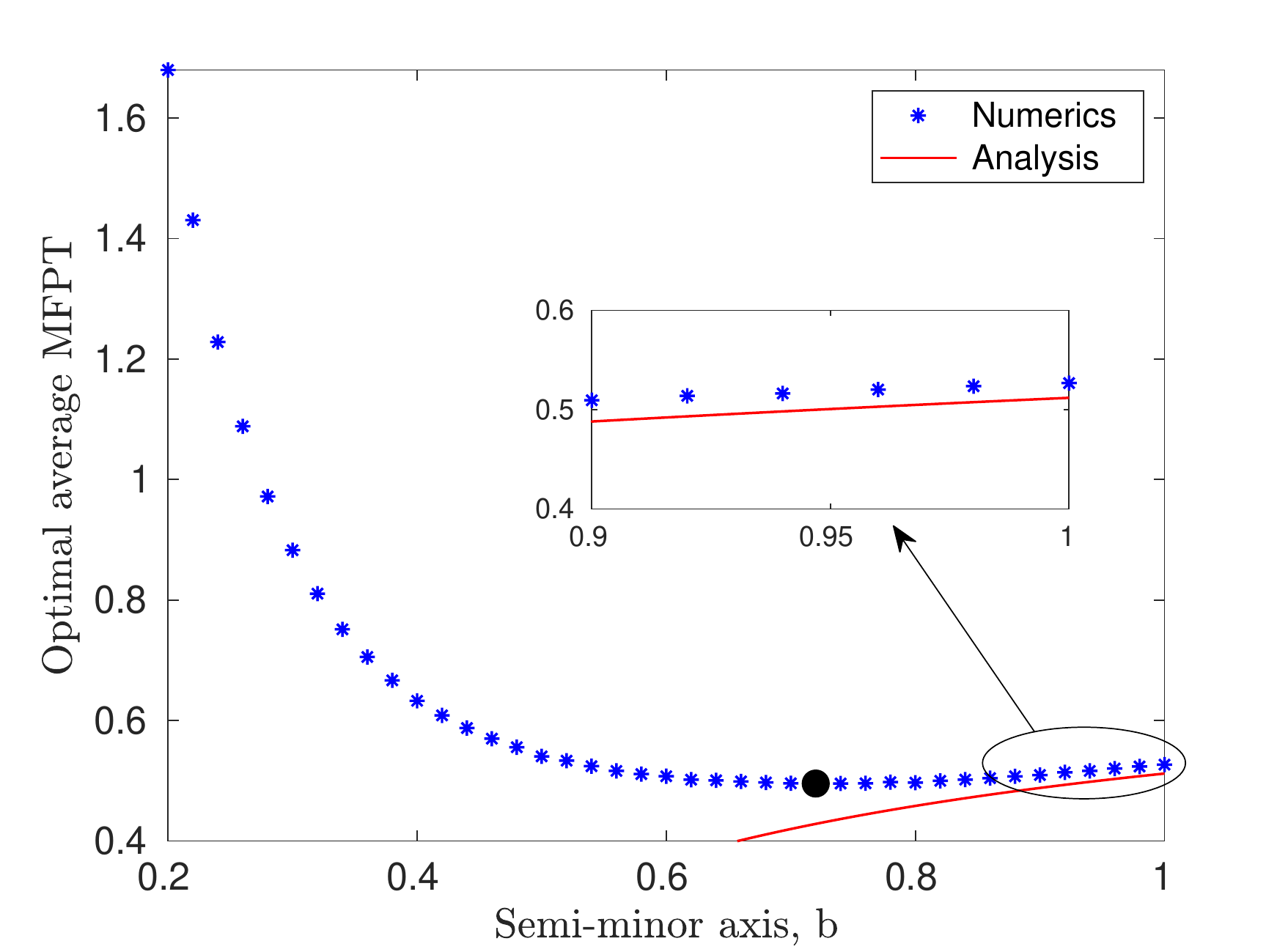} 
        \caption{Optimal average MFPT }
        \label{Ellpt_Opt_MFPT_Analysis}
    \end{subfigure}
    \vspace*{-3ex}
    \caption{Two traps in an ellipse: a comparison of the perturbation results in
        \eqref{rc_mfpt_optimal} (thin lines) with the full numerical
        results (asterisks) of Figure~\ref{Ellpt_MFPT_Opt} for the deforming elliptical region
        containing two traps of radius $\varepsilon=0.05$.
        The asymptotic theory is valid for semi-minor axis $b\to 1^-$
        (early stages of disk deformation).
        (a) optimal distance of the traps from the
        center of the ellipse versus $b$. (b) optimal average MFPT
        versus $b$.  The dot is the globally optimal
        average MFPT found earlier in Figure~\ref{Ellpt_MFPT_Opt}.}
    \label{Ellpt_MFPT_Opt_Analysis}
\end{figure}

For $D=1$, Figures~\ref{Ellpt_Opt_X0_Analysis} and
\ref{Ellpt_Opt_MFPT_Analysis} show a comparison of our analytical
results \eqref{rc_mfpt_optimal} for the optimal location of the traps
and the optimal average MFPT with the corresponding full numerical
results computed using the CPM in Figure~\ref{Ellpt_MFPT_Opt}.
Although our analysis is only
valid for $b \to 1^-$, Figure~\ref{Ellpt_Opt_X0_Analysis} shows that
our perturbation result for the optimal trap locations agree closely
with the numerical result even for moderately small values of
$b$. However, this is not the case for the optimal average MFPT, where
the perturbation result deviates rather quickly from the numerical
result as $b$ decreases. The key qualitative conclusion from the
analysis is that the optimal average MFPT decreases as $b$ decreases
below $b=1$. This establishes that, for the class of elliptical
domains with fixed area $\pi$, the optimal average MFPT is minimized
not for the unit disk, but for a particular ellipse.

Next, we apply our theory to the cases $m=\mc{N}=3$ and $m=\mc{N}=4$,
which were studied numerically in Figure~\ref{fig:3and4star} when
$\sigma=0.2$. For traps of radii $\varepsilon=0.05$ and $D=1$, we
obtain from \eqref{rc_prop} and \eqref{ubar_opt} that when
$\sigma\ll 1$ the optimal ring radius and optimal average MFPT are
\begin{align}
  r_{c,\textrm{opt}} \sim 0.5517  + 0.2664 \, \sigma \,, \qquad
  \overline{u}_{\textrm{opt}} \sim 0.2964 - 0.1168\, \sigma \,; \qquad
  m=\mc{N}=3 \,, \\
 r_{c,\textrm{opt}} \sim 0.5985  + 0.1985 \, \sigma \,, \qquad
  \overline{u}_{\textrm{opt}} \sim 0.1998 - 0.0663\, \sigma \,; \qquad
  m=\mc{N}=4 \,.
\end{align}
For $\sigma=0.2$, this yields that $r_{c,\textrm{opt}} \approx 0.6049$
when $m=\mc{N}=3$ and $r_{c,\textrm{opt}} \approx 0.6382$ when
$m=\mc{N}=4$. Although $\sigma=0.2$ is not very small, the asymptotic
results still provide a rather decent approximation to the numerical
results for the optimal trap locations shown in Figure~\ref{fig:3and4star}.

\subsection{Asymptotics for high-eccentricity elliptical domains}
\label{sec:skinnyellipse}

In this subsection we provide two different approximation schemes for
estimating the optimal average MFPT for an elliptical domain of
high-eccentricity that contains either two or three traps centered
along the semi-major axis.

\subsubsection{Approximation by thin rectangular domains}\label{sec:thin}

We consider a Brownian particle in a thin elliptical domain of area
$\pi$ with semi-major axis $a$ and semi-minor axis $b$, that contains
two circular absorbing traps each of radius $\varepsilon$ on its major
axis (see Figure \ref{Ellpt_MFPT_Opt}) for $b \ll 1$.  In order to
estimate the MFPT for this particle, the elliptical region is replaced
with a thin rectangular region defined by
$[-a_0, a_0] \times [-b_0, b_0]$
satisfying $(a_0/b_0) \gg 1$. Moreover, the circular traps in the
ellipse are replaced with thin vertical trap strips of width
$2 \varepsilon_0$ centered at $(-x_0,0)$ and $(x_0,0)$, namely
$\Omega_1 = \Phi_1 \times [-b_0, b_0]$ and
$\Omega_2 = \Phi_2 \times [-b_0, b_0]$
where
$\Phi_1 = [-x_0 - \varepsilon_0 \leq x \leq -x_0 + \varepsilon_0]$ and
$\Phi_2 = [x_0 - \varepsilon_0 \leq x \leq x_0 + \varepsilon_0]$.
The MFPT in this rectangular domain satisfies
\begin{equation}\label{SkinnyEllp_2Traps_Prob}
\begin{aligned}
  \nabla^2 u & = -1/D \,, \quad \mbox{in} \quad
               \x \in [-a_0, a_0] \times [-b_0, b_0]
               \setminus \lbrace{\Omega_1,\,\Omega_2\rbrace} \,,\\
  \partial_x u  &= 0\,, \quad \mbox{on}   \quad x=\pm a_0 \quad \mbox{for}
                                      \quad |y|\leq b_0 \,, \\
  \partial_y u  &= 0\,, \quad \mbox{on}  \quad y=\pm b_0 \quad \mbox{for}
                  \quad x \in [-a_0,a_0] \setminus
                  \lbrace{\Phi_1,\,\Phi_2\rbrace} \,, \\
  u &= 0\,, \quad \text{for $x \in \Omega_1 \cup \Omega_2$.}
\end{aligned}
\end{equation}
To ensure that the area of the rectangular
region is $\pi$ and that the rectangular traps have the same area as
the circular traps in the elliptical region, we impose that
\begin{align}\label{SkinnyEllp_Cond_for_Area}
    4a_0\,b_0  = \pi  \quad \mbox{and} \quad
   4 \varepsilon_0 \,b_0 = \pi \,\varepsilon^2 \,.
\end{align}
The PDE \eqref{SkinnyEllp_2Traps_Prob} has a 1-D solution that is even
in $x$, namely
   $u_1(x) \equiv \frac{1}{2D}\big((x_0 - \varepsilon)^2 - x^2\big)$
   for $0 \leq x \leq x_0 - \varepsilon$,
   and
   $u_2(x) \equiv  \frac{1}{2D} \big{[} x(2a_0 - x) +
         (x_0 + \varepsilon_0) ( x_0 + \varepsilon_0 - 2a_0) \big{]}$
   for $x_0 + \varepsilon \leq x \leq a_0$.
Then, we calculate $I_1=\int_{0}^{x_0-\varepsilon} u_1 \, \text{d}x$ and
$I_2=\int_{x_0+\varepsilon}^{a_0} u_2 \, \text{d}x$, and observe that the
average MFPT is given by
$\overline{u} = 4b_0 (I_1+I_2) / \!\left(\pi (1-2\varepsilon^2)\right)$.
We get
\begin{equation}\label{SkinnyEllp_AveMFPT}
\overline{u} = \frac{4\,b_0}{D \pi (1 - 2\varepsilon^2)} \Big{[}  {\left(a_{0} - 2 \,
\varepsilon_{0}\right)} x_{0}^{2} - {\left(a_{0}^{2} - 2 \, a_{0}
\varepsilon_{0}\right)} x_{0} + \frac{1}{3} \, a_{0}^{3} - a_{0}^{2}
\varepsilon_{0} + a_{0}\varepsilon_{0}^{2} - \frac{2}{3} \, \varepsilon_{0}^{3} 
 \Big{]}.
\end{equation}
The optimal locations of the traps are found
by minimizing $\overline{u}$ with respect to $x_0$. This yields
\begin{equation}\label{SkinnyEllp_OptLoc}
  x_{0\, {\textrm{opt}} } = \frac{a_0}{2} =  \frac{\pi}{8b_0}\,, \quad \mbox{and}
  \quad \overline{u}_{\textrm{opt}} = \frac{\pi^2}{192\, D \,b_{0}^{2}}
  \Big{(} 1 - 4 \, \varepsilon^{2} + \mathcal{O}(\varepsilon^4) \Big{)}\,.
\end{equation}
Here we used $a_0 = \pi/(4\,b_0)$ and
$\varepsilon_0={\pi\varepsilon^2/(4b_0)}$ as given in
\eqref{SkinnyEllp_Cond_for_Area}.

As one would expect, the optimal location in \eqref{SkinnyEllp_OptLoc}
is the point at which the area of the half-rectangle
$[0, a_0] \times [-b_0, b_0]$ is divided into two equal pieces. This
equal area rule will minimize the capture time of the Brownian
particle in the half-rectangle.

Next, we relate this optimal
MFPT in the thin rectangular domain to that in the thin elliptical
domain.  One possibility is to set $a_0 = a$, so that the length of the
rectangular domain and the ellipse along the major axis are the
same. From the equal area condition \eqref{SkinnyEllp_Cond_for_Area},
we obtain $b_0 = (\pi b)/4$, where $b$ is the semi-minor axis of the
ellipse. For this choice \eqref{SkinnyEllp_OptLoc} becomes
\begin{equation}\label{SkinnyEllp_CaseI}
  x_{0\, {\textrm{opt}} } =   \frac{1}{2b} \quad \mbox{and} \quad
  \overline{u}_{\textrm{opt}} \approx  \frac{1}{12\, D \,b^{2}}  \Big{(} 1 - 4 \,
  \varepsilon^{2} + \mathcal{O}(\varepsilon^4) \Big{)}\,; \qquad
  \mbox{Case I:} \,\,\, (a=a_0) \,.
\end{equation}
A second possibility is to choose $b_0 = b$, so that the width of
the thin rectangle and ellipse are the same. From
\eqref{SkinnyEllp_OptLoc} this yields that
\begin{equation}\label{SkinnyEllp_CaseII}
  x_{0\, {\textrm{opt}} } =     \frac{\pi}{8b} \quad \mbox{and} \quad
  \overline{u}_{\textrm{opt}} \approx   \frac{\pi^2}{192\, D \,b^{2}}
  \Big{(} 1 - 4 \, \varepsilon^{2} + \mathcal{O}(\varepsilon^4) \Big{)}\,;
  \qquad   \mbox{Case II:} \,\,\, (b=b_0) \,.
\end{equation}
Both estimates \eqref{SkinnyEllp_CaseI} and \eqref{SkinnyEllp_CaseII}
are applicable only when $b\ll 1$. Together they suggest that the
optimal locations of the traps and the optimal average MFPT for the
thin ellipse satisfy the scaling laws
$x_{0\, {\textrm{opt}} }=\mathcal{O}(b^{-1})$ and
$\overline{u}_{\textrm{opt}}=\mathcal{O}(b^{-2})$, respectively.

\begin{figure}
    \centering
    \begin{subfigure}[b]{0.42\textwidth}
        \includegraphics[width=\textwidth]{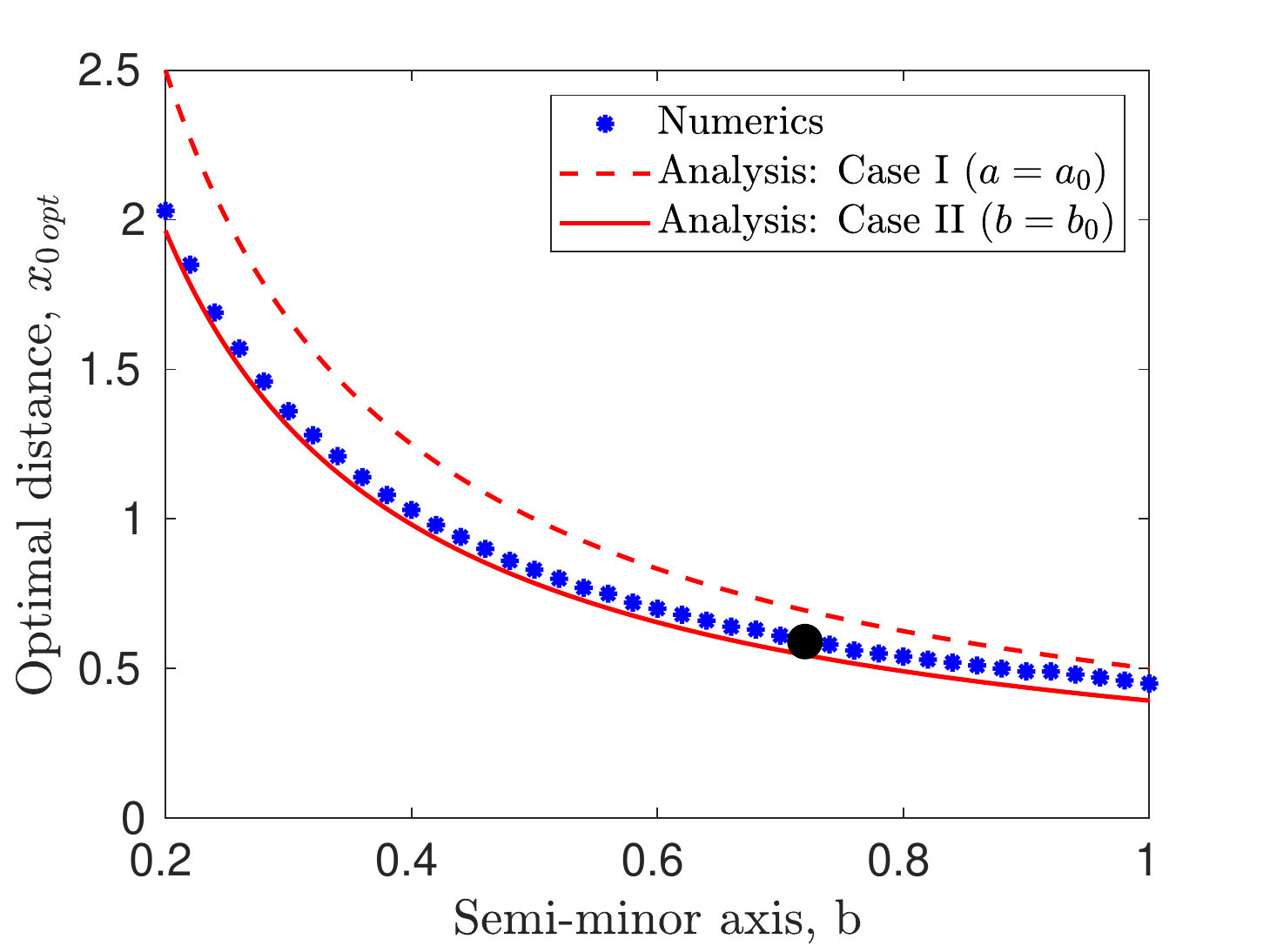}
        \caption{Optimal location of traps}
        \label{SkinnyEllp_Opt_X0_Analysis}
    \end{subfigure} \qquad 
    \begin{subfigure}[b]{0.42\textwidth}
        \includegraphics[width=\textwidth]{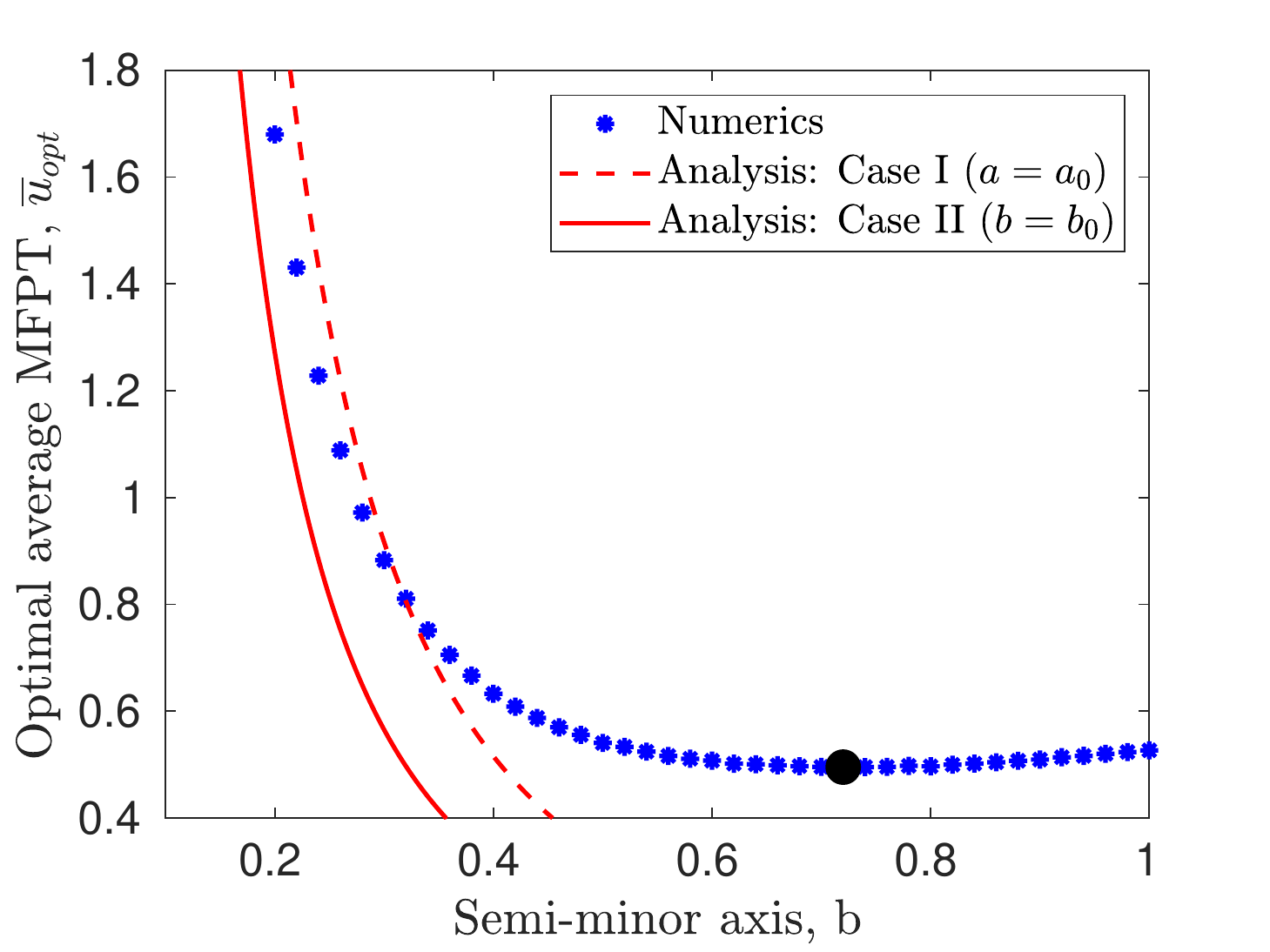}
        \caption{Optimal average MFPT }
        \label{SkinnyEllp_Opt_MFPT_Analysis}
    \end{subfigure}
    \vspace*{-3ex}
    \caption{Two traps in an ellipse:
      the thin-rectangle approximations (valid for small $b$) of
      \eqref{SkinnyEllp_CaseI} (dashed lines) and
      \eqref{SkinnyEllp_CaseII} (solid lines)
      are compared with the full numerical results (asterisks)
      of Figure~\ref{Ellpt_MFPT_Opt},
      for the optimal trap locations (a) and optimal average MFPT (b).
      The dot is the globally optimal average MFPT found earlier.}
    \label{SkinnyEllp_MFPT_Opt_Analysis}
\end{figure}

Figure~\ref{SkinnyEllp_MFPT_Opt_Analysis} compares the full numerical
results for the optimal trap locations and optimal average MFPT of
Figure~\ref{Ellpt_MFPT_Opt} with the analytical results given in
\eqref{SkinnyEllp_CaseI} and \eqref{SkinnyEllp_CaseII} with $D = 1$.
We observe that the
two simple analytical results provide relatively decent approximations
to the full numerical results for small $b$.
More specifically, we observe
that the two limiting
approximations \eqref{SkinnyEllp_CaseI} and \eqref{SkinnyEllp_CaseII}
provide upper and lower bounds for the full numerical results,
respectively.  When $a_0 = a$, \eqref{SkinnyEllp_CaseI}
is seen to overestimate both the optimal location of the trap and the optimal
average MFPT, when $b\ll 1$.  This is because when $a_0 = a$, the
equivalent rectangular region is thinner than the elliptical region
near the center of the region. As a result, the optimal location of
the traps for the elliptical region are closer to the center of the
domain than for the rectangular region. This effect will overestimate
the optimal average MFPT.  Alternatively, when $b_0 = b$,
\eqref{SkinnyEllp_CaseII} is seen to underestimate both the
optimal location of the traps and the optimal average MFPT, when
$b\ll 1$.  For this choice, the length of the equivalent rectangular
region on the horizontal axis is shorter than the length of the major
axis of the elliptical region.  Because the optimal location
of the trap when $b\ll 1$ depends mostly on the horizontal axis, and
the rectangular region is shorter than the elliptical region, the
results given by \eqref{SkinnyEllp_CaseII} will be underestimates.

\subsubsection{A perturbation approach for long thin domains}\label{sec:long_thin}

Next, we develop a more refined asymptotic approach, which
incorporates the shape of the domain boundary, to estimate the
optimal average MFPT in a thin ellipse that contains three circular
traps of radius $\varepsilon$. One trap is at the center of the
ellipse while the other two are centered on the major axis symmetric
about the origin. Recall that a pattern of three colinear traps was
shown in Figure~\ref{fig:three_traps_ellipse_pso} of
\S~\ref{Static_3Trap_Ellipse} to provide a global minimum of the
average MFPT in a thin ellipse.  Our goal here is to approximate the
optimal trap locations and corresponding MFPT for this pattern.

Although our theory is developed for a class of long thin
domains, we will apply it only to an elliptical domain.  For
$\delta\ll 1$, we consider the family of domains
\begin{equation}
\Omega = \{ (x,y) \mid -1/\delta < x < 1/\delta\,, -\delta F(\delta x)
< y < \delta F(\delta x) \}\,.
\end{equation}
We assume that the boundary profile $F(X)$ satisfies $F(X) > 0$ on
$|X|<1$, with $F(\pm 1)=0$. We label $\Omega_a$ as the union of the
traps that are located at $\{(0,0), (\pm x_0,0)\}$. The MFPT problem is to solve
\begin{equation}
  \partial_{xx} u + \partial_{yy} u = -1 / D\,, \quad \mbox{in}
  \quad \Omega\setminus\Omega_a\,; \quad \partial_n u = 0\,, \,\,\
  \mbox{on} \,\,\, \partial\Omega\,;
  \quad u = 0 \,,\,\,\, \mbox{on} \,\,\, \partial\Omega_a\,.
\label{eqn:three_traps_mfpt}
\end{equation}
Using a perturbation analysis, valid for long thin domains with
$\delta\ll 1$, in \S~\ref{supp:long_thin} of the Supplementary Material
we show that
$u(x,y)\sim \delta^{-2}U_0(\delta x) + {\mathcal O}(\delta^{-1})$,
where $U_0(X)$, with $x={X/\delta}$ and $d={x_0/\delta}$, satisfies the following
multi-point boundary value problem (BVP) on $|X|<1$:
\begin{equation}\label{sec:long_u0}
  \left[F(X)U_0^{\prime} \right]^{\prime} = -F(X) / D\,, \,\,\,
  \mbox{on} \,\,\, (-1,1)\setminus\{0,\pm d\}\,; \quad U_0 = 0\,\,\,\,
  \mbox{at} \,\,\, X = 0, \pm d\,,
\end{equation}
with $U_0$ and $U_0^{\prime}$ bounded as $X\to \pm 1$, where
$F(\pm 1)=0$. Observe in this formulation that the traps are replaced
by zero point constraints for $U_0$.

Although the solution to \eqref{sec:long_u0} can be reduced to
quadrature for an arbitrary $F(X)$, we will find an explicit solution
for the case of a thin elliptical domain of area $\pi$ with boundary
$\frac{x^2}{a^2} + \frac{y^2}{b^2} = 1$, where $a = 1/\delta \gg 1$
and $b = \delta\ll 1$. For this case, $F(X) = \sqrt{1-X^2}$ and we readily
obtain, after performing some quadratures, that
\begin{subequations}
\begin{equation}\label{thin:u0_solve_1}
U_0(X) = \begin{cases}
  -\frac{1}{4D} \left[ (\sin^{-1}{X})^2 + X^2 + \pi
    \sin^{-1} X + c_2 \right]\,, \quad -1\leq X \leq -d\,, \\
  -\frac{1}{4D} \left[ (\sin^{-1}{X})^2 + X^2 + c_1
    \sin^{-1} X \right]\,, \quad -d\leq X \leq 0 \,,\\
    U_0(X) = U_0(-X)\,, \quad 0 \leq X \leq 1\,,
\end{cases}
\end{equation}
where $c_1$ and $c_2$ are given by
\begin{equation}\label{thin:u0_solve_2}
  c_2 = \pi \sin^{-1}{d} - d^2 - \left( \sin^{-1}{d}\right)^2 \,, \qquad
  c_1 = \frac{ d^2 + (\sin^{-1}{d})^2}{\sin^{-1}{d}}\,.
\end{equation}
\end{subequations}

In terms of $U_0(X)$, the average MFPT for \eqref{eqn:three_traps_mfpt}
is estimated for $\delta\ll 1$ by
\begin{equation}\label{thin:ave}
  \overline{u} \sim \frac{1}{\pi} \int_{-1/\delta}^{1/\delta}
  \int_{-\delta F(\delta x)}^{\delta F(\delta x)}  u \, \text{d}x \text{d}y \sim
  \frac{4}{\pi \delta^2} \int_{-1}^{0} F(X) U_0(X) \, \text{d}X \,.
\end{equation}
For the ellipse, where $F(X)=\sqrt{1-X^2}$, we set
\eqref{thin:u0_solve_1} in \eqref{thin:ave} and integrate to get
\begin{subequations}
\begin{equation}\label{thin:ave_ell}
  \overline{u} \sim \frac{1}{\pi D\delta^2} \left( {\mathcal H}(d)
    -\int_{-1}^{0}\!\sqrt{1-X^2} \left[ \left( \sin^{-1}{X}\right)^2 + X^2 +
    \pi\sin^{-1}{X} \right]
    \text{d}X \! \right).
\end{equation}
Here ${\mathcal H}(d)$ is defined in terms of $c_1$ and $c_2$, as
given in \eqref{thin:u0_solve_2}, by
\begin{equation}\label{thin_ave_H}
  {\mathcal H}(d) \equiv \frac{c_2}{2} \!\left[ d\sqrt{1-d^2} + \sin^{-1}{d}
    \right] - \frac{c_2\pi}{4}  + (\pi-c_1)
  \int_{-d}^{0}\!\!\left(\sin^{-1}{X}\right) \! \sqrt{1-X^2}\, \text{d}X \,.
\end{equation}
\end{subequations}

To estimate the optimal average MFPT we minimize ${\mathcal H}(d)$ in
\eqref{thin_ave_H} on $0<d<1$. We compute that
$d_{\textrm{opt}}\approx 0.5666$. Then, by evaluating
${\mathcal H}(d_{\textrm{opt}})$, \eqref{thin:ave_ell} determines the
optimal value of $\overline{u}$. In terms of the original $x$
variable, and recalling $b=\delta$, we have for the thin ellipse that
the optimal trap location and optimal average MFPT satisfy
\begin{equation}\label{3trap:skin}
  x_{0 \textrm{opt}}\sim {0.5666/b} \,, \qquad \overline{u}_{\textrm{opt}} \sim
  {0.0308/( b^2 D)} \,, \qquad \mbox{for} \quad b\ll 1\,.
\end{equation}
In Figure~\ref{3SkinnyEllp_MFPT_Opt_Analysis} we show favorable comparisons
between these thin domain asymptotic results in \eqref{3trap:skin}
and the full numerical results computed using the CPM,
for the optimal trap locations and optimal average MFPT.
We also show upper and lower bounds derived using approximation via thin
rectangular domains, similar to \S~\ref{sec:thin}.
These bounds are given by \eqref{3SkinnyEllp_CaseI}
and~\eqref{3SkinnyEllp_CaseII} of \S~\ref{thin:rec} of the Supplementary Material.
We note that the thin domain
asymptotic results \eqref{3trap:skin} provide a closer agreement
with the full numerical results than do the bounds based on rectangles.

\begin{figure}
  \centering
    \makebox{%
      \raisebox{0.5ex}{\small{(a)}}
      \includegraphics[width=.39\textwidth]{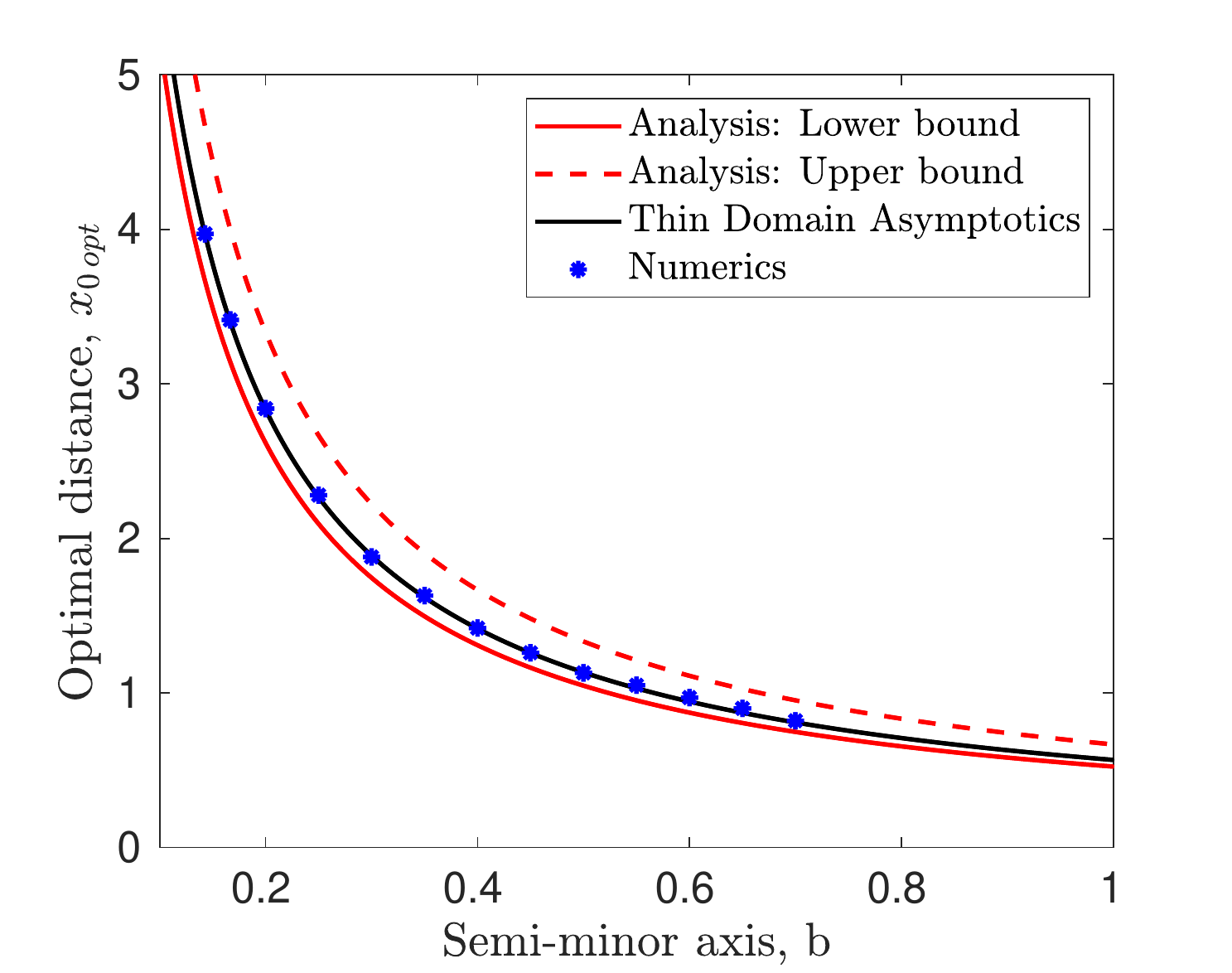}
      \phantomsubcaption
      \label{SkinnyEllp_3Opt_X0_Analysis}
    } \quad
    \makebox{%
      \raisebox{0.5ex}{\small{(b)}}
      \includegraphics[width=.39\textwidth]{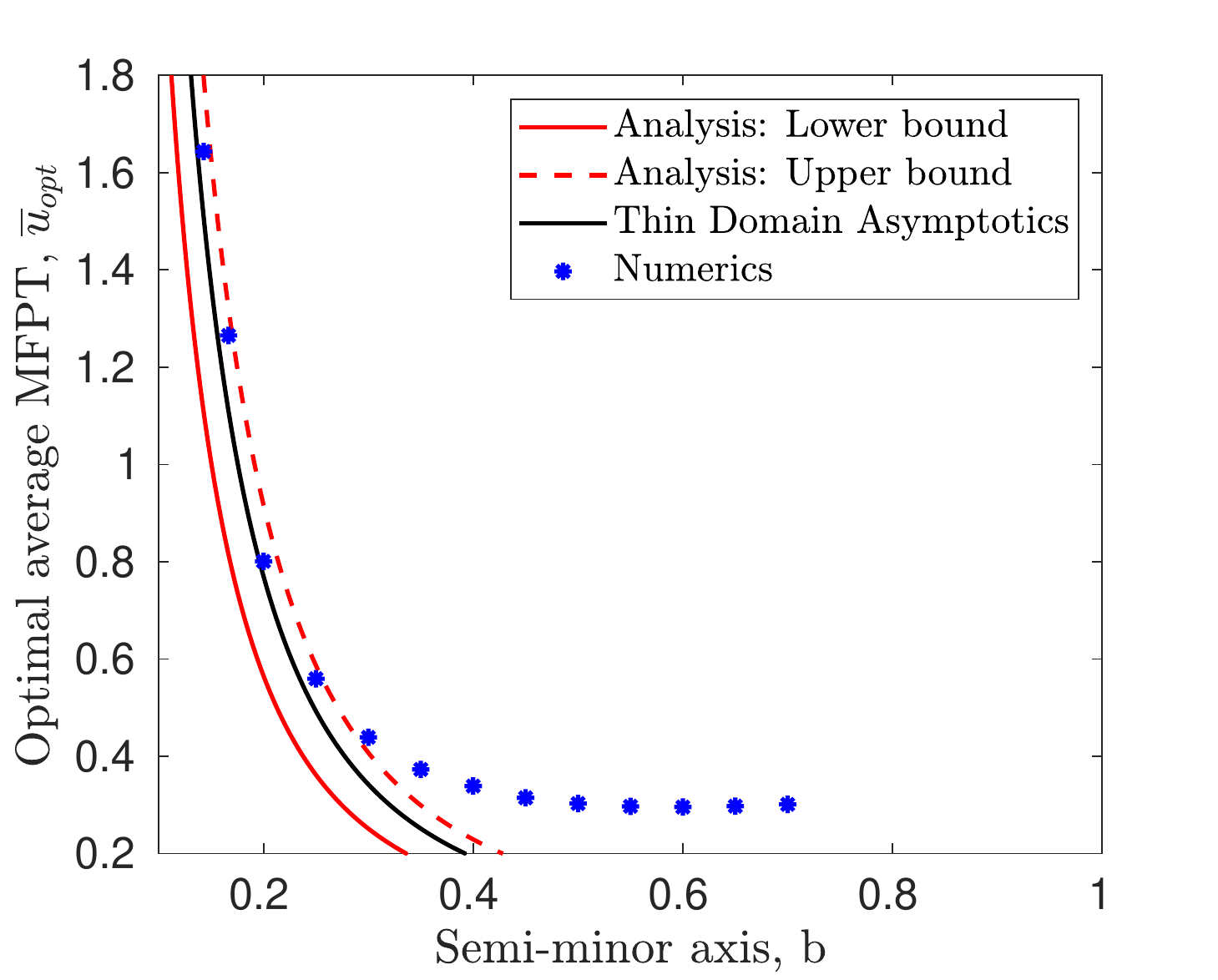}
      \phantomsubcaption
      \label{SkinnyEllp_3Opt_MFPT_Analysis}
    }
    \vspace*{-1ex}
    \caption{Three traps in an ellipse:
      optimal trap location (a) and optimal average MFPT (b)
        for a thin elliptical domain of area $\pi$ and semi-minor axis
        $b \ll 1$ that contains a trap centered at the origin and
        additional traps on either side of the origin at a distance
        $x_0$ from the center. The three traps are circular of radius
        $\varepsilon=0.05$.  The thin domain asymptotic results in
        \eqref{3trap:skin} (solid dark lines) are compared with full
        numerical results (asterisks) and the upper (red dashed
        lines) and lower (red solid lines) bounds based on
        thin-rectangle approximation.
      }
  \label{3SkinnyEllp_MFPT_Opt_Analysis}
\end{figure}

\subsection{Asymptotics of a rapidly rotating trap}
\label{sec:fastrottrap}

In the unit disk, we analyze the two-trap problem of
\S~\ref{sec:TwoTrapsDisk} in the limit where the moving trap on the
ring rotates about the center of the disk at an angular frequency
$\omega \gg \mathcal{O}(\eta^{-1})$, where $\eta\ll 1$ is the radius
of the moving trap. The fixed trap at the center of the disk is chosen
to have a possibly different radius $\varepsilon\ll 1$. In the high
frequency limit $\omega\gg 1$, the fast moving trap creates an
absorbing band along its entire path as shown in
Figure~\ref{Opt_r0_TwoTraps}. For $\omega\gg 1$, we will calculate
asymptotically the optimal radius of rotation of the moving trap in
terms of $\eta$ and $\varepsilon$.

\begin{figure}
  \centering
    \makebox{
      \raisebox{2ex}{\small{(a)}}
      \raisebox{2.8ex}{%
        \includegraphics[width=0.31\textwidth]{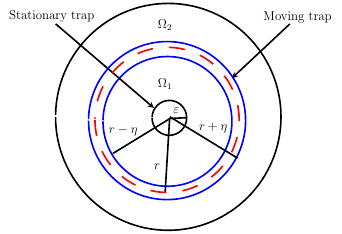}}
      \phantomsubcaption
      \label{Opt_r0_TwoTraps:a}
    }\qquad
    \makebox{
      \raisebox{2ex}{\small{(b)}}
      \includegraphics[width=0.31\textwidth]{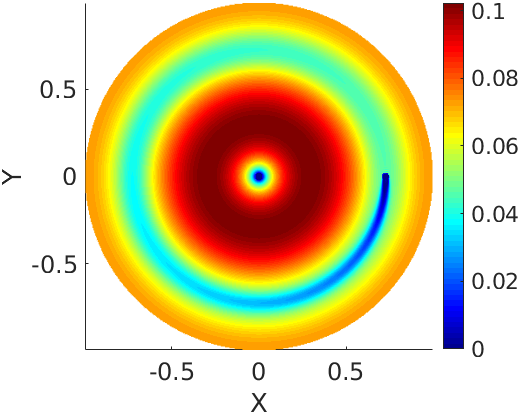}
    }
  \vspace*{-2ex}
  \caption{Optimizing the radius of rotation for a fast
      rotating trap in the unit disk that has a stationary trap at its
      center.  Left: schematic plot showing the two absorbing traps in the
      disk.  Right: MFPT for a Brownian particle with trap radii
      $\varepsilon = \eta = 0.02$. The moving trap rotates at an
      angular frequency of $\omega = 2000$ on a ring of radius
      $r = 0.727$. Computed using the CPM with mesh
      size $\Delta x = 0.005$.}
  \label{Opt_r0_TwoTraps}
\end{figure}

We formulate the $\omega\to\infty$ limiting problem as a stationary
trap problem, where the absorbing band created by the rotating trap is
used to partition the unit disk into two regions, as shown
in Figure~\ref{Opt_r0_TwoTraps}.
In the high-frequency limit $\omega\gg 1$, the limiting problem for the MFPT is to
solve the multi-point BVP
\begin{equation}\label{fast_mfpt}
\begin{split}
  &u_{\rho\rho} + \rho^{-1}u_{\rho} = -{1/D} \,, \quad \mbox{in} \quad
  \varepsilon \le \rho \le
  r-\eta\,, \quad \mbox{and} \quad r+\eta\le\rho <1 \,,\\
  & u =0 \quad \mbox{on} \quad \rho = \varepsilon\,,\,\, \rho= r - \eta
  \,,\,\, \rho = r + \eta\,; \quad \partial_\rho u = 0 \quad \mbox{on}
  \quad \rho=1\,,
\end{split}
\end{equation}
for $u\equiv u(\rho)$. Here, we have imposed zero-Dirichlet boundary
conditions on the inner and outer edges of the absorbing band created
by the fast moving trap.

As detailed in \S~\ref{supp:anal_3} of the Supplementary Material, we
first solve (\ref{fast_mfpt}) for $u$, and then calculate the average
MFPT $U(r)$ over the unit disk. This yields that
\begin{equation}\label{Analy_Int}
  U(r) = \frac{C}{\log\!\left(\frac{\varepsilon}{\alpha}\right)}
  \Big[ \alpha^{4} - 2 \, \alpha^{2} \varepsilon^{2} +
    \varepsilon^{4} + {\left(\alpha^{4} - \beta^{4} - \varepsilon^{4}
        + 4 \, \beta^{2} - 4 \, \log\beta - 3\right)}
    \log\!\left(\tfrac{\varepsilon}{\alpha}\right)\!\Big],
\end{equation}
where $\alpha = r - \eta$, $\beta = r + \eta$, and $C$ is a constant
independent of the radius of rotation $r$. To determine the optimal
$r = r_{\textrm{opt}}$, we calculate numerically the root of
$U^{\prime}(r_{\textrm{opt}})=0$, which is given by the zero of
\eqref{Deriv_AveMFPT} in the Supplementary Material.
\begin{figure}
  \hfill
  \makebox{
    \raisebox{1ex}{\small{(a)}}
    \includegraphics[width=0.38\textwidth]{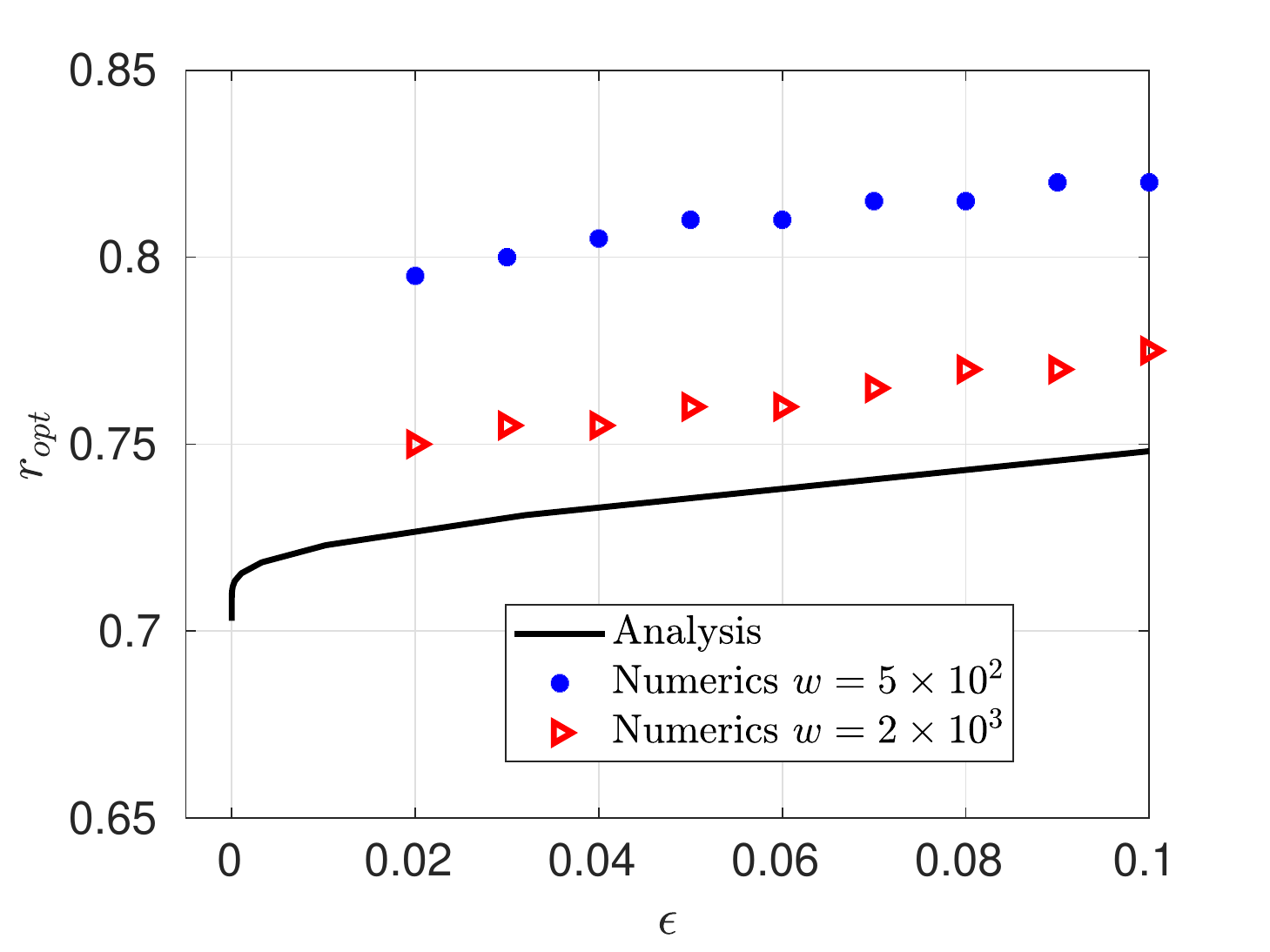}
    \phantomsubcaption
    \label{TwoTrap_Disk_Analysis_Eps}
  }
  \hfill
  \makebox{
    \raisebox{1ex}{\small{(b)}}
    \includegraphics[width=0.38\textwidth]{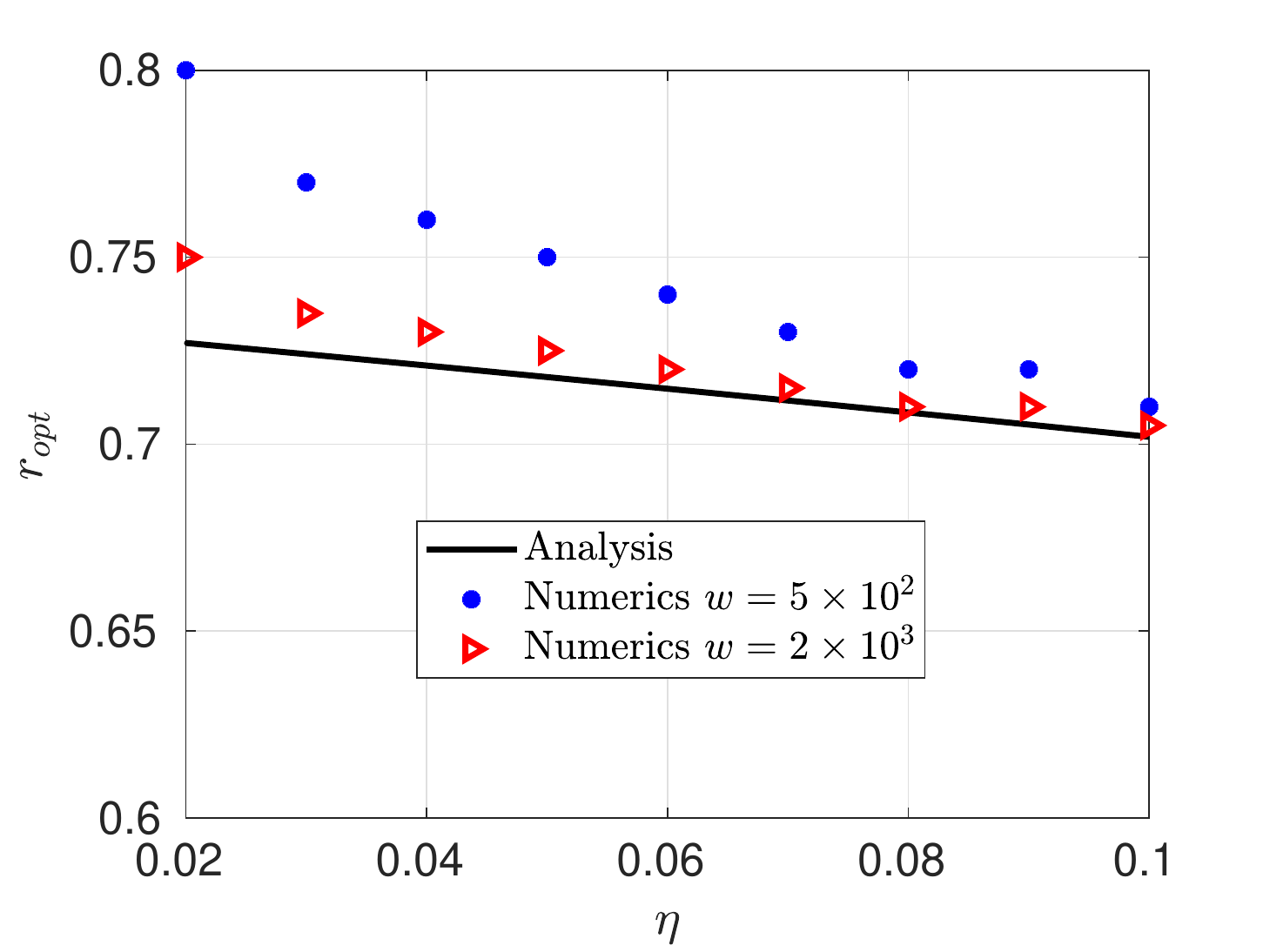}
    \phantomsubcaption
    \label{TwoTrap_Disk_Analysis_Eta}
  }
  \hfill
  \vspace*{-1ex}
  \caption{Optimal radius of rotation $r_{\textrm{opt}}$ for
      an absorbing trap of radius $\eta$ moving at constant angular
      frequency $\omega$ on a ring in a unit disk that contains an
      additional absorbing trap of radius $\varepsilon$ at the
      center of the disk.
      In (a) we fix $\eta=0.02$ and in (b) we fix $\varepsilon=0.02$.
      Numerical results (symbols) get closer to the asymptotic
      result (solid curve) for larger values of $\omega$.}
      \label{TwoTrap_Disk_Analysis}
\end{figure}
In Figure~\ref{TwoTrap_Disk_Analysis}, we show a comparison between
this asymptotic result for 
$r_{\textrm{opt}}$ and full numerical optimization results at the two
frequencies $\omega = 500$ and $\omega = 2000$, as obtained by using
the CPM with $\Delta x = 0.005$ and $\Delta r = 0.001$. As expected,
the asymptotic result, which is valid for $\omega \to \infty$, is seen to
agree more closely with the full numerical results when $\omega=2000$
than for $\omega=500$.

In Figure~\ref{TwoTrap_Disk_Analysis_Eps}, we show how the optimal
radius of rotation of a moving trap of radius $\eta = 0.02$ depends on
the radius $\varepsilon$ of the stationary trap centered at the
origin.  We observe that the optimal rotating trap moves closer to the
boundary of the unit disk as $\varepsilon$ increases. Since this
increase would reduce the MFPT for particles between the two traps,
the rotating trap tends to move closer to the boundary of the domain
in order to reduce the MFPT for particles between the moving trap and
the boundary of the unit disk. This in turn reduces the overall
average MFPT.  Alternatively, as the static trap radius shrinks, the
optimal radius of rotation decreases and, in the limit
$\varepsilon \to 0$, the optimal radius converges to
$r_{\textrm{opt}} = 0.7028$. Moreover,
$r_{\textrm{opt}} \to {1/\sqrt{2}}\approx 0.707$ as $\eta\to 0$. This
limiting radius for $\eta\to 0$ is the one that divides the unit disk
into two regions of equal area, and is consistent with that given in
equation (2.4) of \cite{tzou2015mean}.

In Figure~\ref{TwoTrap_Disk_Analysis_Eta}, we fix the radius of the
stationary trap at $\varepsilon=0.02$ and show how the optimal radius
of rotation of the moving trap depends on its radius $\eta$. For this
case, $r_{\textrm{opt}}$ decreases as $\eta$ increases.

\section{Discussion}\label{sec:discuss}

We have developed and implemented a Closest Point Method (CPM) to
numerically compute the average MFPT for a Brownian particle in a
general bounded 2-D confining domain that contains small stationary
circular absorbing traps.  A CPM approach was also formulated to
compute the average MFPT in domain that has a mobile
trap moving periodically along a concentric path within the
domain. Through either a refined discrete sampling procedure or from
a particle swarm optimizer routine \cite{kennedy2010}, optimal trap
configurations that minimize the average MFPT were identified
numerically for various examples.

For the stationary trap problem with a small number of traps, some
optimum trap configurations that minimize the average MFPT were
computed for a class of star-shaped domains and for an elliptical
domain with arbitrary aspect ratio. In particular, we have identified
numerically the optimum arrangement of three traps in an ellipse of a
fixed area as its boundary is deformed continuously.  Under this
boundary deformation we have shown that the optimal three-trap
arrangement changes from a ring-pattern of traps in the unit disk to a
colinear pattern of traps when the ellipse has a sufficiently large
aspect ratio.  Two distinct perturbation approaches were used in
\S~\ref{sec:skinnyellipse} to approximate the optimal trap locations and
optimal average MFPT for such a colinear trap pattern in a long, thin,
ellipse.

For a class of near-disk domains with boundary
$r=1+\sigma \cos(\mc{N}\theta)$ and $\sigma\ll 1$, we have used a
perturbation approach to calculate the leading-order and 
${\mathcal O}(\sigma)$ correction term for the average MFPT for a
pattern of $m$ equally-spaced traps on a ring (i.e.~ring pattern).
When ${\mc N}=k m$, for $k\in\mathbb{Z}^{+}$, we have shown
analytically from this formula that the optimal trap locations on a
ring must coincide with the maxima of the boundary
deformation. Explicit results for the perturbed optimal ring radius
are derived. In contrast, when ${{\mc N}/ m}\notin\mathbb{Z}^{+}$, we
have shown analytically that the problem of optimizing the average
MFPT for a ring pattern of traps is degenerate in the sense that
the ${\mathcal O}(\sigma)$ correction to the average MFPT vanishes for
{\em any} ring radius.  An open problem is to develop a hybrid
asymptotic-numerical approach to identify optimal trap configurations
allowing for arbitrary trap locations under an arbitrary, but small,
star-shaped boundary deformation of the unit disk given by
$r=1+\sigma h(\theta)$, where $\sigma\ll 1$ and $h(\theta)$ is a
smooth $2\pi$ periodic function. Such a general approach could be
applied to predict the initial change in the optimal locations of
three traps in the ellipse as computed using the CPM in
Figure~\ref{fig:three_traps_ellipse_pso}.

An interesting mobile trap problem is path optimization: for a given
domain, what is the optimal path for a trap to follow, subject to
e.g., an arclength constraint?  We can solve this problem
numerically using the techniques developed here
using constrained optimization.

Further improvements to our numerical method are possible.
Our periodic moving trap problem involves relaxing over many periods;
as a practical matter, we can
decrease the expense by running the algorithm using an initially
coarse spatial grid.  After the solution has converged (in time) on
the coarse grid, we can project the solution at time $t = NT$ onto a
finer spatial grid and repeat.

Finally, we note the numerical algorithms described here can be
applied for traps on manifolds where the Laplacian is replaced with
the Laplace--Beltrami operator.

\section*{Acknowledgements}
Colin Macdonald and Michael Ward were supported by NSERC
Discovery grants. Tony Wong was partially supported by a UBC 4YF.
The authors thank Justin Tzou for discussions that lead to the
time-relaxation algorithm for moving trap problems.

\bibliographystyle{plain}
\bibliography{Reference}

\begin{thebibliography}{10}

\bibitem{ChenMacdonald:ellipticCPM}
Y.~Chen and C.~B Macdonald.
\newblock The closest point method and multigrid solvers for elliptic equations
  on surfaces.
\newblock {\em SIAM J. Sci. Comput.}, 37(1):A134--A155, 2015.

\bibitem{cheviakov2010asymptotic}
A.~F Cheviakov, M.~J Ward, and R.~Straube.
\newblock An asymptotic analysis of the mean first passage time for narrow
  escape problems: Part ii: The sphere.
\newblock {\em SIAM J. Multiscale Model. Simul.}, 8(3), 2010.

\bibitem{coombs2009}
D.~Coombs, R.~Straube, and M.~Ward.
\newblock Diffusion on a sphere with localized traps: Mean first passage time,
  eigenvalue asymptotics, and fekete points.
\newblock {\em SIAM J. Appl. Math.}, 70(1), 2009.

\bibitem{engquist2005discretization}
Bj{\"o}rn Engquist, Anna-Karin Tornberg, and Richard Tsai.
\newblock Discretization of dirac delta functions in level set methods.
\newblock {\em Journal of Computational Physics}, 207(1):28--51, 2005.

\bibitem{grigoriev2002kinetics}
I.~V Grigoriev, Y.~A Makhnovskii, A.~M Berezhkovskii, and V.~Yu Zitserman.
\newblock Kinetics of escape through a small hole.
\newblock {\em The Journal of chemical physics}, 116(22):9574--9577, 2002.

\bibitem{holcman2004escape}
D.~Holcman and Z.~Schuss.
\newblock Escape through a small opening: receptor trafficking in a synaptic
  membrane.
\newblock {\em Journal of Statistical Physics}, 117(5-6):975--1014, 2004.

\bibitem{kennedy2010}
James Kennedy.
\newblock Particle swarm optimization.
\newblock {\em Encyclopedia of machine learning}, pages 760--766, 2010.

\bibitem{kolokolnikov2005optimizing}
T.~Kolokolnikov, M.~S Titcombe, and M.~J Ward.
\newblock Optimizing the fundamental {N}eumann eigenvalue for the {L}aplacian
  in a domain with small traps.
\newblock {\em European Journal of Applied Mathematics}, 16(2):161--200, 2005.

\bibitem{Venu}
V.~Kurella, J.~C Tzou, D.~Coombs, and M.~J Ward.
\newblock Asymptotic analysis of first passage time problems inspired by
  ecology.
\newblock {\em Bulletin of Mathematical Biology}, 77(1), 2015.

\bibitem{lindsay2017optimization}
A.~E Lindsay, J.~C Tzou, and T.~Kolokolnikov.
\newblock Optimization of first passage times by multiple cooperating mobile
  traps.
\newblock {\em SIAM J. Multiscale Model. Simul.}, 15(2), 2017.

\bibitem{macdonald2011}
C.~B Macdonald, J.~Brandman, and S.~J Ruuth.
\newblock Solving eigenvalue problems on curved surfaces using the closest
  point method.
\newblock {\em J. Comput. Phys.}, 230(22), 2011.

\bibitem{MacdonaldMerrimanRuuth:ptclouds}
C.~B. Macdonald, B.~Merriman, and S.~J. Ruuth.
\newblock Simple computation of reaction-diffusion processes on point clouds.
\newblock {\em Proc. Natl. Acad. Sci.}, 110(23), 2013.

\bibitem{macdonald2009}
C.~B Macdonald and S.~J Ruuth.
\newblock The implicit closest point method for the numerical solution of
  partial differential equations on surfaces.
\newblock {\em SIAM J. Sci. Comput.}, 31(6), 2009.

\bibitem{mirny2009}
L.~Mirny, M.~Slutsky, Z.~Wunderlich, A.~Tafvizi, J.~Leith, and A.~Kosmrlj.
\newblock How a protein searches for its site on dna: the mechanism of
  facilitated diffusion.
\newblock {\em Journal of Physics A: Mathematical and Theoretical}, 42(43),
  2009.

\bibitem{pillay2010asymptotic}
S.~Pillay, M.~J Ward, A~Peirce, and T.~Kolokolnikov.
\newblock An asymptotic analysis of the mean first passage time for narrow
  escape problems: Part i: Two-dimensional domains.
\newblock {\em SIAM J. Multiscale Model. Simul.}, 8(3), 2010.

\bibitem{redner}
S.~Redner.
\newblock {\em A guide to first-passage processes}.
\newblock Cambridge University Press, 2001.

\bibitem{ricc1985}
L.~M Ricciardi.
\newblock Diffusion approximations and first passage time problems in
  population biology and neurobiology.
\newblock In {\em Mathematics in Biology and Medicine}, pages 455--468.
  Springer, 1985.

\bibitem{ruuth2008}
S.~J Ruuth and B.~Merriman.
\newblock A simple embedding method for solving partial differential equations
  on surfaces.
\newblock {\em J. Comput. Phys.}, 227(3), 2008.

\bibitem{schuss2007narrow}
Z.~Schuss, A.~Singer, and D.~Holcman.
\newblock The narrow escape problem for diffusion in cellular microdomains.
\newblock {\em PNAS}, 104(41):16098--16103, 2007.

\bibitem{tzou2015mean}
J.~C Tzou and T.~Kolokolnikov.
\newblock Mean first passage time for a small rotating trap inside a reflective
  disk.
\newblock {\em SIAM J. Multiscale Model. Simul.}, 13(1), 2015.

\bibitem{van1992stochastic}
N.~G Van~Kampen.
\newblock {\em Stochastic processes in physics and chemistry}, volume~1.
\newblock Elsevier, 1992.

\bibitem{vonGlehnMarzMacdonald:cpmol}
I.~von Glehn, T.~M\"{a}rz, and C.~B Macdonald.
\newblock An embedded method-of-lines approach to solving partial differential
  equations on surfaces.
\newblock 2019.
\newblock Submitted.

\bibitem{ward2018spots}
M.~J Ward.
\newblock Spots, traps, and patches: {A}symptotic analysis of localized
  solutions to some linear and nonlinear diffusive systems.
\newblock {\em Nonlinearity}, 31(8):R189, 2018.

\bibitem{ward1993strong}
M.~J Ward and J.~B Keller.
\newblock Strong localized perturbations of eigenvalue problems.
\newblock {\em SIAM Journal on Applied Mathematics}, 53(3):770--798, 1993.

\end{thebibliography}


\cleardoublepage
\begin{appendix}
\renewcommand{\theequation}{\Alph{section}.\arabic{equation}}
\setcounter{equation}{0}
\setcounter{page}{1}
\headers{Supplementary Material: Simulation and Optimization of MFPT}{S. Iyaniwura, T. Wong, M. J. Ward, and C. B. Macdonald}

\begin{center}
  \textbf{\uppercase{Simulation and Optimization of Mean First Passage Time Problems
      in 2-D Using Numerical Embedded Methods and Perturbation Theory:\\
      Supplementary Material}
  }

  \medskip

  \small{Sarafa Iyaniwura, Tony Wong, Michael J. Ward, and Colin~B.~Macdonald}
\end{center}

\thispagestyle{empty}

\bigskip

\par\refstepcounter{section}

\subsection{Asymptotic analysis of the MFPT for a perturbed
  unit disk}\label{supp:anal_1}

We summarize the derivation of the result given in
Proposition~\ref{u_bar_prop_main} of
\S~\ref{sec:asymp_perturbed_unit_disk}.

We start by studying the leading-order problem \eqref{LeadingOrder}
using the method of matched asymptotic expansions.  In the inner
region near each of the traps, we introduce the inner variables
$\v{y} = \varepsilon^{-1}(\x - \x_j)$ and
$u_0(\x) = v_j(\varepsilon \v{y}+\x_j)$ with $ \rho = |\v{y}| $, for
$j = 0, \dots, m-1$.  Upon writing 
\eqref{LeadingOrder} in terms of these variables, we have for
$\varepsilon \to 0$ that for each $j=0,\ldots,m-1$
\begin{equation}\label{LeadingOrderInner}
\begin{split}
 \Delta_{\rho}\, v_j  & = 0 \,, \quad \rho > 1\,;\qquad
 v_j =0 \quad   \mbox{on} \quad \rho = 1\,, 
\end{split}
\end{equation} 
where
$\Delta_{\rho} \equiv \partial_{\rho \rho} + \rho^{-1}
\partial_{\rho}$. The radially symmetric solution is
  $v_j=A_j\log\rho$, where $A_j$ for $j = 0, \dots, m-1$ are
  constants to be determined.  By matching the inner solution to the
outer solution we obtain the singularity behavior of the outer
solution $u_0$ as $\x \to \x_j$ for $j = 0, \dots, m-1$.
This leads to the following problem for $u_0$:
\begin{subequations}\label{LeadingOrder_CompleteOuter}
\begin{align}
  D\, \nabla^2 u_0  = -1\,, \quad \x & \in \Omega\setminus
      \lbrace{\x_0,\ldots,\x_{m-1}\rbrace}\,;\quad \partial_r u_{0} =0\,,
    \quad \x \in \partial \Omega\,; \label{LeadingOrder_CompleteOuterA}\\
  u_0 \sim A_j  \log |\x - \x_j| &+ A_j/\nu \quad \text{as} \quad\x \to \x_j
       \qquad j = 0, \dots, m-1\,.\label{LeadingOrder_CompleteOuterB}
\end{align}
\end{subequations} 
Here $\nu \equiv -1/\log\varepsilon$.  In terms of a
  Dirac forcing, this problem for $u_0$ is equivalent to
\begin{equation}\label{LeadingOuter_delta}
  \nabla^2 u_0  = -\frac{1}{D} + 2 \pi \sum_{j = 0 }^{m-1} A_j \delta(\x - \x_j)\,,
 \qquad \partial_r u_{0} =0\,,  \,\,\, \x \in \partial \Omega\,.
\end{equation}
From integrating \eqref{LeadingOuter_delta} over the unit disk, and using
the divergence theorem, we get
\begin{align}\label{SolvabilityCondition}
 \sum_{j=0}^{m-1} A_j  = \frac{|\Omega|}{2\pi D} \,.
\end{align}
Next, we introduce the Neumann Green's function $G(\x ; \x_j)$, 
which satisfies 
\begin{subequations}\label{GreenFunctionProb}
\begin{align}
  \nabla^2 G  &= \frac{1}{|\Omega|} - \delta(\x - \x_j)\quad \x \in \Omega\,;
 \quad \partial_n G =0\,, \quad \x \in \partial \Omega\,;
                \label{GreenFunctionProb_A}\\
  G  \sim -\frac{1}{2\pi}& \log{|\x - \x_j|} + R_j + o(1) \quad \text{as}
  \quad\x \to \x_j\,; \qquad \int_{\Omega} G \,\text{d}\x=0\,,
                           \label{GreenFunctionProb_B}
\end{align}
\end{subequations}
where $R_j \equiv R(\x_j)$ is the regular part of the Green's function
at $\x = \x_j$. In terms of this Green's function, we write the solution to
\eqref{LeadingOuter_delta} as
\begin{align}\label{SolutionOuterLead}
u_0 = -2 \pi \sum_{i=0}^{m-1}  A_i G(\x ; \x_i)  + \overline{u}_0 \,,
\end{align}
where $\overline{u}_0 = (1/|\Omega|)\int_{\Omega} u_0 \, \text{d}\x$
is the leading-order average MFPT. Expanding \eqref{SolutionOuterLead}
as $\x \to \x_j$ for each of the traps, and using the singularity
behavior of $G(\x ; \x_j)$ given in \eqref{GreenFunctionProb_B}, we
obtain for each $j=0,\ldots,m-1$ that
\begin{align}\label{SolutionOuterLead_Expand}
  u_0 \sim A_j \log{|\x - \x_j|}  -2 \pi A_j\,R_j  -2 \pi
  \sum_{i \neq j}^{m-1} A_i \, G(\x_j ; \x_i) + \overline{u}_0 \,.
\end{align}
The asymptotic matching condition in  this local behavior of the
outer solution  must agree with the behavior
\eqref{LeadingOrder_CompleteOuterB} as $\x \to \x_j$.  In this way, and
  recalling \eqref{SolvabilityCondition}, we obtain an algebraic
  system of equations for $\overline{u}_0, A_0, \dots, A_{m-1}$ given
  in matrix form as
\begin{align}\label{Alg_Matrix}
  ( I + 2\pi \nu \, \mathcal{G}) \mathcal{A} = \nu\, \overline{u}_0 \,\v{e}\,,
  \qquad \v{e}^T \mathcal{A} = \frac{|\Omega|}{2\pi D} \,.
\end{align} 
Here, $\v{e} \equiv (1,\dots,1)^T$, $\nu = -1/\log\varepsilon$,
$I$ is the identity matrix,
$\mathcal{A} \equiv (A_0, \dots, A_{m-1})^T$, and $\mathcal{G}$ is the
symmetric Green's matrix whose entries are defined in terms of the
Neumann Green's function of \eqref{GreenFunctionProb} by
\begin{align}\label{GreenMAtrix}
  (\mathcal{G})_{jj} = R_j \equiv R(\x_j)\,\,\, \text{for} \,\,\, i = j
  \quad \text{and} \quad (\mathcal{G})_{ij} = (\mathcal{G})_{ji}  =
  G(\x_i ; \x_j) \,\,\, \text{for} \,\,\, i \neq j \,.
\end{align} 
Since the traps are equally-spaced on the ring, the Green's matrix
$\mc{G}$ in \eqref{GreenMAtrix} is also cyclic. Thus, from
\cite[Prop 4.3]{kolokolnikov2005optimizing}, $\v{e}$ is an
eigenvector of $\mc{G}$ and we have that
\begin{equation}\label{GreenEigVal}
  \mc{G}\v{e} = \kappa_1 \v{e}\,, \qquad  
  \kappa_1 =  \frac{1}{2\pi } \left[-\log(m\,r_c^{m-1}) -
    \log(1 - r_c^{2m}) + m r_c^2 -\frac{ 3}{4}m \right]\,.
\end{equation} 
Then, by setting $\mc{A} = A_c \, \v{e}$, for some common value
$A_c$, in \eqref{Alg_Matrix}, we readily obtain
\begin{equation}\label{Alg_A}
  A_c = \frac{|\Omega|}{2\pi m D} = \frac{1}{2mD}, \quad \text{and} \quad
  \overline{u}_0 = \frac{1}{2m \nu D}(1 + 2\pi \nu \kappa_1)\,,
\end{equation} 
where $\kappa_1$ is given in \eqref{GreenEigVal}. Since
$\kappa_1 \equiv \kappa_1(r_c)$, any ring radius $r_c$ that minimizes
$\kappa_1$ also minimizes the leading-order average MFPT $\overline{u}_0 $.  
This yields the leading-order term in Proposition~\ref{u_bar_prop_main} of
\S~\ref{sec:asymp_perturbed_unit_disk}.

Next, we study the $\mathcal{O}(\sigma)$ problem for $u_1$ given in
\eqref{OrderSigma}. Following a similar approach used to solve the
leading-order problem, we construct an inner region close to each of
the traps and introduce the inner variables
$\v{y} = \varepsilon^{-1}(\x - \x_j)$ and
$u_1(\x) = V_j(\varepsilon \v{y}+\x_j)$ with $ \rho = |\v{y}|
$. From \eqref{OrderSigma}, this yields the leading-order inner problem
\begin{equation}\label{OrderSigmaInner}
 \Delta_{\rho}\, V_j   = 0 \,, \quad \rho > 1\,;\qquad 
 V_j =0\,, \quad \mbox{on} \,\,\, \rho = 1\,, 
\end{equation} 
where
$\Delta_{\rho} \equiv \partial_{\rho \rho} + \rho^{-1}
\partial_{\rho}$.  The radially symmetric solution is
$V_j = B_j \log\rho$, where $B_j$ for $j = 0, \dots, m-1$ are
constants to be determined.  Matching this inner solution to the outer
solution, we derive the singularity behavior of the outer solution
$u_1$ as $\x \to \x_j$ for $j = 0, \dots, m-1$.  In this way, from
\eqref{OrderSigma}, we obtain that $u_1$ satisfies
\begin{subequations}\label{OrderSigma_CompleteOuter}
\begin{align}
\nabla^2 u_1  & = 0\,, \quad \x\in \Omega \setminus \{\x_0,\dots,\x_{m-1}  \}\,;
 \quad \partial_r u_1  = -h u_{0rr} + h_{ \theta} u_{0 \theta}\,, \quad \text{on}
                  \quad r = 1; \label{OrderSigma_CompleteOuterA}\\
  u_1 \sim & B_j  \log{|\x - \x_j|} + B_j/\nu \quad \text{as} \quad
   \x \to \x_j, \qquad  j = 0, \dots, m-1\,,\label{OrderSigma_CompleteOuterB}
\end{align}
\end{subequations} 
where $\nu = -1/ \log\varepsilon$. To determine $u_1$, we
need to derive its boundary condition on $r=1$ using the leading-order
MFPT $u_0$ given in \eqref{SolutionOuterLead} in terms of the Neumann
Green's function $G(\x;\x_i)$.  To do so, we use the Fourier series
representation of the Neumann Green's function \eqref{GreenFunctionProb}
in the unit disk given by
 \begin{align}\label{GreenSolution}
   G(\x;\x_k) = \frac{1}{4\pi} (r^2 + r_c^2) - \frac{3}{8\pi} - \frac{1}{2\pi}
   \log {r_{>}} + \frac{1}{2\pi} \sum_{n=1}^{\infty} \frac{r^n_{<}}{n} (r_{>}^n
   + r_{>}^{-n})\cos(n(\theta - \theta_k))\,,
\end{align} 
where $\x = r\, e^{i \theta}$, $\x_k = r_c \,e^{i (2\pi k/m + \psi)}$,
$r_{>} = \max(r,r_c)$, and $r_{<} = \min(r,r_c)$. For any point $\x$ on
the boundary of the unit disk, $r_{>} = r = 1$, and $r_{<} = r_c
$. Upon substituting \eqref{GreenSolution} into
\eqref{SolutionOuterLead}, and using $A_c$ as given in \eqref{Alg_A},
we conclude that 
\begin{align}\label{U0_fulSol}
  u_0 = -2\pi & A_c \left[\frac{m}{4\pi}(1 + r_c^2) - \frac{3m}{8\pi} + \frac{1}{\pi} \sum_{n=1}^{\infty} \frac{r_c^n}{n} S_n  \right] + \overline{u}_0 \,, \quad    \mbox{on} \quad r=1 \,, \\
  \text{where} \qquad S_n = \sum_{k=0}^{m-1} & \cos(n (\theta - \theta_k))\,,
   \quad \text{with} \quad \theta_k = \frac{2 \pi k}{m} + \psi\,.
                \qquad\qquad\qquad\qquad \nonumber
\end{align} 
To determine a Fourier series representation for $u_0$, we first
  need to sum $S_n$. To do so we need the following simple lemma:

\begin{lem}\label{Lemma1}
For $d \neq 2 \pi l$ for $l = 0, \pm1, \pm2, \dots$, we have
\begin{equation}\label{trig:iden}
  C \equiv \sum_{k=0}^{m-1} \cos(a + kd) =
  \frac{\sin(md/2)}{\sin(d/2)} \cos\left[a + (m-1)d/2\right]\,.
\end{equation}
\begin{proof}
  We multiply both sides of \eqref{trig:iden} by
  $2 \sin\left({d/2}\right)$ and use the trigonometric product-to-sum
  formula, $2 \sin(x)\cos(y) = \sin(x+y) - \sin(x-y)$. This yields a
  telescoping series, which is readily summed as
\begin{align*}
  2 C\sin(d/2) &=  \sum_{k=0}^{m-1} 2\cos(a + kd) \sin(d/2) \,, \\
     & = \sum_{k=0}^{m-1}\left( \sin\left(a + \frac{(2k+1)}{2}d \right) -
       \sin\left(a + \frac{(2k-1)}{2}d \right) \right) \,, \\
    & = \sin\left(a - \frac{d}{2} \right) +
  \sin\left( \left( a - \frac{d}{2} \right) +  md \right)  \,, \\
   & = 2 \sin\left(\frac{md}{2} \right) \cos\left[a + \frac{(m-1)d}{2}
                 \right]\,. 
\end{align*}
Now, suppose that $\sin(d/2) \neq 0$, so that $d \neq 2 \pi l$ for any
$l = 0, \pm1, \pm2, \dots$. Then,
\begin{align*}
C = \frac{\sin(md/2)}{\sin(d/2)} \cos\left[a + \frac{(m-1)d}{2} \right]\,.
\end{align*}
\end{proof}
\end{lem}

\noindent By using Lemma \ref{Lemma1}, we can calculate $S_n$,
  as defined in \eqref{U0_fulSol}, as follows:
\begin{lem}\label{Lemma2}
	For $n \geq 1$ and $j^{\prime}=1,2,\ldots$, we have
\begin{align}\label{lemma2:sn}
S_n = \begin{cases}
  m \cos \Big{(} j^{\prime}m(\theta - \psi) \Big{)}, \quad \text{if}
  \quad n = j^{\prime} m\\
0, \quad \text{if} \quad n \neq j^{\prime} m \,.
\end{cases}
\end{align}
\end{lem}
\begin{proof}
  Define $a$ and $d$ by $a = n(\theta - \psi)$ and $d = - {2\pi n/m}$.
  From Lemma \ref{Lemma1}, it follows that if $d \neq 2\pi l$ for
  $l = 0, \pm1, \pm2, \dots$, then $S_n$ satisfies
\begin{align}
  S_n & =  \sum_{k=0}^{m-1} \cos \Big{(}n (\theta - \psi) -  \frac{2 \pi nk}{m}
        \Big{)}
        = \frac{\sin(\pi n)}{\sin\left(\frac{\pi n}{m}\right)}
        \cos \Big{(} n(\theta - \psi) - \pi
      n \frac{(m-1)}{m} \Big{)}\,, \nonumber \\
      &=  \frac{\sin(\pi n)}{\sin\left(\frac{\pi n}{m}\right)}
        \left[  \cos \Big{(} n(\theta - \psi)
  \Big{)} \cos \Big{(}\frac{\pi n(m-1)}{m} \Big{)} +
  \sin \Big{(} n(\theta - \psi) \Big{)} \sin \Big{(}\frac{\pi n (m-1)}{m}
  \Big{)} \right] \label{Sn_term}
\end{align}
This equation is valid provided that
$(n/m) \neq j^{\prime} \in \lbrace{1,2,\ldots\rbrace}$. We observe
from \eqref{Sn_term} that $S_n = 0$ for $n=1,2,\dots$ with
$n \neq j^{\prime}m$. Alternatively, if $n = j^{\prime}m$ for some
$j^{\prime}=1,2,\ldots$, then we need to evaluate the prefactor in
\eqref{Sn_term} using L'H\^{o}pital's rule. To this end, we define
$g(x) \equiv \frac{\sin(\pi x)}{\sin(\pi x/m)}$, so that using
L'H\^{o}pital's rule we get
$g(x) \to {m \cos(\pi j^{\prime}m)/[\cos(\pi j^{\prime})]}$ as
$x \to j^{\prime}m$. Therefore, from \eqref{Sn_term}, we derive for
$n = j^{\prime}m$ that
\begin{equation}\label{Sn_n_jm}
  S_n  =  \frac{m \cos(\pi j^{\prime}m)}{\cos(\pi j^{\prime})}
 \cos \Big{(} j^{\prime}m(\theta - \psi) \Big{)}\Big{[}
        \cos (\pi j^{\prime}m ) \cos(\pi j^{\prime}) \Big{]} =
        m \cos \Big{(} j^{\prime} m(\theta - \psi) \Big{)}\,.
\end{equation}
      
\end{proof}

Next, by substituting \eqref{lemma2:sn} for $S_n$, together with
$A_c = 1/(2mD)$ (see \eqref{Alg_A}), in \eqref{U0_fulSol}, we obtain
the Fourier series representation for $u_0$ on $r=1$ given by
\begin{equation}\label{OrderSigma_new}
\begin{split}
  \qquad  u_0   &= c_0 +  \sum_{j^{\prime}=1}^{\infty} c_{j^{\prime}}
  \cos \big{(} j^{\prime}m(\theta - \psi) \big{)}\,, \quad \mbox{on} \,\,\,
  r=1 \,, \\
 \text{where}  \quad
 c_0 &= -\frac{1}{8 D} \Big{(}2(1 + r_c^2) - 3 \Big{)}  +
 \overline{u}_0\,; \qquad c_{j^{\prime}} = - \frac{r_c^{j^{\prime}m}}{j^{\prime} mD}
 \,, \quad j^{\prime}=1,2,\ldots\,. 
\end{split}
\end{equation}

We return to the $\mc{O}(\sigma)$ outer problem
\eqref{OrderSigma_CompleteOuter} for $u_1$ and simplify the boundary
condition on $r=1$ given in \eqref{OrderSigma_CompleteOuterA} as
$u_{1r} = F(\theta) \equiv -h u_{0rr} + h_{\theta} u_{0 \theta}$ on
$r=1$. Since $u_0$ satisfies the MFPT PDE, in polar coordinates we
have that $u_{0rr} + r^{-1}u_{0r} + r^{-2}u_{0 \theta \theta} =
-1/D$. Evaluating this on $r=1$ where $u_{0r}=0$, we get that
$u_{0rr} =- u_{0 \theta \theta} -1/D$ on $r=1$. Upon substituting this
expression for $u_{0rr}$ into $F(\theta)$, we derive
\begin{align}\label{F_u1_BC}
  u_{1r}= F(\theta) = (h u_{0 \theta})_{\theta} + \frac{h}{D} \,, \quad \text{on}
  \quad r =1 \,,
\end{align}
where $u_0$ on $r=1$ is given in \eqref{OrderSigma_new} and
$h(\theta)=\cos(\mc{N} \theta)$. 

Next, we write the problem \eqref{OrderSigma_CompleteOuter}
  for $u_1$ as
\begin{align}\label{OrderSigma_CompleteOuter2}
  \nabla^2 u_1   = 2\pi \sum_{i=0}^{m-1} B_i \,\,\delta(\x - \x_i)\,,
  \quad \x \in \Omega \,; \qquad u_{1r}  = F(\theta)\,, \quad \text{on}
  \quad r = 1\,.
\end{align}
Integrating \eqref{OrderSigma_CompleteOuter2} over the unit disk, and
using the divergence theorem and the fact that
$\int_0^{2 \pi} F(\theta) \, \text{d}\theta = 0$, we conclude that
$\sum_{j=0}^{m-1} B_j = 0$.  It is then convenient to decompose
$u_1$ as
\begin{equation}\label{u1_Decompose}
u_1 =  u_{1H} + u_{1p} + \overline{u}_1 \,,
\end{equation}
where the unknown constant $\overline{u}_1$ is the average
of $u_1$ over the unit disk.  Here, $u_{1 H}$ is taken to be the
unique solution to
\begin{equation}\label{U1H_Prob}
  \nabla^2 u_{1H}  = 2\pi  \sum_{i=0}^{m-1} B_i\, \delta(\x - \x_i)\,,
  \quad \x  \in \Omega\,; \quad \partial_r u_{1H}  = 0 \,, \quad
  \text{on} \quad r = 1\,; \quad \int_{\Omega} u_{1H} \, \text{d}\x = 0 \,.
\end{equation}
In addition, $u_{1p}$ is defined to be the unique solution to 
\begin{align}\label{u1P_Prob}
 \nabla^2 u_{1p} &= 0, \quad \x  \in \Omega ; \quad
  \partial_r u_{1p} =  F(\theta) \, \quad \text{on} \quad r = 1; \quad
       \int_{\Omega} u_{1p} \, \text{d}\x = 0 \,,
\end{align}
which is readily solved using separation of variables once
$F(\theta)$ is represented as a Fourier series.

The solution to \eqref{U1H_Prob} is represented in terms of the
Neumann Green's function $G(\x;\x_i)$ of \eqref{GreenFunctionProb}, so
that
\begin{align}\label{u1_Sol}
u_1 = -2 \pi \sum_{i=0}^{m-1}  B_i  G(\x ; \x_i) + u_{1p}+ \overline{u}_1.
\end{align}
Expanding \eqref{u1_Sol} as $\x \to \x_j$, and using the singularity
behavior of $G(\x ; \x_j)$ as given in \eqref{GreenFunctionProb_B},
we derive the local behavior of $u_1$ as $\x \to \x_j$, for each
$j=0,\ldots,m-1$, which must agree with that given in
\eqref{OrderSigma_CompleteOuterB}. This yields an $(m+1)$ dimensional
algebraic system of equations for the constants $B_0,\dots,B_{m-1}$
and $\overline{u}_1$ given in matrix form by
\begin{equation}\label{System_BNu_Mat}
  (I + 2 \pi \nu \mc{G} )\v{B} = \nu \overline{u}_1 \v{e} + \nu \v{u}_{1p}\,,
  \qquad  \v{e}^T \v{B} = 0 \,.
\end{equation}
	Here, $I$ is the $m \times m$ identity matrix,
$\v{B} = (B_0,\dots,B_{m-1})^T$, $\v{e} = (1,\dots,1)^T$, and
$\v{u}_{1p} = (u_{1p}(\x_0), \dots, u_{1p}(\x_{m-1}))^T$.  Upon
multiplying this equation for $\v{B}$ on the left by $\v{e}^T$, we can
isolate $\overline{u}_1$ as
\begin{equation*}
  \nu\, \overline{u}_1 = \frac{1}{m} \Big{(} 2\pi \nu \v{e}^T \mc{G} \v{B} -
  \nu \v{e}^T \v{u}_{1p}  \Big{)} \,.
\end{equation*}
Upon re-substituting this expression into \eqref{System_BNu_Mat}, we
conclude that $\v{e}^T \v{B}=0$ and that
\begin{equation}\label{u1_bar}
\Big{[} I + 2 \pi \nu (I - E)\mc{G}  \Big{]} \v{B} = \nu (I - E)\v{u}_{1p}\,,
\quad \text{and}  \quad \overline{u}_1 = - \frac{1}{m}
\Big{(} \v{e}^T \v{u}_{1p} - 2 \pi \v{e}^T \mc{G} \v{B}  \Big{)}\,,
\end{equation}
where we have defined $E={\v{e}\v{e}^T/m}$.  This gives an equation
for the $\mc{O}(\sigma)$ average MFPT $\overline{u}_1$ in terms of the
Neumann Green's matrix $\mc{G}$, and the vectors $\v{B}$ and
$\v{u}_{1p}$.

The next step in this calculation is to solve \eqref{u1P_Prob} so as
to calculate $u_{1p}(\x_j)$ for $j=0,\ldots,m-1$. To do so, we first
need to find an explicit Fourier series representation for
$F(\theta)$, as defined in \eqref{F_u1_BC} in terms of $u_0$ on $r=1$.

By using \eqref{OrderSigma_new} for $u_0$ on $r=1$, together with
$h=\cos(\mc{N} \theta)$, we calculate that
\begin{equation*}
\begin{split}
  h u_{0 \theta} &= - \frac{\cos(\mc{N}\psi)}{2} \sum_{j^{\prime} =1}^{\infty}
  c_{j^{\prime}} j^{\prime}m \Big{[}  \sin \Big{(} (j^{\prime}m +\mc{N})
  (\theta - \psi)\Big{)} + \sin \Big{(} (j^{\prime}m - \mc{N})
  (\theta - \psi)\Big{)} \Big{]} \\
  & + \frac{\sin(\mc{N}\psi)}{2} \sum_{j^{\prime} =1}^{\infty} c_{j^{\prime}}
  j^{\prime}m \Big{[}  \cos \Big{(} (j^{\prime}m - \mc{N})(\theta - \psi)\Big{)}
  - \cos \Big{(} (j^{\prime}m + \mc{N})(\theta - \psi)\Big{)} \Big{]}\,.
\end{split}
\end{equation*}
	Upon differentiating this expression with respect to $\theta$, we obtain
  after some algebra that
\begin{equation}\label{Der_h_u0}
\Big{(} h(\theta) u_{0 \theta} \Big{)}_{\theta}= - \sum_{j^{\prime} =1}^{\infty}              
\frac{c_{j^{\prime}} j^{\prime}m}{2} \Big{[} j^{\prime}_{+}  \cos
\Big{(} j^{\prime}_{+}(\theta - \psi)+ \mc{N}\psi \Big{)} +
j^{\prime}_{-}  \cos \Big{(} j^{\prime}_{-}(\theta - \psi)- \mc{N}\psi \Big{)}
\Big{]} \,,
\end{equation}
where we have defined $j^{\prime}_{\pm}$ by
$j^{\prime}_{\pm}= j^{\prime}m \pm \mc{N}$.  Upon substituting \eqref{Der_h_u0}
into \eqref{F_u1_BC}, and recalling that
$c_{j^{\prime}} = -(r_c^{j^{\prime} m})/(j^{\prime}mD)$, we conclude
that
\begin{equation}\label{F_equa}
  F(\theta) = \frac{1}{D} \cos(\mc{N} \theta) + \frac{1}{2D}
  \sum_{j^{\prime} =1}^{\infty} r_c^{j^{\prime}m}\Big{[} j^{\prime}_{+}
  \cos \Big{(} j^{\prime}_{+}(\theta - \psi)+ \mc{N}\psi \Big{)} +
  j^{\prime}_{-}  \cos \Big{(} j^{\prime}_{-}(\theta - \psi)- \mc{N}\psi \Big{)}
  \Big{]} \,.
\end{equation}
With $F(\theta)$ as given in \eqref{F_equa}, by separation of
variables the solution $u_1$ to \eqref{u1P_Prob} that is bounded as
$r\to 0$ is
\begin{equation}\label{u1P_Sol}
\begin{split}
u_{1p} &=   \sum_{\substack{j^{\prime} =1\\ j^{\prime}_{-} \neq\, 0} }^{\infty}              
\frac{r_c^{j^{\prime}m}}{2D}  \Big{[} r^{ j^{\prime}_{+}}  \cos \Big{(}
j^{\prime}_{+}(\theta - \psi)+ \mc{N}\psi \Big{)} +
\gamma \, r^{ |j^{\prime}_{-}|} \cos \Big{(} j^{\prime}_{-}(\theta - \psi)-
\mc{N}\psi \Big{)}  \Big{]} \\
     & \qquad + \frac{r^{\mc{N}} \cos(\mc{N} \theta)}{\mc{N} D} \,,
\end{split}
\end{equation}
where $\gamma = \text{sign}( j^{\prime}_{-})$, $m$ is the number of
traps on the ring of radius $r_c$, and $\mc{N}$ is the number of folds
on the star-shaped domain. If $\mc{N} > m$, then $j'_{-} < 0$ at least
for $j^{\prime} = 1$, while when $\mc{N} = m $ then
$j^{\prime}_{-} = 0$ when $j^{\prime}=1$.

Next, using the explicit solution \eqref{u1P_Sol}, we calculate
$u_{1p}$ at the centers of the traps given by
$\x_j = r_c \,\exp{\Big{(}(2\pi j/m + \psi)i \Big{)} }$ for
$j=0,\ldots,m-1$. At $\x = \x_j$, we have
$\theta = 2\pi {j/m} + \psi$, so that
$\cos(\mc{N} \theta) = \cos \Big{(} \mc{N}\psi + 2\pi j \mc{N}/m
\Big{)}$. Similarly, we obtain
\begin{equation}\label{Jplus_Jminus}
  \cos \Big{(} j^{\prime}_{+} (\theta - \psi) + \mc{N}\psi  \Big{)} =
  \cos \Big{(} j^{\prime}_{-} (\theta - \psi) - \mc{N}\psi  \Big{)} =
  \cos \Big{(}  \mc{N}\psi + 2\pi j \mc{N}/m \Big{)}\,.
\end{equation}
Upon evaluating \eqref{u1P_Sol} at $\x=\x_j$ and using
\eqref{Jplus_Jminus}, we obtain that
\begin{equation}\label{u1P_Sol_2}
  u_{1p}(\x_j) = \frac{r_c^{\mc{N}}}{2D}\cos
 \left(   \mc{N} \Big{(} \psi + \frac{2\pi j}{m}
  \Big{)} \right) \left[\frac{2}{\mc{N}} +
    \sum_{j^{\prime}=1}^{\infty} r_c^{2mj^{\prime}}
    + \sum_{\substack{j^{\prime} =1\\ j^{\prime}_{-} \neq\, 0} }^{\infty}
    \text{sign}( j^{\prime}_{-}) r_c^{( j^{\prime}m + |j^{\prime}_{-}|-\mc{N})}\right]
\end{equation}
for $j=0,\ldots,m-1$. This expression is used to determine the vector
$\v{u}_{1p}$ in \eqref{u1_bar}. Observe from \eqref{u1P_Sol_2} that
$u_{1p}(\x_j) $ is independent of $j$ when $\mc{N}/m$ is a positive
integer. In other words, $u_{1p}$ is independent of the location of
the traps when the number of folds $\mc{N}$ of the perturbation of the
boundary is an integer multiple of the number of traps $m$ contained
in the domain.

Finally, upon substituting $h(\theta) = \cos(\mc{N} \theta)$ and
$u_0$, as given in \eqref{OrderSigma_new}, into
\eqref{AveMFPT_Perturb}, we can evaluate the third integral in
\eqref{AveMFPT_Perturb}. In this way, we conclude that a two-term
expansion in $\sigma$ for the average MFPT $\overline{u}$ is
\begin{equation}\label{AveMFPT_Perturb2}
\begin{split}
\overline{u} \sim \overline{u}_0 + \sigma  \overline{u}_1 +   
\begin{cases}
  0, \quad &\text{if} \quad  ( \mc{N}/m) \notin \mathbb{Z}^{+} \\
  - \sigma \Big{(} r_c^{\mc{N}} \cos(\mc{N} \psi)\Big{)}/(\mc{N} D),
  \quad &\text{if} \quad (\mc{N}/m) \in \mathbb{Z}^{+}
\end{cases} \,,
\end{split}
\end{equation}
where $\mathbb{Z}^{+}$ is the set of positive integers. Here
$\overline{u}_0$ and $\overline{u}_1$ are the leading-order and
$\mc{O}(\sigma)$ average MFPT given by \eqref{Alg_A} and the solution
to \eqref{u1_bar}, respectively.

The remainder of the calculation depends on whether
${\mc{N}/m}\in \mathbb{Z}^{+}$ or ${\mc{N}/m}\notin \mathbb{Z}^{+}$. We will
consider both cases separately.

\subsubsection{Number of folds is an integer multiple of the
  number of traps: \texorpdfstring{$(\mc{N}=k m)$} {\mc{N}=k m}}

When the number of folds on the star-shaped domain is an integer
multiple of the number of traps contained in the domain, then, from
\eqref{u1P_Sol_2}, we conclude that $u_{1p}(\x_j) $ is independent of
$j$. Therefore, using \eqref{u1P_Sol_2} and noting that
$j_{-}=(j^{\prime}-k)m$ and
$\text{sign}(j_{-})=\text{sign}(j^{\prime}-k)$, we calculate
$\v{u}_{1p} = (u_{1p}(\x_0),\dots,u_{1p}(\x_{m-1}))^T$ as
\begin{equation}\label{u1P_with_S}
\begin{split}
  \v{u}_{1p} & \equiv u_{1pc}\, \v{e}\,, \quad  \text{with}
  \quad u_{1pc} = \frac{1}{D} \cos(m \psi)\, \chi\,,\\
  \text{where} \quad \chi & \equiv \frac{r_c^{\mc{N}}}{\mc{N}} +
  \frac{1}{2} r_c^{\mc{N}} \sum_{j^{\prime}=1}^{\infty} r_c^{2m j^{\prime}} 
  -\frac{1}{2} \sum_{j^{\prime}=1}^{k-1} r_{c}^{j^{\prime} m + m (k-
    j^{\prime})} + \frac{1}{2} \sum_{j^{\prime}=k+1}^{\infty}
      r_c^{j^{\prime} m + m(j^{\prime}-k)} \,.
\end{split}
\end{equation}
{We observe that the third term in $\chi$ is proportional to
  $(k-1)$, and that we can combine the second and fourth terms into
  a single geometric series by shifting indices. In this way, and by using
  $m k=\mc{N}$, we can calculate $\chi$ explicitly as
\begin{equation}\label{u1P_with_SS}
  \chi = r_c^{\mc{N}} \left( \frac{1}{\mc{N}} - \frac{1}{2} (k-1)\right)
  + r_c^{\mc{N}} \sum_{j^{\prime}=1}^{\infty} r_c^{2 j^{\prime} m}  =
  r_c^{\mc{N}} \left( \frac{1}{\mc{N}} - \frac{1}{2} (k-1) \right)
  + \frac{r_c^{\mc{N} + 2m}}{1 - r_c^{2m}} \,.
\end{equation}}
Substituting \eqref{u1P_with_S} into \eqref{u1_bar}, and
noting that $(I-E)\v{u}_{1p}=0$ and that the matrix
$(I + 2 \pi \nu (I - E)\mc{G} )$ is invertible, we conclude that
$\v{B} = \v{0}$.  Therefore, from \eqref{u1_bar} we get that
$\overline{u}_1 = -u_{1pc}$.  In this way, by using
  \eqref{u1P_with_S}, \eqref{u1P_with_SS}, and
  \eqref{AveMFPT_Perturb2} we obtain that the ${\mathcal O}(\sigma)$
  correction, denoted by $\overline{U}_1$, to the average MFPT is
\begin{equation}\label{u1P_chi}
  \overline{U}_1 \equiv -u_{1pc} - \frac{\Big{(} r_c^{\mc{N}}
    \cos(\mc{N} \psi)\Big{)}}{\mc{N} D} =
   -\frac{\cos(\mc{N}\psi)}{D} \left(
     \frac{2r_c^{\mc{N}}}{\mc{N}} - \frac{r_c^{\mc{N}}}{2}(k-1) 
  + \frac{r_c^{\mc{N} + 2m}}{1 - r_c^{2m}}\right) \,.
\end{equation}
Finally, by combining the terms in \eqref{u1P_chi} we obtain the
main result given in Proposition~\ref{u_bar_prop_main} of
\S~\ref{sec:asymp_perturbed_unit_disk}.

\subsubsection{Number of folds is not an integer multiple of the
  number of traps: \texorpdfstring{$(\mc{N}\neq k m)$}
  {\mc{N}not=k m}}

When ${\mc{N}/m}\notin \mathbb{Z}^{+}$, we will first
  establish that $\v{e}^T\v{u}_{1p} = 0$. To show this, we define
  $z\equiv e^{2\pi i \mc{N}/m}$, where $i=\sqrt{-1}$, and calculate
  that
\begin{equation*}
  \sum_{j=0}^{m-1} \cos\left( \mc{N} \psi + \frac{2\pi j\mc{N}}{m} \right)
  = \mbox{Re} \left( e^{i \mc{N}\psi} \sum_{j=0}^{m-1} z^j \right) =
  \mbox{Re} \left( e^{i \mc{N}\psi} \frac{(1-z^m)}{1-z} \right) =0 \,,
\end{equation*}
since $z^m=1$ but $z\neq 1$, owing to the fact that
${\mc{N}/m}\neq \mathbb{Z}^{+}$. As a result, by summing the terms in
\eqref{u1P_Sol_2} over $j$, we obtain that $\v{e}^T\v{u}_{1p} = 0$.
We conclude that $\v{u}_{1p}\in {\mathcal Q}$, where
${\mathcal Q} \equiv \lbrace{ \v{q}\in \mathbb{R}^{m-1} \,\, \vert
  \,\, \v{q}^T \v{e}=0\rbrace}$. Consequently, from \eqref{u1_bar}, the
problem for $\v{B}$ and $\overline{u}_1$ reduces to
\begin{equation}\label{s:u1_bar}
\Big{[} I + 2 \pi \nu (I - E)\mc{G}  \Big{]} \v{B} = \nu \v{u}_{1p}\,,
\quad \text{and}  \quad \overline{u}_1 =  \frac{2 \pi}{m}
\v{e}^T \mc{G} \v{B}\,.
\end{equation}

Next, since the Neumann Green's matrix $\mc{G}$ is cyclic and
symmetric, its matrix spectrum is given by
\begin{equation}
  \mc{G}\v{e}=\kappa_1\v{e} \,; \qquad   \mc{G}\v{q}_j=\kappa_j\v{q}_j \,,
  \quad j=2,\ldots,m \,,
\end{equation}
where $\v{q}_j^T\v{q}_i=0$ for $i\neq j$ and $\v{e}^T\v{q}_j=0$ for
$j=2, \ldots,m$. Therefore, the set
$\lbrace{\v{q}_2,\ldots,\v{q}_m\rbrace}$ forms an
orthogonal basis for the subspace ${\mathcal Q}$. As such, since
$\v{u}_{1p}\in {\mathcal Q}$, we have
$\v{u}_{1p}=\sum_{j=2}^{m} d_j \v{q}_j$, for some coefficients $d_j$,
for $j=2,\ldots,m$, and we can seek a solution for
$\v{B}$ in \eqref{s:u1_bar} in the form
$\v{B}=\sum_{j=2}^{m} b_j \v{q}_j$ for some $b_j$, $j=2,\ldots,m$. Since
$E\v{q}_j=0$, we readily calculate that
\begin{equation}\label{sp:bsolve}
  \v{B} = \nu \sum_{j=2}^{m} \frac{d_j}{1+2\pi \nu \kappa_j} \v{q}_j \,,
  \qquad \mbox{where} \qquad
  d_j = \frac{\v{q_j}^T \v{u}_{1p}}{\v{q_j}^T\v{q}_j} \,.
\end{equation}
Then, since $\mc{G}\v{B} \in {\mathcal Q}$ and $\v{e}^T\v{q}=0$ for
$\v{q}\in {\mathcal Q}$, it follows that $\v{e}^T\mc{G}\v{B}=0$ so
that $\overline{u}_1=0$ in \eqref{s:u1_bar}. Finally, in view of
\eqref{AveMFPT_Perturb2}, we conclude that the correction of order
${\mathcal O}(\sigma)$ in the average MFPT vanishes. This establishes
the result given in Proposition~\ref{u_bar_prop_main} of
\S~\ref{sec:asymp_perturbed_unit_disk} when ${\mc{N}/m}\notin \mathbb{Z}^{+}$.

\subsection{Approximations for optimal trap configurations
  in a thin ellipse}\label{supp:anal_2}

We provide some details for the two different approximation schemes
outlined in \S~\ref{sec:skinnyellipse} for estimating the optimal
average MFPT for an elliptical domain of high-eccentricity that
contains three traps centered along the semi-major axis.

\subsubsection{Equivalent thin rectangular domains: Three traps}\label{thin:rec}

We extend the calculation of \S~\ref{sec:thin} to the case of three
circular absorbing traps of a common radius $\varepsilon$, where one
of the traps is located at the center of the ellipse, while the other
two traps are centered on the major axis symmetric about the origin.

We follow a similar approach as for the two traps case in
\S~\ref{sec:thin}, where we replace the ellipse with a thin
rectangular region, chosen so that the area of the region and that of
the traps is preserved. The corresponding MFPT problem on the
rectangle is to solve \eqref{SkinnyEllp_2Traps_Prob} with the
additional requirement that $u=0$ for $x=\pm \varepsilon_0$ on
$|y|\leq b$. Upon calculating the 1-D solution $u(x)$ to this MFPT
problem, we then integrate it over the rectangle to determine the
average MFPT $\overline{u}$ as
\begin{equation}\label{SkinnyEll_3Traps_AveMFPT}
  \overline{u} = C \left( - \frac{1}{4} \, x_{0}^{3}  + \frac{1}{2} \,
  {\left(2 \, a_{0} -3 \, \varepsilon_{0}\right)} x_{0}^{2}   -
  {\left(a_{0}^{2} - 2 \, a_{0} \varepsilon_{0}\right)} x_{0} + \frac{1}{3}
  \, a_{0}^{3} - a_{0}^{2} \varepsilon_{0} + a_{0}
\varepsilon_{0}^{2} - \varepsilon_{0}^{3} \right) \,,
\end{equation}
where $C={4 \,b_0/\left[\pi\,D\,(1-3 \varepsilon^2)\right]}$ and $x_0$ is the
$x$-coordinate of the right-most trap.
    
To determine the optimal average MFPT as $x_0$ is varied, we set
${d\overline{u}/dx_0}=0$ in \eqref{SkinnyEll_3Traps_AveMFPT}. The
critical point that minimizes the average MFPT is
\begin{equation}\label{SkinnyEll_3Traps_Opt_x0}
x_{0\, {\textrm{opt}} }= \frac{2a_0}{3} = \frac{\pi}{6\,b_0} \,,
\end{equation}
where we used $a_0={\pi/(4b_0)}$ from
\eqref{SkinnyEllp_Cond_for_Area}. This gives the optimal trap
locations as $(\pm 2a_0/3, 0)$.  As compared to the result in
\S~\ref{sec:thin} for two traps, the optimal traps have moved closer
to the reflecting boundaries at $x=\pm a_0$. Upon substituting
\eqref{SkinnyEll_3Traps_Opt_x0} into \eqref{SkinnyEll_3Traps_AveMFPT},
and writing $a_0$ and $\varepsilon_0$ in terms of the width of the
rectangular region $b_0$ using the equal area condition
\eqref{SkinnyEllp_Cond_for_Area}, we obtain that the optimal average
MFPT for the rectangle is
\begin{equation}\label{SkinnyEllp_3Traps_OptAve}
  \overline{u}_{\textrm{opt}} =  \frac{\pi^2}{432\,D \,b_{0}^{2}}
  \Big{(} 1 - 6 \,
 \varepsilon^{2} + \mathcal{O}(\varepsilon^4) \Big{)}.
\end{equation}
This shows that $\overline{u}_{\textrm{opt}}=\mathcal{O}( b_0^{-2})$, and as
expected, the optimal average MFPT is smaller than that in
\eqref{SkinnyEllp_OptLoc} of \S~\ref{sec:thin} for the case of two traps.

To relate the optimal MFPT in the thin rectangular domain to that in the
thin elliptical domain, we proceed as in \S~\ref{sec:thin} for the
two-trap case. We first set $a=a_0$, so that the length of the
rectangular domain and the ellipse along the major axis are the
same. From \eqref{SkinnyEllp_Cond_for_Area},
we obtain $b_0 = (\pi b)/4$, where $b$ is the semi-minor axis of the
ellipse, and so \eqref{SkinnyEll_3Traps_Opt_x0} and
\eqref{SkinnyEllp_3Traps_OptAve} become
\begin{equation}\label{3SkinnyEllp_CaseI}
x_{0\, {\textrm{opt}} } =   \frac{2}{3b} \quad \mbox{and} \quad
  \overline{u}_{\textrm{opt}} \approx  \frac{1}{27\, D \,b^{2}}  \Big{(} 1 - 6 \,
  \varepsilon^{2} + \mathcal{O}(\varepsilon^4) \Big{)}\,; 
    \qquad   \mbox{Case I:} \,\,\, (a=a_0) \,.
\end{equation}
The second possibility is to choose $b = b_0$, so that the width of
the thin rectangle and ellipse are the same. From
\eqref{SkinnyEll_3Traps_Opt_x0} and \eqref{SkinnyEllp_3Traps_OptAve},
we get
\begin{equation}\label{3SkinnyEllp_CaseII}
  x_{0\, {\textrm{opt}} } =     \frac{\pi}{6b} \quad \mbox{and} \quad
  \overline{u}_{\textrm{opt}} \approx
  \frac{\pi^2}{432\, D \,b^{2}}  \Big{(} 1 - 6 \,
 \varepsilon^{2} + \mathcal{O}(\varepsilon^4) \Big{)}\,; \qquad
   \mbox{Case II:} \,\,\, (b=b_0) \,.
\end{equation}
Similarly to the two-trap case, the results in
\eqref{3SkinnyEllp_CaseI} and \eqref{3SkinnyEllp_CaseII} provide upper
and lower bounds, respectively, for the optimal locations of the trap and the optimal average MFPT in the
thin elliptical region.

\subsubsection{A perturbation approach for long thin domains}\label{supp:long_thin}

In the asymptotic limit of a long thin domain, we use a
perturbation approach on the MFPT PDE \eqref{eqn:three_traps_mfpt}
in \S~\ref{sec:long_thin} for $u(x,y)$ in order to derive the
limiting problem \eqref{sec:long_u0}.

We first introduce the stretched variables $x$ and $y$ by
$X = \delta x, Y = {y/\delta}$ and $d = {x_0/\delta}$, and we label
$U(X,Y)=u({X/\delta},Y\delta)$. Then the PDE in
\eqref{eqn:three_traps_mfpt} becomes
\begin{equation}
\delta^4 \partial_{XX} U + \partial_{YY} U = -\frac{\delta^2}{D} \,.
\label{eqn:three_traps_nondim_pde}
\end{equation}
For $\delta\ll 1$, this suggests an expansion of $u$ given by
\begin{equation}
U = \delta^{-2} U_0 + U_1 + \delta^2 U_2 + \ldots \,.
\label{eqn:three_traps_expansion}
\end{equation}
Upon substituting (\ref{eqn:three_traps_expansion}) into
(\ref{eqn:three_traps_nondim_pde}), and equating powers of $\delta$, we obtain
\begin{equation}
\begin{aligned}
{\mathcal O}(\delta^{-2})\,:&\quad\partial_{YY} U_0 = 0\,, \\
{\mathcal O}(1)\,: &\quad\partial_{YY} U_1 = 0\,, \\
{\mathcal O}(\delta^2)\,: &\quad\partial_{YY} U_2 = -\frac{1}{D} -
\partial_{XX} U_0 \,.
\end{aligned}
\label{eqn:three_traps_pde_different_order}
\end{equation}

On the boundary $y = \pm\delta F(\delta x)$, or equivalently
$Y = \pm F(X)$, the unit outward normal is
$\hat{\mathbf{n}} = {\mathbf{n}/|\mathbf{n}|}$, where
$\mathbf{n} \equiv (-\delta^2 F^{\prime}(X),\pm1)$. The condition for the
vanishing of the outward normal derivative in
\eqref{eqn:three_traps_mfpt} becomes
\begin{equation*}
  \partial_n u = \hat{\mathbf{n}} \cdot (\partial_x u, \partial_y u) =
  \frac{1}{|\mathbf{n}|}(-\delta^2F^{\prime}, \pm 1) \cdot
  (\delta\partial_X U, \delta^{-1}\partial_Y U) = 0\,, \,\,\, \mbox{on}
  \,\,\, Y = \pm F(X) \,.
\end{equation*}
This is equivalent to the condition that
\begin{equation}
  \partial_Y U = \pm \delta^4 F^{\prime}(X) \partial_X U\quad \mbox{on}
  \quad  Y = \pm F(X) \,. \label{eqn:three_traps_bc}
\end{equation}
Upon substituting (\ref{eqn:three_traps_expansion}) into
(\ref{eqn:three_traps_bc}) and equating powers of $\delta$ we obtain on
$Y=\pm F(X)$ that
\begin{equation}
\begin{aligned}
{\mathcal O}(\delta^{-2})\,: &\quad\partial_Y U_0 =0\,, \\
{\mathcal O}(1)\,; &\quad\partial_Y U_1 = 0\,, \\
{\mathcal O}(\delta^2)\,; &\quad\partial_Y U_2 = \pm F^{\prime}(X)
\partial_X U_0 \,.
\end{aligned}
\label{eqn:three_traps_bc_different_order}
\end{equation}

From (\ref{eqn:three_traps_pde_different_order}) and
(\ref{eqn:three_traps_bc_different_order}) we conclude that
$U_0 = U_0(X)$ and $U_1 = U_1(X)$.  Assuming that the trap radius
$\varepsilon$ is comparable to the domain width $\delta$ we will
approximate the zero Dirichlet boundary condition on the three traps
as zero point constraints for $U_0$ at $X=0,\pm d$.

A multi-point BVP for $U_0(X)$ is derived by imposing a solvability
condition on the ${\mathcal O}(\delta^2)$ problem for $U_2$ given by
\begin{equation}\label{long:u2}
  \partial_{YY} U_2 = -\frac{1}{D} - U_0^{\prime\prime}\,, \,\,\,
  \mbox{in}\,\,\, \Omega\setminus\Omega_a\,; \quad \partial_Y U_2 =
  \pm F^{\prime}(X) U_0^{\prime}\,, \,\,\, \mbox{on} \,\,\,
   Y = \pm F(X)\,, \,\, |X|<1 \,.
\end{equation}
To derive this solvability condition for \eqref{long:u2}, we multiply
the problem for $U_2$ by $U_0$ and integrate in $Y$ over
$-F(X)<Y<F(X)$. Upon using Lagrange's identity and the boundary
conditions in \eqref{long:u2} we get
\begin{equation}
\begin{aligned}
  \int_{-F(X)}^{F(X)} \left(U_0 \partial_{YY} U_2 - U_2 \partial_{YY} U_0\right)\,
  \text{d}Y &= \left[ U_0 \partial_Y U_2 - U_2 \partial_Y U_0  \right]
  \Big{\vert}_{-F(X)}^{F(X)}\,, \\
  \int_{-F(X)}^{F(X)} U_0 \left( -\frac{1}{D} - U_{0}^{\prime\prime} \right) \, \text{d}Y
  &= 2U_0 F^{\prime}(X) U_0^{\prime} \,, \\
  2F(X)U_0\left(-\frac{1}{D} - U_0^{\prime\prime}\right) &= 2U_0 F^{\prime}(X)
  U_0^{\prime}\,. \\
\end{aligned}
\end{equation}
Thus, $U_0(X)$ satisfies the ODE 
$\left[F(X)U_0^{\prime}\right]^{\prime}= -{F(X)/D}$ as given in
\eqref{sec:long_u0} of \S~\ref{sec:long_thin}.

\subsection{Asymptotic analysis of a fast rotating trap}\label{supp:anal_3}

We summarize the derivation of the result given in
\S~\ref{sec:fastrottrap} for the optimal radius of rotation of the
rotating trap problem of \S~\ref{sec:TwoTrapsDisk} in the limit of
fast rotation $\omega\gg 1$. In this limit, the asymptotic MFPT
$u(\rho)$ satisfies the multi-point BVP
(\ref{fast_mfpt}), which has the solution
\begin{align*}
  u &= \frac{1}{4}\Big{(}(r - \eta)^2 - \rho^2 \Big{)} +
 \frac{1}{4\log\left(\frac{\varepsilon}{r - \eta}\right)}
 \left[(\varepsilon^2 - (r - \eta)^2 ) \log\left(\frac{\rho}{r - \eta}\right)
  \right] \,, \,\,\, \varepsilon\le \rho \le r-\eta \,, \\
  u &= \frac{1}{4}((r + \eta)^2 - \rho^2) + \frac{1}{2}
   \log\left(\frac{\rho}{r + \eta}\right)\,, \,\,\, r+\eta \le \rho \le 1 \,.
\end{align*}    

To compute the average MFPT, denoted by $U(r)$, over the
  unit disk, we need to calculate
  $I=\int_{0}^{r-\eta} u\rho \, \text{d}\rho + \int_{r+\eta}^{1} u\rho \,
  \text{d}\rho$.  By doing so, we obtain that $U(r)$ is given in
  (\ref{Analy_Int}). To optimize the average MFPT with respect to the
  radius of rotation of the fast moving trap, we simply set
  $U^{\prime}(r)=0$.  This leads to the following transcendental
  equation for $r$ in terms of the radii $\eta$ and $\varepsilon$ of
  the two traps:
\begin{equation}\label{Deriv_AveMFPT}
  \mathcal{A}(r) + 4 \mathcal{B}(r)
  \log\left(\frac{\varepsilon}{r-\eta }\right)^{2} -
  4 \log\left(\frac{\varepsilon}{r-\eta }\right) \mathcal{C}(r) =0 \,.
\end{equation}
Here $\mathcal{A}(r) $, $\mathcal{B}(r) $, and $\mathcal{C}(r) $ are defined
by
\begin{align*}
  \mathcal{A}(r) & = \varepsilon^{4} \eta - 2 \, \varepsilon^{2} \eta^{3} +
                   \eta^{5} - 3 \,
\eta r^{4} + r^{5} - 2 \, \left(\varepsilon^{2} - \eta^{2}\right)
r^{3} + 2 \, \left(\varepsilon^{2} \eta + \eta^{3}\right) r^{2}  \\
& \qquad +\left(\varepsilon^{4} + 2 \, \varepsilon^{2} \eta^{2} - 3 \,
    \eta^{4}\right) r\,,\\ 
\mathcal{B}(r) & = 2 \, \eta^{5} - 6 \, \eta r^{4} - 2 \,
    \eta^{3} + 2 \, \left(2 \, \eta^{3} + \eta\right) r^{2} + 2 \, r^{3} -
\left(2 \, \eta^{2} + 1\right) r + \eta \,,  \\
\mathcal{C}(r) & = \varepsilon^{2} \eta^{3} - \eta^{5} + 3 \, \eta
r^{4} - r^{5} + \left(\varepsilon^{2} - 2 \, \eta^{2}\right)
r^{3} - \left(\varepsilon^{2} \eta + 2 \, \eta^{3}\right) r^{2} -
 \left(\varepsilon^{2} \eta^{2} - 3 \, \eta^{4}\right) r\,.
\end{align*}
To determine the optimal $r$ we need to numerically compute the
root of \eqref{Deriv_AveMFPT}. The results were shown in
Figure \ref{TwoTrap_Disk_Analysis}.

\end{appendix}

\end{document}